\documentclass[12pt,reqno,twoside]{amsart}
\usepackage{tikz}
\usepackage{amssymb}
\usepackage{amsmath}
\usepackage{amsthm}
\usepackage{amsfonts}
\usepackage{amsthm,amscd}
\usepackage{amsbsy}
\usepackage{bm}
\usepackage{url} 
\usepackage{setspace}
\usepackage{float}
\usepackage{psfrag}
\usepackage{cite}
\usepackage{tikz}
\usepackage{float}
\usepackage{epstopdf}
\usepackage{caption}

\AtBeginDocument{
  \def\MR#1{}
}

\usetikzlibrary{positioning}

\usepackage[colorlinks, linkcolor=blue]{hyperref}

\restylefloat{figure}

\newtheorem{Theorem}{Theorem}[section]
\newtheorem{Lemma}[Theorem]{Lemma}

\newtheorem{Proposition}[Theorem]{Proposition}
\newtheorem{Remark}[Theorem]{Remark}

\numberwithin{equation}{section}
\usepackage{setspace}
\usepackage{anysize}
\marginsize{0.7in}{0.7in}{0.43in}{0.43in}
\def\be{\begin{equation}}
	\def\ee{\end{equation}}
\def\ben{\begin{eqnarray}}
	\def\een{\end{eqnarray}}

\newcommand{\ncom}{\newcommand}
\ncom{\N}{\mathbb N}
\ncom{\C}{\mathbb{C}} 
\ncom{\T}{\mathcal{T}}
\ncom{\A}{\mathcal{A}}
\ncom{\K}{\mathcal{K}}
\ncom{\D}{\mathcal{D}}
\ncom{\Ab}{\mathbb{A}}
\ncom{\Cb}{\mathbb{C}}
\ncom{\Nb}{\mathbb{N}}
\ncom{\Rb}{\mathbb{R}}
\ncom{\B}{\mathcal{B}}
\ncom{\Oc}{\mathcal{O}}

\ncom{\Af}{\mathbf{A}}
\ncom{\Bf}{\mathbf{B}}
\ncom{\Hf}{\mathbf{H}}
\ncom{\Gf}{\mathbf{G}}
\ncom{\Uf}{\mathbf{U}}
\ncom{\Yf}{\mathbf{Y}}
\ncom{\Pf}{\mathbf{P}}
\ncom{\Ff}{\mathbf{K}}
\ncom{\Lf}{\mathbf{L}}
\ncom{\Xf}{\mathbf{X}}
\ncom{\Vf}{\mathbf{V}}
\ncom{\Sf}{\mathbf{S}}
\ncom{\Tf}{\mathbf{T}}
\ncom{\Wf}{\mathbf{W}}
\ncom{\Lc}{\mathcal{L}}

\ncom{\Ac}{\mathcal{A}}
\ncom{\Bc}{\mathcal{B}}
\ncom{\Pc}{\mathcal{P}}
\ncom{\Acw}{\mathcal{A}_\omega}
\ncom{\Dc}{\mathcal{D}}

\ncom{\wt}{\widetilde}
\ncom{\Yt}{\widetilde{\textbf{Y}}}
\ncom{\ut}{\widetilde{u}}
\ncom{\wh}{\widehat}

\ncom{\nuh}{\widehat{\nu}}

\ncom{\Aw}{\mathbf{A}_\omega}

\ncom{\LL}{L^2(0,\infty;L^2(\Omega))}
\ncom{\LH}{L^2(0,\infty;H^1(\Omega))}
\ncom{\LHH}{L^2(0,\infty;H^2(\Omega))}
\ncom{\HL}{H^1(0,\infty;L^2(\Omega))}
\ncom{\HH}{H^1(0,\infty;H^2(\Omega))}
\ncom{\LU}{L^2(0,\infty;U)}
\ncom{\Lt}{L^2(\Omega)}
\ncom{\Hio}{H^1_0(\Omega)}
\ncom{\Lh}{L^2(0,\infty;\Hf)}

\ncom{\Brm}{{\bf\rm{B}}}
\ncom{\Yrm}{{\bf\rm{Y}}}
\ncom{\Zrm}{{\bf\rm{Z}}}
\ncom{\Arm}{{\bf\rm{A}}}
\ncom{\Mrm}{{\bf\rm{M}}}
\ncom{\Nrm}{{\bf\rm{N}}}
\ncom{\Krm}{{\bf\rm{K}}}
\ncom{\Prm}{{\bf\rm{P}}}
\ncom{\Mcl}{\mathcal{M}_h}
\ncom{\Ach}{\mathcal{A}_{\omega_h}}
\ncom{\Bch}{\mathcal{B}_h}
\ncom{\Yc}{\mathcal{Y}}

\ncom{\ym}{\Yt_h(t)}
\ncom{\um}{\ut_h(t)}
\ncom{\uc}{\ut(t)}
\ncom{\yc}{\Yt(t)}
\ncom{\ap}{\Aw{_{,\Pf}}}
\ncom{\ahp}{\Af{_{{\omega_h},\Pf}}}
\ncom{\aph}{\Aw{_{,\Pf_h}}}
\ncom{\ahph}{\Af{_{\omega_h,\Pf_h}}}
\ncom{\ymm}{\Yt_h(t)}
\ncom{\umm}{\ut_h(t)}
\ncom{\ucc}{\ut(t)}
\ncom{\ycc}{\Yt(t)}

\ncom{\n}{\normalfont}

\newcommand{\vertiii}[1]{{\left\vert\kern-0.25ex\left\vert\kern-0.25ex\left\vert #1 
		\right\vert\kern-0.25ex\right\vert\kern-0.25ex\right\vert}}

\long\def\/*#1*/{}
\title[Stabilization of Burgers equation around non-constant steady state]{Feedback Stabilization and Finite Element Error Analysis of Viscous Burgers Equation around Non-Constant Steady State}
\date{\today}

\author{WASIM AKRAM}\address{Wasim Akram \newline\indent Department of Mathematics, Indian Institute of Technology Bombay, \newline \indent 
Powai, Mumbai - 400076, India, Email- {\normalfont{wakram@math.iitb.ac.in, wakram2k11@gmail.com}} }

\begin{document}

	\begin{abstract}
In this article, we explore the feedback stabilization of a viscous Burgers equation around a non-constant steady state using localized interior controls and then develop error estimates for the stabilized system using finite element method. The system is not only feedback stabilizable but exhibits an exponential decay $-\omega<0$ for any $\omega>0$. The derivation of a stabilizing control in feedback form is achieved by solving a suitable algebraic Riccati equation posed for the linearized system. In the second part of the article, we utilize a conforming finite element method to discretize the continuous system, resulting in a finite-dimensional discrete system. This approximated system is also proven to be feedback stabilizable (uniformly) with exponential decay $-\omega+\epsilon$ for any $\epsilon>0$. The feedback control for this discrete system is obtained by solving a discrete algebraic Riccati equation. To validate the effectiveness of our approach, we provide error estimates for both the stabilized solutions and the stabilizing feedback controls. Numerical implementations are carried out to support and validate our theoretical results.
\end{abstract}
	\maketitle
\pagenumbering{arabic}



\section{Introduction}
\noindent The Burgers equation serves as a compelling model in applied mathematics and finds applications in numerous physical and natural phenomena, including shock flows, turbulence, weather patterns, mass transport, and more. Furthermore, it plays a pivotal role as a simplified representation of more intricate models like the Navier-Stokes equations, traffic flow models, and others.

 
 \subsection{Model problem}\noindent Let $\Omega$ be an open bounded set in $\mathbb{R}^d$ for $d\in \{1,2,3\}.$ Set $Q=\Omega\times (0,\infty),$ $\Sigma=\partial\Omega\times(0,\infty)$ and let $\textbf{v}=(v_1,\ldots,v_d)^T\in \mathbb{R}^d$ be a given fixed vector. Consider
\begin{equation} \label{eqVB-mainburger}
	\begin{aligned}
		& y_t(x,t) + y(x,t)\textbf{v}.\nabla y(x,t) -\eta \Delta y(x,t) +\nu_0 y(x,t)= f_s+u(x,t)\chi_{\mathcal{O}}(x) \text{ in } Q, \\
		& y(x,t)=0 \text{ on }\Sigma, \, y(x,0)=y_0(x) \text{ in }\Omega,
	\end{aligned}
\end{equation}
with the interior control $u$ acting on the open subset $\mathcal{O}\subset \Omega,$ where $\eta>0$ is given viscosity coefficient, $\Delta$ denotes the Laplace operator in space variable $x,$ and $f_s\in L^2(\Omega)$ is a given stationary data. Let $y_s(x)\in H^2(\Omega)\cap \Hio$ be a given stationary solution to the equation
\begin{equation} \label{eqVB-steadystate}
	\begin{aligned}
		-\eta \Delta y_s(x)+y_s(x) \textbf{v}.\nabla y_s(x) +\nu_0 y_s(x) & = f_s \text{ in }\Omega, 
		& y_s =0 \text{ on }\partial\Omega.
	\end{aligned}
\end{equation}
Our aim is to discuss local stabilization of \eqref{eqVB-mainburger} around the stationary solution $y_s,$ with a prescribed exponential decay rate $-\omega,$ for any $\omega>0.$  In particular, the aim is to find a $u$ in feedback form such that $y-y_s$ is exponentially stable with decay rate $-\omega.$ It is convenient to write the non-linear system satisfied by $z=y-y_s:$ 
\begin{equation} \label{eqVB-linarnd-ys}
	\begin{aligned}
		& z_t +z\textbf{v}\cdot \nabla z -\eta \Delta z + y_s \textbf{v} \cdot \nabla z+\textbf{v}\cdot \nabla y_s z +\nu_0 z = u\chi_{\mathcal{O}} \text{ in } Q, \\
		& z=0 \text{ on }\Sigma, \, z(x,0)=y_0(x)-y_s(x)=:z_0(x) \text{ in }\Omega.
	\end{aligned}
\end{equation}
If we observe that \eqref{eqVB-linarnd-ys} exhibits exponential stability with decay rate $-\omega>0$ for some $\omega>0$, it follows that \eqref{eqVB-mainburger} is also exponentially stable around $y_s$ with the same decay rate. Our primary objective is to investigate the feedback stabilizability of the associated linear system of \eqref{eqVB-linarnd-ys} and subsequently analyze the feedback stabilizability of the non-linear system \eqref{eqVB-linarnd-ys}.

\medskip
\noindent The linear equation corresponds to \eqref{eqVB-linarnd-ys} is 
\begin{equation} \label{eqVB-linarnd-ys-ln}
	\begin{aligned}
		& z_t-\eta \Delta z + y_s \textbf{v} \cdot \nabla z+\textbf{v}\cdot \nabla y_s z +\nu_0 z= u\chi_{\mathcal{O}}\text{ in } Q, \\
		& z=0 \text{ on }\Sigma, \, z(x,0)=z_0(x) \text{ in }\Omega.
	\end{aligned}
\end{equation}


\noindent The unbounded operator $(\Ac,D(\Ac))$ on $\Lt$ associated to \eqref{eqVB-linarnd-ys} and \eqref{eqVB-linarnd-ys-ln} is defined by 
\begin{equation} \label{eqdefVB-A}
	\begin{aligned}
		\Ac z=\eta \Delta z-y_s \textbf{v}\cdot \nabla z-\textbf{v}\cdot \nabla y_s z -\nu_0 z \text{ with } D(\Ac)= H^2(\Omega)\cap H^1_0(\Omega),
	\end{aligned}
\end{equation}
and we define the control operator $\Bc\in \Lc(\Lt)$ by
\begin{equation} \label{eqVB-B}
	\Bc u=u\chi_{\mathcal{O}} \text{ for all } u\in \Lt.
\end{equation}  
\noindent Then the systems \eqref{eqVB-linarnd-ys} and \eqref{eqVB-linarnd-ys-ln} are written on the space $\Lt$ as 
\begin{equation} \label{eqVB-MainOptForm}
	\begin{aligned}
		z'(t)=\Ac z(t) + \kappa F(z(t)) +\Bc u(t)\, \text{ for all } t>0,\quad z(0)=z_0,
	\end{aligned}
\end{equation} 
with $\kappa=1$ for \eqref{eqVB-linarnd-ys} and $\kappa=0$ for \eqref{eqVB-linarnd-ys-ln}, $F(z)=-z\textbf{v}\cdot \nabla z.$

\noindent By $(\Ac^*,D(\Ac^*))$ and  $\Bc^*,$ we denote the adjoint operators corresponding to $(\Ac,D(\Ac))$ and $\Bc,$ respectively, on $\Lt$. 

\medskip 
\noindent To study the exponential stabilizability of \eqref{eqVB-MainOptForm} with decay $-\omega<0,$ for any given $\omega>0,$ set $\wt z(t):=e^{\omega t}z(t)$ and $\wt u(t):=e^{\omega t}u(t).$ Then $(\wt z(t), \wt u(t))$ satisfies 
\begin{equation} \label{eqVB-shiftedSysComb}
	\wt z'(t)=\Acw \wt z(t) +\kappa e^{-\omega t} F(\wt z(t)) +\Bc \wt u(t) \text{ for all }t>0, \wt z(0)=z_0,
\end{equation}
where
\begin{align}\label{eqVB-Acw-Acw*}
	\Acw=\Ac+\omega I \text{ with } D(\Acw)=D(\Ac),
\end{align}
and $I:\Lt\rightarrow \Lt$ is the identity operator. Assuming that the exponential stabilizability of \eqref{eqVB-shiftedSysComb} by a feedback control $\widetilde{u}(t) = K\widetilde{z}(t)$, we can establish that \eqref{eqVB-MainOptForm} is also exponentially stabilizable with decay rate $-\omega<0$ by employing the control $u(t) = e^{-\omega t}\widetilde{u}(t)$. Consequently, our focus can be narrowed down to investigating the stabilizability of \eqref{eqVB-shiftedSysComb} rather than \eqref{eqVB-MainOptForm}. To facilitate this analysis, we introduce the adjoint operator $\Acw^*$ corresponding to $\Acw$, defined as follows:
\begin{align}\label{eqVB-Acw}
\Acw^* := \Ac^* + \omega I \text{ with } D(\Acw^*) = D(\Ac^*).
\end{align}

\subsection{Literature survey}
 \noindent Feedback stabilization of evolution equations via Riccati is well-known and well-studied; for example, see \cite{WKRPCE,WKR,VBarBook, BDDM, Lasiecka1, Lasieckabook,Burns91,BreKun17} and references therein. Feedback stabilization by solving a finite-dimensional Riccati equation has been developed extensively; for example, see   \cite{WKR, BarLasTri06, BarTri04, JPR2010, Badra-T} and references therein.

\noindent The authors in \cite{BYR98,Burns91} discuss the following one-dimensional Burgers equation 
\begin{equation} \label{eqInt-1DBurgs}
	\begin{aligned}
		& y_t(x,t)-\eta y_{xx}(x,t) +y(x,t)y_x(x,t)=0 \text{ in } (0,l),\\
		& y(0,t)=y(l,t)= 0 \text{ for all }t>0, \, y(x,0)=y_0(x) \text{ for all } x\in (0,l),
	\end{aligned}	
\end{equation}
with $\eta>0$ as the viscosity constant and establish local stabilization results by using (finite-dimensional) interior control. Under a smallness assumption on initial condition $y_0,$ the stabilizability of the non-linear equation is obtained.
Here, the feedback control is obtained by solving the Riccati equation posed for a linearized system, and then a local stabilization is shown for the non-linear system. In \cite{LyMT97}, the smallness assumption on the initial condition (in \cite{Burns91}) is relaxed and the stabilization results in one dimension are established. The Neumann boundary control problem is also analyzed in \cite{LyMT97}. The paper \cite{Thev-JPR} considers a two dimensional Burgers equation
 \begin{equation*}
 	\begin{aligned}
 		& y_t - \eta \Delta y + y y_{x_1}+yy_{x_2}=f_s \text{ in }\Omega \times (0,\infty), \\
 		& \eta \frac{\partial y}{\partial n} = g+m u \text{ on }\partial \Omega \times (0,\infty), \quad y(0)=y_0 \text{ in }\Omega,
 	\end{aligned}
 \end{equation*}
 where $m$ is a function that localizes the control in a part of boundary $\partial\Omega$ and establish local stabilization results around given stationary solution $w$ that satisfies
  \begin{align*}
  	-\eta \Delta w +ww_{x_1}+ww_{x_2}=f_s \text{ in }\Omega, \quad \frac{\partial w}{\partial n}=g \text{ on }\partial\Omega.
  \end{align*}
 The Hamilton-Jacobi-Bellman (HJB) equation given by
    \begin{align*}
    	\langle Ay+F(y), \mathcal{G}(y)\rangle - \frac{1}{2}\|B^*\mathcal{G}(y)\|_U^2+\frac{1}{2}\|y\|^2=0 \text{ for all } y\in D(A),
    \end{align*}
where $(A, B)$ are associated linear operators and $F(y)$ denotes the non-linear term, is solved for $\mathcal{G}$ using Taylor series expansion to obtain a non-linear feedback control
\begin{align*}
	u=-B^*Py+B^*(A-BB^*P)^{-*}PF(y).
\end{align*}
Here, $P$ is solution of an algebraic Riccati equation and $(A-BB^*P)^{-*}$ is the inverse of $(A-BB^*P)^*.$ Local stabilization results for two dimensional Burgers equation by using linear feedback boundary control is investigated in \cite{JPR1}.

\medskip
\noindent In 1999, Krstic \cite{Krstic} studies the global stabilizability of a one-dimensional viscous Burgers equation around a constant steady state with Neumann and Dirichlet boundary control by using the Lyapunov technique. In 2000, by introducing a cubic Neumann boundary feedback control, Balogh and Krstic \cite{KrsBal00} prove global asymptotic stability. In \cite{Kundu18, Kundu19, Kundu20}, global stabilization results by using Lyapunov functional are established for Burgers equation, and BBM Burgers equation
\begin{align*}
	& y_t - \eta_1 y_{xxt} - \eta_2 y_{xx}+y_x+yy_x=0 \text{ for all } (x,t)\in (0,1)\times (0,\infty), \\
	& y_x(0,t)=u_0(t) \text{  for all } t\in (0,\infty), \quad y_x(1,t)=u_1(t) \text{  for all } t\in (0,\infty),\\
	& y(x,0)=y_0(x) \text{ for all } x\in \Omega,
\end{align*}
 where $\eta_1, \eta_2>0$ are given constants, $u_0(t)$ and $u_1(t)$ are controls. In these articles, the authors also study error estimates for the stabilized system and verify their results by providing numerical examples. In 1999, S. Volkwein \cite{Volk01} studies the optimal control problem for viscous Burgers equation with distributed control.

\noindent  This Riccati-based technique is used extensively to study the stabilization of parabolic equations, such as incompressible Navier-Stokes equations \cite{JPR, JPR2010}
\begin{align*}
	& y_t- \eta \Delta y +(y\cdot \nabla )y +\nabla q =f , \quad  \text{div }y=0\text{ in } \Omega\times (0,\infty), \\
	& y=g \text{ on }\partial \Omega\times (0,\infty), \quad y(0)=y_0 \text{ in } \omega,
\end{align*}
around solution to stationary Navier-Stokes equation
\begin{align*}
	-\eta \Delta w + (w\cdot \nabla)w +\nabla p =f , \quad \text{div }w=0 \text{ in }\Omega, w= g \text{ on }\partial\Omega,
\end{align*}
 coupled parabolic-ODE systems \cite{WKR}, and other few models in \cite{BreKun17,Badra-ths}. In \cite{BreKun14, BreKun17}, the Fitzhugh-Nagumo model, a nonlinear reaction-diffusion coupled system, has been considered. A characterization of the stabilization of parabolic systems is obtained in \cite{BadTak14}.
 
 \medskip
\noindent The exponential decay may be determined by the spectrum of the linear operator associated with the system, provided the operator satisfies some suitable assumptions. The fact that the exponential decay for the feedback stabilizability of a linear system associated with a linear operator $A$ and a control operator $B$ is bounded by the accumulation point of the spectrum of $A$ provided $A$ and $B$ satisfy some suitable assumptions, has been studied in the literature (for example see \cite[Proposition 5.1 and Corollary 5.1]{Trig75}). The open loop stabilizability for a linear system can be obtained by checking the Hautus condition under some suitable conditions, and then a feedback control can be constructed by solving an optimal control problem in an infinite time horizon. The feedback stabilization using an appropriate Riccati equation posed on an infinite dimensional Hilbert space has been obtained in \cite{BreKun14}.

\medskip
\noindent Due to the wide applications of Burgers equations, various numerical methods are proposed to solve it. The finite element method is used extensively to solve Burgers equations, and we refer to \cite{Cald81, PanyNN07, AKPNN08, Dogan04, YChen19} and references therein. Error estimates for one dimensional Burgers equation are established and verified by a numerical experiment in \cite{PanyNN07}. The paper \cite{AKPNN08} obtains an optimal error estimate for one dimensional Burgers equation with the numerical investigation. Using a weak Galerkin finite element method, \cite{YChen19} establishes optimal order error estimates for one dimensional Burgers equation. Semi-discrete and fully discrete systems are considered and verified here with numerical implementation.`

\medskip
\noindent In \cite{Lasiecka1}, the numerical theory as the counterpart of the known continuous theory for feedback stabilization has been developed for abstract parabolic systems of the form \eqref{eqVB-MainOptForm} with $\kappa=0$ under certain hypotheses. This book provides a numerical approximation theory of continuous dynamics and algebraic Riccati equations. In finite-dimensional space $H_h\subset H$ with $h$ as discretization parameter, a family of approximate system  
\begin{align}\label{eqInt:LinConDis-1}
	y_h'(t)=A_h y_h(t)+B_h u_h(t) \text{ for all }t>0, \quad y_h(0)=y_{0,h},
\end{align} 
of \eqref{eqVB-MainOptForm} with $\kappa=0$ is constructed, where $A_h$ and $B_h$ denote the approximated operators corresponding to $A$ and $B,$ respectively, and $y_{0,h}\in H_h$ is an approximation of $y_0.$ Under the assumptions that $(A,D(A))$ generates analytic semigroup on $H$, $(A,B)$ is stabilizable, and $\|A^{-1} - A_h^{-1}\pi_h\|_{\Lc(H)} \le Ch^s,$ where $\pi_h:H\rightarrow H_h$ is projection, it is established that the discrete system \eqref{eqInt:LinConDis-1} is uniformly stabilizable. A feedback stabilizing control of the form $u_h(t)=-B_hB_h^*P_h y_h(t),$ where $P_h \in \Lc(H_h)$ is solution of the discrete algebraic Riccati equation 
\begin{align}
	A_h^*P_h+P_hA_h-P_hB_hB_h^*P_h+I_h=0, \quad P_h=P_h^*\ge 0 \text{ on } H_h,
\end{align}
is obtained. The error estimates for the trajectories and feedback controls have been obtained with the `optimal rate' of convergence in this setup. The application of this theory and related works can be found in \cite{Las-Tr-91, Ls-Tr-87-I,Ls-Tr-87-II} and references therein. 

\medskip
\noindent
In \cite{KK91}, the authors consider linear quadratic control problems for parabolic equations with variable coefficients. They provide the approximation of the Riccati equation and obtain the convergence rate for the optimal controls and trajectories.  

\medskip
\noindent Burns and Kang \cite{Burns91} compute a feedback control using the linearized system to show the stabilizability of one dimensional Burgers equation \eqref{eqInt-1DBurgs}. Further, they study the convergence analysis for the linear system and implement some numerical examples. Using the Lyapunov technique, the stabilizability of one dimensional Burgers equation by non-linear feedback boundary (Neumann) control is established in  \cite{KrsBal00}. They implement an example where the uncontrolled solution converges to non-zero equilibrium, and after applying the feedback Neumann boundary control, the solution converges to zero. In  \cite{Thev-JPR}, stabilization results on the two dimensional viscous Burgers equation around a non-constant steady state by boundary (Dirichlet and Neumann) control are established, and numerical implementations validate theoretical results. Using a finite element method, \cite{JPR1} computes feedback control for a Burgers equation (in two dimension) by solving an algebraic Riccati equation. They numerically show that the non-linear system is stabilizable by the computed feedback control. Global stabilization results by using Lyapunov functional are established for Burgers equation, and BBM Burgers equation with Neumann boundary control in \cite{Kundu18, Kundu19, Kundu20}. In these articles, the error estimates for the stabilized system are established and verified by providing some numerical examples.

\subsection{Contributions}

\noindent First, we show that the linear system \eqref{eqVB-linarnd-ys-ln} is feedback stabilizable with exponential decay $-\omega<0$, for any $\omega>0$, and study associated numerical analysis. This stabilization process is centered around a non-constant steady state, resulting in the associated linear unbounded operator $\Ac$ having variable coefficients. We show that the operator $(\Ac, D(\Ac))$ generates an analytic semigroup $\{e^{t\Ac}\}_{t\ge 0}$ by showing that the spectrum of $\Ac$  is contained in a sector $\Sigma(-\nuh;\theta_0) := \{-\nuh + re^{\pm i\theta}\, \mid \, r > 0 ; \theta \in (-\pi,\pi]; |\theta|\ge \theta_0\}$ for some $-\nuh>0$ and $\frac{\pi}{2}<\theta_0<\pi$ and outside this sector a resolvent estimate holds. Due to the variable coefficients in $\Ac$, exact expressions of its eigenvalues and corresponding eigenfunctions are unavailable. Therefore, we cannot proceed to verify the Hautus condition analogously as in \cite{WKRPCE, WKR} to study the stabilizability of the linear system. 
To address this challenge, we establish the compactness of the resolvent operator $R(-\nuh, \Ac)$ for some $\nuh>0$ and ascertain that the spectrum of $\Ac$ is discrete. Next, using a unique continuation principle, we verify the Hautus condition and establish that the linear system is open-loop stabilizable (see Proposition \ref{ppsVB-stabOpenLoop}).Subsequently, employing standard techniques involving the construction of an optimal control problem, we derive a feedback control that stabilizes the linear system \eqref{eqVB-linarnd-ys-ln} with exponential decay $-\omega<0$ for any $\omega>0$ (refer to Theorem \ref{thVB-mainstab}). We investigate the stabilizability of the nonlinear system by utilizing the control obtained for the linear system, as elucidated in Theorem \ref{thVBMain-NonLin}, through the application of the Banach fixed-point theorem.

\medskip
\noindent Next, we study the associated numerical analysis for both the linear and non-linear system. Employing the finite element method, we derive an approximation of the system within a discretized space. It is shown that the approximated operator $\Ac_h$ corresponding to $\Ac$ generates an uniformly analytic semigroup on the discretized space $V_h.$ Furthermore, we demonstrate the uniform stabilizability of the approximated linear system, along with deriving error estimates for the linear stabilized system and its stabilizing control. These results are achieved by establishing crucial properties (denoted as (a) - (d)) outlined in Property $(\mathcal{A}_1)$ on Page 13. These properties closely mirror those discussed in Sections 4.3 - 4.4 of \cite{WKRPCE}. Utilizing these crucial results, an analogous analysis done in Sections 4.5 - 4.7 (of \cite{WKRPCE}) leads to the desired results on the linear system and hence is skipped for this part. The main results on the linearized system (Theorems \ref{thVB:dro} - \ref{thVB:main-conv-new}) are analogous to the results obtained in the case of the parabolic coupled system studied in \cite{WKRPCE}.
Here, we obtain a quadratic rate of convergence for the stabilized solution and stabilizing control. To validate these results, we conduct illustrative examples implemented for the linear system.

\medskip
\noindent We extend the numerical analysis done for the linear system to the nonlinear system. The feedback discrete control obtained for the linear approximated system is used to obtain a nonlinear discrete closed-loop system and prove that it is stable by utilizing some energy argument (see Theorem \ref{thVB-DisMainStabNL}). The main difficulty is to derive an error estimate for the nonlinear stabilized system, and it doesn't follow immediately from any available result. However, using a certain elliptic projection and in a better regular space, using energy argument, we establish an error estimate for the nonlinear closed-loop system (see Theorem \ref{thVB-mainErrNL}). Here, we obtain a quadratic rate of convergence for the stabilized solutions and stabilizing control. We implement an example to show that the nonlinear system is stabilizable and compute convergence orders to verify theoretical results. To implement, we explicitly provide the algorithm used to show the stabilizability and to obtain the rate of convergence.

\medskip
\noindent The main novelty of this work is that we establish an error estimate for the stabilization problem of a nonlinear parabolic system where the feedback stabilization (around a constant or non-constant steady state) is obtained by solving an algebraic Riccati equation. This is new in the literature, and the technique used here can be useful for many important nonlinear parabolic systems. In \cite{WKRPCE}, we investigated the stabilization of the viscous Burgers equation with memory around a zero steady state. Given the significance of understanding stabilizability around a non-constant steady state and deriving an error estimate, it is meaningful to explore the stabilizability of \eqref{eqVB-mainburger} around such a state. Furthermore, developing an error estimate for the stabilizing system along with an algorithm for implementation is crucial.

\subsection{Organization} The remaining part of this article is organized as follows. A few results and notations that are repeatedly used throughout this article are listed in Section \ref{sec:preli}. All the results on linear system are established in Section \ref{secVB-LinSys}. In Subsections \ref{subsecVB-ASGP-SA} and \ref{subsecVB-StabCont}, the analytic semigroup with spectral analysis of $\Ac$ and stabilizability are discussed, respectively. The uniform stabilizability, error estimates for stabilized solutions, and stabilizing control are developed in Subsection \ref{secVB-FEM}, and implementation on the linear system is presented in Subsection \ref{subVB-NI-Lin}. Section \ref{secVB-NLS} focuses on the nonlinear system, where stabilizability of the nonlinear system is discussed in Subsection \ref{secVB-stabNL}, and error estimates are analyzed in Subsection \ref{secVB-ErrEst-NL}. Finally, an algorithm on implementation of the nonlinear system and a numerical example are discussed in Section \ref{secVB-impl-NL}.




\section{Preliminaries} \label{sec:preli}
\noindent In this section, we mention some preliminary results and define some notations we frequently use in this article.  

\noindent \textbf{Young's inequality.} Let $a,b,$ be any non-negative real numbers. Then for any $\epsilon>0,$ the following inequality holds:
\begin{align} \label{eqPR-YoungIneq}
	ab \le \frac{\epsilon a^2}{2} + \frac{b^2}{2\epsilon}.
\end{align}

\begin{Proposition}[Generalized H\"older's inequality] \label{ppsPR-GenHoldIneq}
	Let $f\in L^p(\Omega),$ $g\in L^q(\Omega),$ and $h\in L^r(\Omega),$ where $1\le p, q ,r<\infty$ are such that $\frac{1}{p}+\frac{1}{q}+\frac{1}{r}=1.$ Then $fgh\in L^1(\Omega)$ and 
	\begin{align*}
		\|fgh\|_{L^1(\Omega)} \le \|f\|_{L^p(\Omega)} \|g\|_{L^q(\Omega)} \|h\|_{L^r(\Omega)}.
	\end{align*} 
\end{Proposition}

\noindent In the following lemma, we recall Sobolev embedding \cite[Theorem 2.4.4]{Kes} in our context.
\begin{Lemma}[Sobolev embedding] \label{lemPR:SobEmb}
	Let $\Omega$ be an open bounded domain in $\mathbb{R}^d$ of class $C^1$ with smooth boundary with $d\in \Nb.$ Then, we have the following continuous inclusion with constant $s_0:$
	\begin{itemize}
		\item[$(a)$] $H^1(\Omega)\hookrightarrow L^p(\Omega)$ for $d>2,$ where $p=\frac{2d}{d-2},$
		\item[$(b)$] $H^1(\Omega) \hookrightarrow L^q(\Omega)$ for all $d=2\le q<\infty,$ and
		 \item[$(c)$] $H^1(\Omega) \hookrightarrow L^\infty(\Omega)$ for $d=1.$
	\end{itemize}
\end{Lemma}
\noindent In particular, the following Sobolev embedding that follows from the above lemma is being used frequently in this thesis:
\begin{align} \label{eqPR:SobEmb}
	H^1(\Omega) \hookrightarrow L^4(\Omega), \text{ with } s_0 \text{ as embedding constant, }
\end{align}
where $\Omega\subset \mathbb{R}^d$ bounded domain for $d\in \{1,2,3\}.$

\begin{Lemma}[Agmon's inequality {\n\cite{Agmon10}}] \label{lemVB:AgmonIE}
	Let $\Omega$ be an open bounded domain $\mathbb{R}^d,\, d\in \{1,2,3\},$  and let $z \in H^2(\Omega)\cap \Hio.$ Then there exists a positive constant $C_a=C_a(\Omega)$ such that
	\begin{equation} \label{eqVB-AgmonIE}
		\|z\|_{L^\infty(\Omega)} \le C_a \|z\|_{H^1(\Omega)}^{1/2} \|z\|_{H^2(\Omega)}^{1/2}.
	\end{equation}
\end{Lemma}

\noindent \textbf{Notations.} Throughout this article, we denote the inner product and norm in $\Lt$ by $\langle \phi \cdot \psi \rangle :=\int_\Omega \phi\overline{\psi} dx $ and $\|\phi\|:= \left(\int_\Omega |\phi|^2 \, dx\right)^{1/2},$ respectively. For any integer $1\le m< \infty,$ $H^m(\Omega)$ denotes the Sobolev space
\begin{align*}
	H^m(\Omega):= \left\lbrace f\in L^2(\Omega) \, |\, D^\alpha f \in L^2(\Omega), \, |\alpha|\le m\right\rbrace,
\end{align*}
with the norm $\|f\|_{H^m(\Omega)}:=\displaystyle \left(\sum_{|\alpha|\le m} \|D^\alpha f\|_{L^2(\Omega)}^2 \right)^{1/2},$
where $\alpha$ is multi-index and the derivatives are in the sense of distributions.  For $a\in \mathbb{R}$ and $\theta_0 \in  (\frac{\pi}{2}, \pi),$ we denote
$\Sigma(a; \theta_0)$ as the sector $\Sigma(a; \theta_0) := \{a + re^{\pm i\theta}\, | \, r > 0 ; \theta \in (-\pi,\pi]; |\theta|\ge \theta_0\}$ in the complex
plane and $\Sigma^c(a; \theta_0)$ denotes its complement in $\mathbb{C}.$ Depending on the context, the absolute value of a real number or the modulus of a complex number are denoted by $|\cdot|.$ The positive constant $C$ is generic and
independent of the discretization parameter $h.$

\section{Linear system} \label{secVB-LinSys}
\noindent This section focuses on the stabilizability and associated numerical analysis of the linear system  
\begin{align} \label{eqVB-shifsysLin}
	\wt z'(t)=\Acw\wt z(t)+\Bc \wt u(t) \text{ for all }t>0, \, \wt z(0)=z_0.
\end{align}



\subsection{Analytic semigroup and spectral analysis} \label{subsecVB-ASGP-SA}
The aim of this subsection is to study the well-posedness of \eqref{eqVB-MainOptForm} with $\kappa=0$ by showing $(\Ac,D(\Ac))$ generates an analytic semigroup on $\Lt$. Also, we discuss some spectral properties of the operator $\Ac$ and hence of $\Acw$. The weak formulation corresponding to \eqref{eqVB-linarnd-ys-ln} seeks $z(\cdot)\in \Hio$ such that 
\begin{align*}
	& \left\langle\frac{d}{dt}z(t),\phi\right\rangle +a(z(t),\phi)  =\langle u(t)\chi_{\mathcal{O}},\phi\rangle  \, \text{ for all }t>0,\\
	& \langle z(0),\phi\rangle = \langle z_0 ,\phi\rangle,
\end{align*}
for all $\phi\in \Hio,$ where the sesquilinear form $a(\cdot,\cdot):\Hio\times\Hio \longrightarrow \Cb$ is defined by 
\begin{align} \label{eqVB-Sesq}
	a(z,\phi)=\eta \langle  \nabla z,\nabla \phi\rangle+\langle y_s\textbf{v}\cdot \nabla z, \phi\rangle+\langle \textbf{v}\cdot \nabla y_sz, \phi\rangle+\nu_0\langle z,\phi\rangle \text{ for all }z,\,\phi \in \Hio.
\end{align}
Utilizing Proposition \ref{ppsPR-GenHoldIneq} with $p=4,q=4,r=2,$ Agmon's inequality \eqref{eqVB-AgmonIE}, and Lemma \ref{lemPR:SobEmb}, the sesquilinear form defined in \eqref{eqVB-Sesq} satisfies
\begin{equation} \label{eqVB-BddBil}
	\begin{aligned}
		|a(z,\phi)| & \le \eta \|\nabla z\|\|\nabla\phi\|+|\textbf{v}|\|y_s\|_{L^\infty(\Omega)} \|\nabla z\|\,\|\phi\| + |\textbf{v}|\|\nabla y_s\|_{L^4(\Omega)} \| z\|_{L^4(\Omega)}\|\phi\| +|\nu_0|\| z\|\,\|\phi\| \\
		& \le \alpha_1 \|\nabla z\|\,\|\nabla\phi\| \text{ for all }z,\, \phi\in \Hio,
	\end{aligned}
\end{equation}
for some $\alpha_1=\alpha_1(\textbf{v},s_0, y_s,C_p)>0.$  Therefore, $a(\cdot,\cdot)$ defined in \eqref{eqVB-Sesq} is bounded.

\noindent We assume that the given vector $\textbf{v}\in \Rb^d,$ $y_s\in H^2(\Omega)\cap \Hio,$ and $\nu_0$ are such that 
	\begin{align} \label{eqdefVB-nu0}
		\nuh:=\nu_0 - \frac{|\textbf{v}|^2}{\eta}(C_a^2+s_0^4) \| y_s\|_{H^2(\Omega)}^2>0,
	\end{align}
where $s_0$ and $C_a$ are as in Lemma \ref{lemPR:SobEmb} and \eqref{eqVB-AgmonIE}, respectively.

\noindent Now, choose $z=\phi\in \Hio$ in \eqref{eqVB-Sesq} to obtain
\begin{align*}
	\eta \|\nabla \phi\|^2 +\nu_0\|\phi\|^2 & = a(\phi,\phi)-\langle y_s\textbf{v}\cdot \nabla \phi, \phi\rangle - \langle  \textbf{v} \cdot\nabla y_s\phi,\phi\rangle.
\end{align*}
Using Proposition \ref{ppsPR-GenHoldIneq} with $p=4,q=4,r=2,$  \eqref{eqVB-AgmonIE}, Lemma \ref{lemPR:SobEmb}, and \eqref{eqPR-YoungIneq} in the above equality, we have 
\begin{equation} \label{eqVB-est term ys}
\begin{aligned}
	\eta \|\nabla \phi\|^2 +\nu_0\|\phi\|^2 &  \le \Re\left( a(\phi,\phi)\right) +|\textbf{v}|\|y_s\|_{L^\infty(\Omega)} \|\nabla \phi\|\,\|\phi\| + |\textbf{v}|\|\nabla y_s\|_{L^4(\Omega)} \| \phi\|_{L^4(\Omega)}\|\phi\| \\
	& \le \Re\left( a(\phi,\phi)\right) + \frac{\eta}{2}\|\nabla \phi\|^2 + \frac{|\textbf{v}|^2}{\eta} (C_a^2+s_0^4)\|y_s\|_{H^2(\Omega)}^2\|\phi\|^2.
\end{aligned}
\end{equation}
The choice of $\nuh$ as in \eqref{eqdefVB-nu0}, respectively lead to
\begin{align} \label{eqVB-coerc}
	\Re\left( a(\phi,\phi)\right) - \nuh \|\phi\|^2 \ge \frac{\eta}{2}\|\nabla \phi\|^2 \text{ for all }\phi\in\Hio,
\end{align}
and therefore $a(\cdot,\cdot)$ defined in \eqref{eqVB-Sesq} is coercive.
\noindent Further, note that $(\Ac,D(\Ac))$ defined in \eqref{eqdefVB-A} satisfies
\begin{equation} \label{eqVB-RelBilAc}
	\begin{aligned}
		& D(\Ac)=\{ z\in \Hio\,|\, \phi\longrightarrow a(z,\phi) \text{ is continuous in } \Lt\}, \\
		& \langle -\Ac z,\phi\rangle =a(z,\phi), \, z\in D(\Ac),\, \phi\in \Hio.
	\end{aligned}
\end{equation}

\begin{Remark}
The condition \eqref{eqdefVB-nu0} is not restrictive. We assume this to obtain $\nuh>0$ such that \eqref{eqVB-coerc} holds and hence the corresponding operator $\Ac$ is stable. If \eqref{eqdefVB-nu0} is not satisfied, then we can proceed as mentioned in {\n\cite[\textit{Remark 3.2}]{WKRPCE}}.
\end{Remark}

\noindent In the next theorem, we show that $(\Ac,D(\Ac))$ generates an analytic semigroup $\{e^{t\Ac}\}_{t\ge 0}$ on $\Lt$ by showing that it satisfies a certain resolvent estimate (see \cite[Definition 3.3]{WKRPCE}).

\begin{Theorem}[analytic semigroup] \label{thVB-anasgp}
Let $(\Ac,D(\Ac))$ as defined in \eqref{eqdefVB-A}, $y_s\in H^2(\Omega)\cap \Hio$ be solution of \eqref{eqVB-steadystate} and $\nuh$ be as defined in \eqref{eqdefVB-nu0}. Then there exists $C>0$ independent of $\mu$ such that the following statements hold.
\begin{itemize}
\item[(a)] There exists $\frac{\pi}{2}<\theta_0<\pi$ such that $\Sigma^c(-\nuh;\theta_0) \subset \rho(\Ac)$ and 
\begin{align*}
\|R(\mu,\Ac)\|_{\Lc(\Lt)} \le \frac{C}{|\mu+\nuh|}\text{ for all }\mu \in \Sigma^c(-\nuh;\theta_0),\, \mu\neq -\nuh.
\end{align*}
\item[(b)] The operator $(\Ac,D(\Ac))$ generates an analytic semigroup $\{e^{t\Ac}\}_{t\ge 0}$ on $\Lt$ with the representation
\begin{align*}
e^{t\Ac}=\frac{1}{2\pi i} \int_\Gamma e^{\mu t} R(\mu,\Ac) d\mu \text{ for all } t>0,
\end{align*}
where $\Gamma$ is any curve from $-\infty$ to $\infty$ and is entirely in $\n\Sigma^c(-\nuh;\theta_0).$
\item[(c)] Furthermore, the semigroup $\{e^{t\Ac}\}_{t\ge 0}$ satisfies $\|e^{t\Ac}\|_{\Lc(\Lt)} \le C e^{-\nuh t}$ for all $t>0.$ 
\end{itemize}
\end{Theorem}

\begin{figure}[ht!]
\begin{center}
\begin{tikzpicture}
\path[fill=black!15] (-1,0)--(-7,3.6)--(-7,-3.6)--cycle;
\draw[gray,dashed,thick,<->] (-7,0)--(3,0);
\draw[gray,dashed,thick,<->] (0,-4)--(0,4);
\draw[thick,->] (-1,0)--(-6,3);
\draw[thick,->] (-1,0)--(-6,-3);
\draw[thick,->] (-0.5,0) arc (0:150:0.5cm);
\node[] at (-0.9,0.6) {$\theta_0$};
\node[] at (0.25,-0.25){$O$};
\draw[thick,->] (-5,0) arc (180:150:4cm);
\draw[thick,->] (-5,0) arc (180:210:4cm);
\node[] at (-4.8,0.5){$\Sigma(-\nuh; \theta_0)$};
\node[] at (-1,0){$\bullet$};
\node[] at (-1,-0.3){$-\nuh$};

\draw[thick,red, dashed,->] (-1,0.85)--(-5,4);
\draw[thick,red, dashed,->] (-5,-4)--(-1,-0.85);
\draw[thick,red,dashed,thick,->] (-1,-0.85) arc (270:360:0.85cm);
\draw[thick,red,dashed] (-1,0.85) arc (90:0:0.85cm);
\node[red] at (-3.5,-2.7){$\Gamma_-$};
\node[red] at (-3.5,2.7){$\Gamma_+$};
\node[red] at (0,0.4){$\Gamma_0$};
\end{tikzpicture}
\end{center}
\caption{$\n\Sigma(-\nuh; \theta_0)$ and $\color{red}{\Gamma}=\Gamma_+\cup
\Gamma_-\cup\Gamma_0$} \label{figVB:spec-t}
\end{figure}
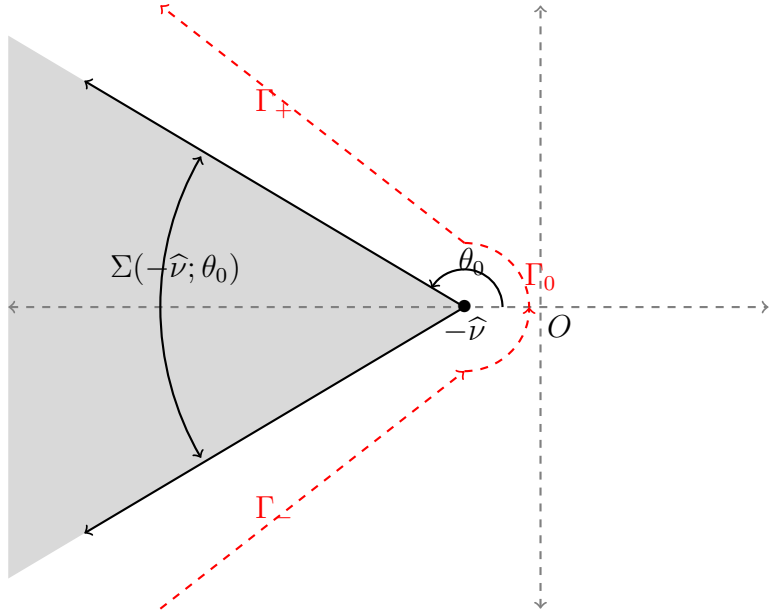

\noindent The proof of the theorem follows from \cite[Theorem 2.12, Section 2, Chapter 1, Part II]{BDDM} using \eqref{eqVB-RelBilAc}, the coercivity estimate \eqref{eqVB-coerc}, and continuity estimate \eqref{eqVB-BddBil}. A detailed proof in the present case is given in Appendix \ref{prof of thVB-anasgp}.

\begin{Remark}[regularity]\label{remVB:reg mu=-nuh}
Using elliptic regularity, we have $R(-\nuh,\Ac)\in \Lc(\Lt, D(\Ac))$ for $\mu=-\nuh,$ with 
	\begin{align*}
		\|R(-\nuh,\Ac)g\|_{H^2(\Omega)} \le C \|g\| \text{ for all }g\in \Lt,
	\end{align*}
for some $C>0.$ The proof follows using \eqref{eqVB-weakform} and \eqref{eqVB-resoleq} from the proof of Theorem \ref{thVB-anasgp} in Appendix \ref{prof of thVB-anasgp}.
\end{Remark}

\noindent The adjoint $\Ac^*:D(\Ac^*)\subset \Lt \longrightarrow \Lt$ of $\Ac$ is determined by 
\begin{equation} \label{eqdefVB-A^*}
	\begin{aligned}
		D(\Ac^*)=H^2(\Omega)\cap \Hio \text{ and } \Ac^* w=\eta \Delta w+y_s\textbf{v}\cdot \nabla w-\nu_0 w
	\end{aligned}
\end{equation}
and the adjoint operator $\Bc^*\in \Lc(\Lt)$ corresponding to $\Bc\in \Lc(\Lt)$ is determined as 
\begin{align} \label{eqdefVB-B*}
	\Bc^* \phi =\phi \chi_{\mathcal{O}} \text{ for all }\phi\in\Lt.
\end{align}

\noindent Since $(\Ac^*,D(\Ac^*))$ generates a strongly continuous semigroup, $\mu \in \rho(\Ac)$ implies $\overline{\mu}\in\rho(\Ac^*)$ and $\|R(\mu,\Ac)\|_{\Lc(\Lt)}=\|R(\mu,\Ac^*)\|_{\Lc(\Lt)},$ (\cite[Proposition 2.8.4]{Tucs}), we have the following result.
\begin{Proposition}[adjoint semigroup] \label{ppsVB-anasgp-Ac^*}
Let $(\Ac^*,D(\Ac^*))$ as defined in \eqref{eqdefVB-A^*} and $\Sigma(-\nuh,\theta_0)$ be as in Theorem \ref{thVB-anasgp}. Then there exists $C>0,$ independent of $\mu,$ such that
\begin{itemize}
\item[$(a)$] $\Sigma^c(-\nuh;\theta_0) \subset \rho(\Ac^*)$ and 
\begin{align*}
\|R(\mu,\Ac^*)\|_{\Lc(\Lt)} \le \frac{C}{|\mu+\nuh|}\text{ for all }\mu \in \Sigma^c(-\nuh;\theta_0),\, \mu\neq -\nuh.
\end{align*}
The operator $(\Ac^*,D(\Ac^*))$ generates an analytic semigroup $\{e^{t\Ac^*}\}_{t\ge 0}$ on $\Lt$ satisfying $\|e^{t\Ac^*}\|_{\Lc(\Lt)} \le C e^{-\nuh t}$ for all $t>0.$
\item[$(b)$]  For $\mu=-\nuh,$ $R(-\nuh,\Ac^*)\in \Lc(\Lt,D(\Ac^*))$ satisfying 
\begin{align*}
	\|R(-\nuh,\Ac^*)p\|_{H^2(\Omega)} \le C \|p\| \text{ for all } p\in \Lt.
\end{align*}
\end{itemize}
\end{Proposition}

\noindent In the proposition below, we discuss the behavior of the spectrum of operator $\Ac$ on $\Lt.$ Note that Theorem \ref{thVB-anasgp}(a) gives that $\sigma(\Ac)$, the spectrum of $\Ac$ is contained in $\Sigma(-\nuh;\theta_0).$ Moreover, Remark \ref{remVB:reg mu=-nuh} gives that $(-\nuh\mathbf{I}-\Ac)^{-1}\in \mathcal{L}(\Lt, D(\Ac))$, is a linear, bounded, compact operator in $\Lt$. Thus, using \cite[Theorems 6.26 and 6.29, Chapter 3]{Kato}, we get the following result. 

\begin{Proposition}[properties of spectrum of $\Ac$] \label{ppsVB:spec Af}
	Let $\n (\Ac,D(\Ac))$ be as defined in \eqref{eqdefVB-A} and $\Sigma(-\nuh;\theta_0)$ be as in Theorem \ref{thVB-anasgp} . Then we have the following:
	\begin{itemize}
		\item[(a)] The spectrum of $\Af,$ $\sigma(\Ac) \subset \Sigma(-\nuh;\theta_0).$ 
		\item[(b)] The set $\sigma(\Ac)$ contains only isolated eigenvalues of $\Ac$ and if there exists a convergence sequence $\{\lambda_k\}_{k\in\Nb}\subset \sigma(\Af),$ then $\Re(\lambda_k)\rightarrow -\infty$ as $n\rightarrow \infty.$
	\end{itemize}
\end{Proposition}

\subsection{Stabilizability} \label{subsecVB-StabCont}
Let $\omega>0$ be any given number. Let $(\Acw,\Bc)$ be as defined in \eqref{eqVB-Acw-Acw*} and \eqref{eqVB-B}, respectively. We first establish the open loop stabilizability of the pair $(\Acw,\Bc)$ and then we obtain a stabilizing control by solving an algebraic Riccati equation. To prove the open loop stabilizability it is enough to prove the Hautus condition (see \cite[Proposition 3.3, Chapter 1,
Part V]{BDDM}). Here, note that the operator $\Ac$ has variable coefficients with lower order terms, so we don't have expressions of eigenfunctions. Using the fact that the resolvent of $\Ac$ is compact and a unique continuation property stated in \cite[Theorem 15.2.1]{Tucs}, we prove the Hautus condition.

\medskip
\noindent 
From Proposition \ref{ppsVB:spec Af},  it follows that the spectrum of $\Acw,$ $\sigma(\Acw)=\{\lambda+\omega\,|\,\lambda\in \sigma(\Ac)\}\subset \Sigma(-\nuh+\omega;\theta_0).$ In fact, for any given $\omega>0,$ Proposition \ref{ppsVB:spec Af} implies that there are only finitely many eigenvalues of $\Acw$ with non-negative real part. We denote by $\{\lambda_n\}_{n\in \Nb},$ the eigenvalues of $\Ac$ and we assume that they are numbered so that $\Re(\lambda_n)\ge \Re(\lambda_{n+1})$ for all $n\in \Nb.$ Thus the eigenvalues of $\Ac_{\omega}$ are $\{\lambda_n +\omega\}_{n\in \Nb}$ and hence there exists $n_\omega\in \mathbb{N}$ such that  
 \begin{align} \label{eqVB-counteigval}
 	\Re(\lambda_n+\omega)\ge 0 \text{ for all } 1\le n\le n_\omega \text{ and } \Re(\lambda_n+\omega)<0 \text{ for all }n>n_\omega.
 \end{align} 
We denote the set of non-negative elements in $\sigma(\Acw)$ by $
\sigma_+(\Acw)=\{\lambda_n+\omega\, |\, 1\le n\le n_\omega\}$
and set of negative elements by $
\sigma_-(\Acw)=\sigma(\Acw)\diagdown \sigma_+(\Acw) =\{\lambda_n+\omega\, |\,  n> n_\omega\}.$ There exists $\epsilon>0$ such that $ \displaystyle \sup_{\Lambda \in \sigma_-(\Aw)}\Re(\Lambda) <-\epsilon.$ Since $\sigma_+(\Acw)$ has only finitely many elements, this set in $\mathbb{C}$ can be enclosed by a simple closed Jordan curve $\Gamma_u$. Then we set the projection $\pi_u\in \mathcal{L}(\Lt)$ associated  to $\sigma_+(\Acw)$ by 
\begin{align*}
	\pi_u :=\frac{1}{2\pi i}\int_{\Gamma_u} R(\mu,\Acw) \, d\mu,
\end{align*}
where $\Gamma_u$ is the simple closed Jordan curve around $\sigma_+(\Acw)$, 
and 
\begin{align}\label{eqVB-pi_s}
	\pi_s:=(I- \pi_u)\in \mathcal{L}(\Lt)
\end{align}
associated to $\sigma_-(\Acw)$. Thus, we have the following lemma and the proof is same as \cite[Lemma 3.11]{WKRPCE}.
\begin{Lemma} \label{lem:growthbound}
	Let $\Acw$ and $\pi_s$ be as defined in \eqref{eqVB-Acw} and \eqref{eqVB-pi_s}, respectively. There exists $M>0$ such that
	\begin{align*}
		\| e^{t\Acw}\|_{\Lc(\Lt)} \le M e^{-\epsilon t} \text{ for all } t>0,
	\end{align*}
where $\epsilon>0$ is such that $\displaystyle \sup_{\lambda \in \sigma_-(\Acw)} \Re(\lambda)<-\epsilon.$
\end{Lemma}

\noindent In the following proposition, we show that $(\Acw, \Bc)$ is open loop stabilizable in $\Lt$.
\begin{Proposition}[open loop stabilizability of $(\Acw,\Bc)$] \label{ppsVB-stabOpenLoop}
For any given $\omega>0,$ let $(\Acw, \Bc)$ be as defined in \eqref{eqVB-Acw-Acw*} and \eqref{eqVB-B}.  Then $(\Acw,\Bc)$ is open loop stabilizable in $\Lt.$
\end{Proposition}
\begin{proof}
For any given $\omega>0,$ 
\begin{itemize}
	\item[(a)] utilizing Theorem \ref{thVB-anasgp} and \cite[Theorem 12.37]{RROG}, $(\Acw , D(\Acw))$ generates an analytic semigroup $\{e^{ t\Acw }\}_{t\ge 0}$ on $\Lt$ and the control operator $\Bc\in \Lc(\Lt),$
	\item[(b)]  $\Acw$ has only finitely many eigenvalues with non-negative real part as mentioned in \eqref{eqVB-counteigval},
	\item[(c)] Lemma \ref{lem:growthbound} implies
	\begin{align*}
		\sup_{\lambda\in \sigma_-(\Acw)} \Re(\lambda)<-\epsilon
		\text{ and } \|e^{t\Acw}\pi_s\|_{\Lc(\Lt)} \le M e^{-\epsilon t} \text{ for all } t>0,
	\end{align*}
	where $\pi_s$ is as defined in \eqref{eqVB-pi_s}.
\end{itemize}
 Thus all the hypotheses in \cite[Proposition 3.3, Chapter 1,
 Part V]{BDDM} are verified and to conclude the proof, it remains to verify  the Hautus condition, that is,  
 \begin{align*}
 \text{Ker }(\lambda I-\Acw^*)\cap \text{Ker }(\Bc^*)=\{0\} \text{ for all }\lambda \in \sigma(\Acw) \text{ with } \Re(\lambda)\ge 0.
 \end{align*}
 Let $\phi \in \text{Ker }(\lambda I-\Acw^*)\cap \text{Ker }(\Bc^*),$ then the equation satisfied by $\phi$ (using \eqref{eqdefVB-A^*} and \eqref{eqdefVB-B*}) is  
 \begin{align}
 	-\eta\Delta \phi - y_s \textbf{v}\cdot \nabla \phi +(\lambda+\nu_0)\phi =0 \text{ in }\Omega \text{ and } \phi=0 \text{ on } \mathcal{O}\cup \partial \Omega.
 \end{align}
Now, $\mathcal{O}$ being a non-empty open set in $\Omega,$ unique continuation principle \cite[Theorem 15.2.1]{Tucs} for elliptic equations implies that $\phi=0$ on whole $\Omega.$  Thus by \cite[Proposition 3.3, Chapter 1,
Part V]{BDDM}, $(\Acw, \Bc)$ is open loop stabilizable in $\Lt.$
\end{proof}

\noindent To obtain the stabilizing control in feedback form, consider an optimal control problem:
\begin{align}\label{eqVB-OCP}
	\min_{\wt u\in E_{z_0}} J (\wt z,\wt u) \text{ subject to } \eqref{eqVB-shifsysLin},
\end{align}
where 
\begin{equation}
	\begin{aligned}
		&	J(\wt z, \wt u):=\int_0^\infty \left(\|\wt z(t)\|^2 +\|\wt u(t)\|^2\right)\, dt, \text{ and }\\
		& E_{z_0}:=\{ \wt u\in L^2(0,\infty;\Lt)\, |\, \wt z \text{ solution of }\eqref{eqVB-shifsysLin} \text{ with control } \wt u\text{ such that } J(\wt z,\wt u)<\infty \}.
	\end{aligned}
\end{equation}

\noindent Utilizing Proposition \ref{ppsVB-stabOpenLoop}, we have the following result and the proof is similar as in \cite[Theorem 2.1]{WKRPCE}.

\begin{Theorem} \label{thVB-mainstab}
	Let $\omega>0$ be any given real number. Let the operators $\Acw$ and $\Bc$ be as defined in \eqref{eqVB-Acw-Acw*} and \eqref{eqVB-B}, respectively. Then the following results hold:
	\begin{itemize}
		\item[$(a)$] There exists a unique operator $\Pc\in \Lc(\Lt)$ satisfying the algebraic Riccati equation
		\begin{equation}\label{eqVB-Riccati}
			\begin{aligned}
				& \Pc\in \Lc(\Lt),\, \Pc=\Pc^* \ge 0,\\
				& \Pc \Acw+\Acw^*\Pc-\Pc\Bc\Bc^*\Pc+I=0 \text{ on }\Lt.
			\end{aligned}
		\end{equation}
		\item[$(b)$] There exists a unique optimal pair $(z^\sharp(t), u^\sharp(t))$ for \eqref{eqVB-OCP}, where 
		\begin{align}\label{eqVB-u-sharp}
			u^\sharp(t)=-\Bc^*\Pc z^\sharp(t)
		\end{align}
		and $ z^\sharp(t)$ is solution of the closed loop system
		\begin{align} \label{eqVB-clsdloopLin}
			z{^\sharp}{'}(t)=(\Acw-\Bc\Bc^*\Pc) z^\sharp(t) \text{ for all } t>0, \quad  z^\sharp(0)=z_0.
		\end{align}
		\item[$(c)$] The optimal cost of $J(\cdot, \cdot)$ is given by $\displaystyle \min_{\wt u\in E_{z_0}} J(\wt z, \wt u)=J(z^\sharp, u^\sharp )=\langle \Pc z_0 ,z_0\rangle.$ 
		\item [(d)] Denoting $\Ac_{\omega,\Pc}:=\Acw-\Bc\Bc^*\Pc$ with $D(\Ac_{\omega,\Pc})=D(\Ac),$ the semigroup $\{e^{t(\Ac_{\omega,\Pc})}\}_{t\ge 0}$ generated by $(\Ac_{\omega,\Pc}, D(\Ac_{\omega,\Pc}) )$ on $\Lt$ is exponentially stable, that is, there exist $\gamma>0$ and $M>0$ such that 
		\begin{align*}
			\| e^{t \Ac_{\omega,\Pc}}\|_{\Lc(\Lt)}\le M e^{-\gamma t} \text{ for all } t>0.
		\end{align*}
		The feedback control $u^\sharp$ in \eqref{eqVB-u-sharp} stabilizes \eqref{eqVB-shifsysLin}.
	\end{itemize} 
\end{Theorem}


\subsection{Finite element approximation} \label{secVB-FEM}
This subsection aims to construct an approximate system corresponding to \eqref{eqVB-shifsysLin} and study its stabilizability by solving an algebraic Riccati equation in a discrete set-up. Further, we are also interested in establishing an error estimate for the Riccati operators, the stabilized solution of the linear system, and stabilizing control. 


\subsubsection{Approximated linear system}
Let the triangulation $\mathcal{T}_h$ of $\overline{\Omega},$ approximated space  $V_h$ of $\Lt$ and define the projection operator $\pi_h:\Lt\rightarrow V_h$ as 
\begin{align} \label{eqn:def of pi_h}
	\langle \pi_h v, \phi_h\rangle = \langle v , \phi_h\rangle \text{ for all } \phi_h \in V_h.
\end{align}
 The orthogonal projection $\pi_h$ has the following properties:
 \begin{Lemma} \label{lem:projerr}
 	The projection operator $\pi_h :\Lt \rightarrow V_h$ defined in \eqref{eqn:def of pi_h} satisfies the following estimates for some $C>0$ independent of $h$:
 	\begin{itemize}
 		\item[(a)]  $\|\pi_h v -v\| \rightarrow 0$ as $h\rightarrow 0$ and $\|\pi_h v\| \le \|v\|$ for all $v \in \Lt,$
 		\item[(b)] $\pi_h^2=\pi_h$ and $\pi_h (I-\pi_h)=0=(I-\pi_h)\pi_h,$
 		\item[(c)] for all $v \in H^2(\Omega)\cap \Hio,$ $\|v-\pi_h v\| \le Ch^2 \|v\|_{H^2(\Omega)}$ and $\|\nabla (v-\pi_h v)\| \le C h \|v\|_{H^2(\Omega)}.$  
 	\end{itemize}
 \end{Lemma}

\noindent For each $h>0,$ we define the discrete operator $\Ac_h$ on $V_h$ corresponding to $\Ac$ by 
\begin{align}\label{eqVB-DisBil}
\langle-\Ac_hz_h,\phi_h)\rangle= a(z_h,\phi_h)= \langle \eta \nabla z_h, \nabla \phi_h\rangle +\langle y_s \textbf{v}\cdot \nabla z_h, \phi_h\rangle +\langle \textbf{v}\cdot\nabla y_s z_h, \phi_h\rangle +\nu_0\langle z_h, \phi_h\rangle 
\end{align}
 for all $z_h,\phi_h\in V_h,$ where the sesquilinear form $a(\cdot,\cdot)$ is defined in \eqref{eqVB-Sesq}.

\noindent In the following theorem, we show that the eigenvalues of the discrete operator $\Ac{_h}$ are contained in a uniform sector $\Sigma(-\nuh;\theta_0)$ and outside this a uniform resolvent estimate holds. In particular, we show that the family of the operators $\{\Ac{_h} \}_{h>0}$ generates a uniform analytic semigroup $\{e^{t\Ac{_h}}\}_{t\ge 0}$ on $V_h$. 

\begin{Theorem}\label{thVB-UniAnaSG}
Let $\Ac{_h}$  and $\Sigma(-\nuh;\theta_0)$ be as introduced in \eqref{eqVB-DisBil} and Theorem \ref{thVB-anasgp}, respectively. Then for all $h>0,$ the following results hold:
\begin{itemize}
\item[(a)] The spectrum of $\Ac{_h},$ $\sigma(\Ac{_h}) \subset \Sigma(-\nuh;\theta_0)$ and for all $\mu \in \Sigma^c(-\nuh;\theta_0),$ $\mu\neq -\nuh,$ there exists $C>0$ independent of $\mu$ and $h$ such that 
\begin{align*}
\|R(\mu,\Ac{_h})\|_{\Lc(V_h)} \le \frac{C}{|\mu +\nuh|}.
\end{align*}
\item[(b)] The operator $\Ac{_h}$ generates a uniformly analytic semigroup $\{e^{t\Ac{_h}}\}_{t\ge 0}$ on $V_h$ and the operator $e^{t\Ac{_h}}\in \Lc(V_h)$ can be represented by 
\begin{align*}
e^{t\Ac{_h} }=\frac{1}{2\pi i} \int_\Gamma e^{\mu t} R(\mu,\Ac{_h}) d\mu, \, t>0,
\end{align*}
where $\Gamma$ is any curve from $-\infty$ to $\infty$ and is entirely in $\n\Sigma^c(-\nuh;\theta_0).$
\item[(c)] Furthermore, the semigroup $\{e^{t\Ac{_h} }\}_{t\ge 0}$ satisfies $\|e^{t\Ac{_h} }\|_{\Lc(V_h)} \le C e^{-\nuh t},$ for all $t>0,$ for some $C>0$ independent of $h.$ 
\end{itemize}
\end{Theorem}
\begin{proof}
First, for all $\mu\in \Cb$ with $\Re(\mu)\ge -\nuh$ and for any given $g_h\in V_h,$ aim is to find a unique $z_h\in V_h$ such that 
\begin{align*}
	\left\langle (\mu I_h-\Ac_h) z_h ,\phi_h\right\rangle= a(z_h,\phi_h) +\mu\left\langle z_h, \phi_h\right\rangle = \left\langle g_h , \phi_h\right\rangle \text{ for all }\phi_h \in V_h,
\end{align*}
where $a(\cdot,\cdot)$ is as introduced in \eqref{eqVB-Sesq}. For all $h>0,$ \eqref{eqVB-BddBil} and \eqref{eqVB-coerc} imply the boundedness and coercivity of $a(\cdot,\cdot)$ on the space $V_h$ with constants $\alpha_1$ and $\frac{\eta}{2},$ independent of $h.$ Therefore, as in Theorem \ref{thVB-anasgp}, for all $\sigma(\Ac_h)\subset \Sigma(-\nuh;\theta_0)$ and there exists a unique $z_h\in V_h$ such that the last displayed equality holds. Now, proceeding as in Theorem \ref{thVB-anasgp}, we obtain (a). As the constant appearing in (a) is independent of $h,$ for all $h>0,$ a similar arguments as in Theorem \ref{thVB-anasgp}(b)-(c) conclude the proof of (b) and (c). 
\end{proof}

\noindent Next, the aim is to establish an error estimate for the system without control, which is crucial to obtain the error estimate for the stabilized system. Consider the dynamics,
\begin{align*}
	z'(t)=\Ac z(t) \text{ for all }t>0, \, z(0)=z_0\in \Lt,
\end{align*}
and its approximation 
\begin{align*}
	z_h'(t)=\Ac_h z_h(t) \text{ for all } t>0, \, z_h(0)=\pi_h z_0,
\end{align*}
whose solutions can be represented by 
\begin{align*}
	z(t)=e^{t\Ac}z_0 \text{ and } z_h(t)= e^{t\Ac_h}\pi_h z_0 \text{ for all } t>0.
\end{align*}
Now, to study the error estimate of the system without control, we need a result stated in the lemma below.

\begin{Lemma} \label{lemVB:resolErrorAAh}
Let $\Ac$ and $\Ac{_h}$ be as introduced in \eqref{eqdefVB-A} and \eqref{eqVB-DisBil}, respectively. Let $\Sigma(-\nuh;\theta_0)$ be as in Theorem \ref{thVB-anasgp}. Then there exists a constant $C>0$ independent of $\mu$ and $h$ such that the resolvent operator satisfies
\begin{itemize}
	\item[(a)]  $\|R(-\nuh,\Ac)-R(-\nuh,\Ac{_h})\pi_h\|_{\Lc(\Lt)} \le C h^2$ and
	\item[(b)]  $\displaystyle \sup_{\mu \in \Sigma^c(-\nuh;\theta_0)}\|R(\mu,\Ac)-R(\mu,\Ac{_h})\pi_h\|_{\Lc(\Lt)} \le C h^2.$
\end{itemize}
\end{Lemma}

\begin{proof}
(a) From Theorems \ref{thVB-anasgp} and \ref{thVB-UniAnaSG}, note that, for all $h>0,$ $-\nuh$ defined in \eqref{eqdefVB-nu0} belongs to $\rho(\Ac)\cap \{\cap_{h>0}\rho(\Ac_h)\}.$ Also, from Remark \ref{remVB:reg mu=-nuh}, for any given $g\in \Lt,$ there exist $u\in D(\Ac)$ and $u_h\in V_h$ such that $R(-\nuh,\Ac)g=u$ and $R(-\nuh,\Ac_h)u_h=\pi_h g$ in the sense that 
\begin{align*}
	a(u,\phi)-\nuh(u,\phi)=\langle g,\phi\rangle \text{ and } a(u_h,\phi_h)-\nuh\langle u_h,\phi_h\rangle=\langle\pi_h g,\phi_h\rangle,
\end{align*}
for all $\phi\in \Hio $ and $\phi_h\in V_h.$
Subtracting the second equality from the first one above and utilizing \eqref{eqn:def of pi_h}, we obtain 
\begin{align} \label{eqVB-orthonormalBil}
	a(u-u_h,\phi_h) -\nuh \langle u-u_h,\phi_h \rangle  =\langle g-\pi_h g,\phi_h\rangle = 0 \text{ for all }\phi_h\in V_h.
\end{align}
A use of \eqref{eqVB-orthonormalBil} with $\phi_h=\pi_h u$ implies
\begin{align*}
	a(u-u_h, u-u_h)-\nuh\langle u-u_h, u-u_h\rangle = a(u- u_h, u-\pi_h u) - \nuh \langle u- u_h, u-\pi_h u\rangle. 
\end{align*}
Therefore, utilizing this, coercivity of $ a(\cdot,\cdot)$ as in \eqref{eqVB-coerc} and Lemma \ref{lem:projerr}(c) for $u\in D(\Ac)$, we have 
\begin{align*}
\frac{\eta}{2}\|\nabla(u-u_h)\|^2 \le &  \alpha_1 \|\nabla(u-u_h)\|\,\|\nabla(u-\pi_hu)\|  \le \alpha_1 C h\|\nabla(u-u_h)\|\, \|u\|_{H^2(\Omega)}.
\end{align*}
 Therefore, a application of Remark \ref{remVB:reg mu=-nuh} in the above inequality leads to 
\begin{align} \label{eqVB-EnrgNmEst}
\|\nabla (u-u_h)\|\le \frac{2\alpha_1 C}{\eta} h \|g\|.
\end{align}
Now, we employ a duality argument for the dual problem: for a given $p\in \Lt,$ solve the following equation 
\begin{align} \label{eqVB-dualArgPro}
(-\nuh I-\Ac)^*\Phi=p \text{ in }\Omega, \quad \Phi=0 \text{ on }\partial\Omega
\end{align}
for $\Phi.$ Proposition \ref{ppsVB-anasgp-Ac^*}(b) yields a unique solution $\Phi \in D(\Ac^*)=H^2(\Omega)\cap H^1_0(\Omega)$ of \eqref{eqVB-dualArgPro}.  Now, 
{\small
\begin{align*}
\langle u -u_h, p\rangle & =\langle u-u_h, (-\nuh I - \Ac)^*\Phi\rangle  = \langle u -u_h, -\eta\Delta  \Phi - y_s \textbf{v}\cdot \nabla \Phi+ \nu_0 \Phi - \nuh \Phi\rangle \\
& = \eta\langle\nabla(u-u_h), \nabla \Phi\rangle + \langle y_s \textbf{v}\cdot \nabla( u-u_h)+ \textbf{v}\cdot \nabla y_s (u-u_h), \Phi\rangle +(\nu_0-\nuh)\langle u -u_h, \Phi\rangle, 
\end{align*}
}and therefore utilizing \eqref{eqVB-Sesq} and \eqref{eqVB-orthonormalBil}, we have
\begin{align*}
	\langle u -u_h, p\rangle = a (u-u_h, \Phi) -\nuh\langle u-u_h, \Phi\rangle= a(u-u_h, \Phi-\pi_h\Phi)-\nuh\langle u-u_h, \Phi-\pi_h\Phi\rangle.
\end{align*}
The boundedness of $a(\cdot,\cdot)$ given in \eqref{eqVB-BddBil}, \eqref{eqVB-EnrgNmEst} and Lemma \ref{lem:projerr}(c) imply
\begin{align}
	\left\vert \langle u-u_h, p\rangle\right\vert \le \alpha_1 \|\nabla(u-u_h)\|\, \|\nabla(\Phi-\pi_h\Phi)\| \le \frac{2C \alpha_1^2}{\eta}h^2 \|g\|\, \|\Phi\|_{H^2(\Omega)}.
\end{align}
Now, choose $p=u-u_h$ in the above inequality and  estimate $\|\Phi\|_{H^2(\Omega)}\le C \|p\|=C\|u-u_h\|$ as in Proposition \ref{ppsVB-anasgp-Ac^*}(b) to obtain 
\begin{align*}
\|u-u_h\|^2 \le \frac{2C\alpha_1^2}{\eta} h^2 \|g\|\,\|u-u_h\|.
\end{align*}
Therefore, finally, we have
\begin{align*}
\|R(-\nuh,\Ac)g-R(-\nuh,\Ac{_h})\pi_h g\|=\|u-u_h\| \le  \frac{2\alpha_1^2 C}{\eta} h^2 \|g\| = C h^2 \|g\|,
\end{align*}
for some $C=C(\alpha_1,\eta)>0$ independent of $h.$ This concludes the proof.

\medskip
\noindent (b) The proof follows the same line as in the proof of \cite[Lemma 4.9(b)]{WKRPCE}.
\end{proof}

\noindent The error estimates for the system without control are established in the next two theorems. These two results are analogous to \cite[Theorems 4.10 - 4.11]{WKRPCE}, and the proof follows the same line. Hence, here we skip the proof.
\begin{Theorem}[error estimate for the system without control] \label{thVB-ErrWC}
	Let the operators $\Ac,$ $\Ac_h$ and $\pi_h$ be as defined in \eqref{eqdefVB-A}, \eqref{eqVB-DisBil} and \eqref{eqn:def of pi_h}, respectively. Then for any $z_0\in \Lt,$ the semigroups $\{e^{t\Ac}\}_{t\ge 0}$ and $\{e^{t\Ac_h}\}_{t\ge 0}$ generated by $\Ac$ and $\Ac_h,$ respectively, satisfy
	\begin{itemize}
		\item[(a)] $\left\| (e^{t\Ac}-e^{t\Ac_h}\pi_h)z_0\right\|\le C h^2 \frac{e^{-\nuh t}}{t} \|z_0\|$ for all $t>0,$
		\item[(b)] $\left\| (e^{t\Ac}-e^{t\Ac_h}\pi_h)z_0\right\|_{L^2(0,\infty;\Lt)} \le C_\theta h^{2\theta}\|z_0\|$ for any $0<\theta<\frac{1}{2},$
	\end{itemize}
for some $C>0$ and $C_\theta>0$ independent of $h.$ 
\end{Theorem}

\noindent Theorem \ref{thVB-ErrWC} yields that the error estimates between the trajectories corresponding to the continuous and discrete system hold uniformly in $t$ for all $t\ge T$ for any $T>0.$ Next, based on Trotter-Katos's theorem, we have the following convergence result for short time $t\in [0,T].$

\begin{Theorem}
	Let $\{e^{t\Ac}\}_{t\ge 0}$ (resp. $\{e^{t\Ac_h}\}_{t\ge 0}$) be the semigroup generated by $\Ac$ (resp. $\Ac_h$) defined in \eqref{eqdefVB-A} (resp. \eqref{eqVB-DisBil}). Let $\pi_h$ be as defined in \eqref{eqn:def of pi_h}. Then for any $z_0\in \Lt$ and for any fixed $T>0,$ we have 
	\begin{align*}
		\sup_{t\in [0,T]} \left\| (e^{t\Ac}-e^{t\Ac_h}\pi_h)z_0\right\| \rightarrow 0 \text{ as } h\downarrow 0.
	\end{align*}
\end{Theorem}

\noindent Next, the aim is to establish the uniform stabilizability of the corresponding approximate linear system (see \eqref{eqVB-approxSys} below) and error estimate. 

\noindent Let $\Bc_h$ be approximating operator corresponding to $\Bc$ defined by 
\begin{align}\label{eqdefVB-Bh}
	\Bc_h=\pi_h\Bc.
\end{align}
With this, consider the discrete system corresponding to \eqref{eqVB-MainOptForm} with $\kappa=0$ as 
\begin{equation} \label{eqVB-approxSys}
	\begin{aligned}
		z_h'(t)=\Ac{_h}z_h(t)+\Bc{_h} u_h(t) \text{ for all } t>0, \, z_h(0)= z_{0_h}. 
	\end{aligned}
\end{equation}
To obtain stabilizability of \eqref{eqVB-approxSys} with decay $-\omega<0,$ for any $\omega>0,$ define approximated operator $\Ac_{\omega_h}:V_h \rightarrow V_h$ corresponding to $\Acw$ by 
\begin{align}\label{eqdefVB-Awh}
	\Ac_{\omega_h}=\Ac_h+\omega I_h,
\end{align}
where $I_h:V_h\rightarrow V_h$ is the identity operator. Now, consider the discrete system that corresponds to \eqref{eqVB-shifsysLin} as
\begin{equation} \label{eqVB-approxShiftSys}
	\begin{aligned}
		\wt z_h'(t)=\Ac_{\omega_h}\wt z_h(t)+\Bc{_h}\wt u_h(t)  \text{ for all } t>0, \, \wt z_h(0)= z_{0_h}, 
	\end{aligned}
\end{equation}
where $z_{0_h}\in V_h$ is an approximation of $z_0.$

\medskip 
\noindent Now, we summarize the properties satisfied by the linear operators $\Ac_{\omega}, \Ac_{\omega_h}$ and control operators $\Bc, \Bc_h.$ These are required (see Property $(\mathcal{A}_1)$ in \cite[Section 5]{WKRPCE}) to establish the stabilizability of \eqref{eqVB-approxShiftSys} and error estimates.

\noindent Property \textbf{$(\mathcal{A}_1).$}
\begin{itemize}
	\item[(a)] For all $h>0,$ the control operators $\Bc$ and $\Bc_h$ given in \eqref{eqVB-B} and \eqref{eqdefVB-Bh} satisfy
	\begin{align*}
		\|\Bc\|_{\Lc(\Lt)} \le C_B \text{ and } \|\Bc_h\|_{\Lc(\Lt,V_h)} \le C_B, 
	\end{align*}
	for some positive constant $C_B$ independent of $h.$
	\item[(b)] In view of \cite[Theorem 12.37]{RROG} and Theorem \ref{thVB-anasgp}, the operator $(\Ac_{\omega}, D(\Ac_{\omega}))$ defined in \eqref{eqVB-Acw-Acw*} generates an analytic semigroup $\{e^{t\Ac_{\omega}}\}_{t\ge 0}$ on $\Lt$ with $\Sigma^c(-\nuh+\omega;\theta_0)\subset \rho(\Ac_{\omega})$ and 
	\begin{align*}
		\|R(\mu,\Ac_{\omega})\|_{\Lc(\Lt)} \le \frac{C_1}{|\mu+\nuh-\omega|} \text{ for all }\mu \in \Sigma^c(-\nuh+\omega;\theta_0), \, \mu\neq -\nuh+\omega,
	\end{align*}
	for some positive constant $C_1$ independent of $\mu.$
	\item[(c)]  Theorem \ref{thVB-UniAnaSG} and \cite[Theorem 12.37]{RROG} lead to the following. For all $h>0$, the operator $\Ac_{\omega_h}$ defined in \eqref{eqdefVB-Awh} generates a uniformly analytic semigroup $\{e^{t\Ac_{\omega_h}}\}_{t\ge 0}$ on $V_h$ with $\Sigma^c(-\nuh+\omega;\theta_0)\subset \rho(\Ac_{\omega_h})$ and 
	\begin{align*}
		\|R(\mu,\Ac_{\omega_h})\|_{\Lc(V_h)} \le \frac{C_1}{|\mu+\nuh-\omega|} \text{ for all }\mu \in \Sigma^c(-\nuh+\omega;\theta_0), \, \mu\neq -\nuh+\omega,
	\end{align*}
	for some positive constant $C_1$ independent of $\mu$ and $h.$
	\item[(d)] As a consequence of Lemma \ref{lemVB:resolErrorAAh}, the operators $(\Ac_{\omega}, D(\Ac_{\omega}))$ and $\Ac_{\omega_h}$ defined in \eqref{eqVB-Acw-Acw*} and  \eqref{eqdefVB-Awh} satisfy
	\begin{align*}
		\sup_{\mu \in \Sigma^c(-\nuh+\omega;\theta_0)} \| R(\mu, \Ac_{\omega}) - R(\mu, \Ac_{\omega_h})\pi_h\|_{\Lc(\Lt)} \le Ch^2, 
	\end{align*}
	for all $h>0$ and for some constant $C>0$ independent of $\mu$ and $h.$
\end{itemize}
 
\noindent The aforementioned properties are analogous to those outlined in \cite[Section 5]{WKRPCE} and play a pivotal role in establishing the current results. The analysis done in \cite[Sections 5-7]{WKRPCE} can be mimicked here to show the following results. The analyses done in \cite[Sections 5-7]{WKRPCE} remain applicable in the present context, given the fulfillment of Property $(\mathcal{A}_1)$.  Here, we omit the detailed repetition of those analyses here.

\medskip
\noindent To obtain a feedback stabilizing control for the system \eqref{eqVB-approxShiftSys}, consider the optimal control problem:
\begin{align} \label{eqVB-DisOCP}
	\min_{\wt u_h \in E_{hz_{0_h}}} J_h(\wt z_h, \wt u_h) \text{ subject to } \eqref{eqVB-approxShiftSys},
\end{align}
where 
\begin{align} \label{eqVB-DisFun}
	J_h(\wt z_h, \wt u_h):= \int_0^\infty \left( \|\wt z_h(t)\|^2 +  \|\wt u_h(t)\|^2\right) dt,
\end{align}
{\small$E_{hz_{0_h}}:=\{ \wt{u}_h\in L^2(0,\infty; L^2(\Omega))\mid J_h(\wt z_h,\wt u_h)<\infty,\,\wt z_h \text{ is solution of }\eqref{eqVB-approxShiftSys} \text{ with control }\wt u_h\}.$}

\noindent In the next theorem, we establish that for each $h>0,$ the optimal control problem \eqref{eqVB-DisOCP} has a unique minimizer and the minimizing control is obtained in feedback form by solving a discrete algebraic Riccati equation posed on $V_h.$ 

\noindent
\begin{Theorem}[uniform stabilizability and discrete Riccati operator]\label{thVB:dro}
	Let $\Ac_{\omega_h}$ and $\Bc_h$ be as defined in \eqref{eqdefVB-Awh} and \eqref{eqdefVB-Bh}, respectively. Then there exists $h_0>0$ such that  for all $0<h<h_0$, the results stated below hold:
	\begin{itemize}
		\item[$(a)$] There exists a unique, non-negative, self-adjoint  Riccati operator $\n\Pc_h\in \Lc(V_h)$ associated with \eqref{eqVB-approxShiftSys} that satisfies the discrete Riccati equation
		\begin{align} \label{eqVB-d-ARE}
			\Ac^*_{\omega_h}\Pc{_h}+\Pc{_h}\Ac_{\omega_h}-\Pc{_h}\Bc_h\Bc^*_h\Pc{_h}+I_h=0,\; \Pc{_h}=\Pc{_h}^*\geq 0 \text{ on }V_h.
		\end{align}
		\item[$(b)$] There exists a unique optimal pair $( z^\sharp_h, u^\sharp_h)$ for \eqref{eqVB-DisOCP},  
		where $ z^\sharp_h(t)$ is solution of corresponding closed loop system
		\begin{equation}\label{eqVB-d-cl-loop}
			{z^\sharp_h}{'}(t)=(\Ac_{\omega_h}-\Bc_h\Bc_h^*\Pc{_h}) z^\sharp_h(t) \text{ for all } t>0,\quad  z^\sharp{_h}(0)=z_{0_h},
		\end{equation}
		and $ u^\sharp_h(t)$ can be expressed in the feedback form as 
		\begin{align}\label{eqVB-uh-sharp}
			u^\sharp_h(t)=-\Bc_h^*\Pc_h z^\sharp_h(t).
		\end{align}
		The optimal cost of $\n J_h(\cdot,\cdot)$ is given by 
		\begin{align}
			\displaystyle\n\min_{\wt u_h\in E_{hz_{0_h}}}J_h(\wt z{_h},\wt u_h)=J_h(z^\sharp_h,u^\sharp_h)=\left\langle \Pc_h z_{0_h},z_{0_h}\right\rangle.
		\end{align} 
		\item[$(c)$] The operator $\Ac_{\omega_h,\Pc_h}:=\Ac_{\omega_h}-\Bc_h\Bc_h^*\Pc_h$ generates uniformly analytic semigroup $\{e^{t\Ac_{\omega_h,\Pc_h}}\}_{t\ge 0}$ on $V_h$ satisfying
		\begin{align*}
			\|e^{t\Ac_{{\omega_h,\Pc_h}}}\|_{\Lc(V_h)}\leq
			M_Pe^{-\omega_P t} \text{ for all }t>0, \text{ for all } 0<h<h_0,
		\end{align*}
		for some positive constants $\normalfont\omega_P$ and $\normalfont M_P$ independent of $h$. 
	\end{itemize}
\end{Theorem}


\begin{Theorem}[error estimates for Riccati and cost functional]\label{thVB:main-conv-P}
	Let $\Pc,$ and $(z^\sharp,u^\sharp),$ for any $z_0\in \Lt,$ be as obtained in Theorem \ref{thVB-mainstab}. Let $h_0,$ $\Pc_h,$ and $(z_h^\sharp, u_h^\sharp)$, for $z_{0_h}=\pi_h z_0$ be as obtained in Theorem \ref{thVB:dro}.
	Then there exists $\wt h_0\in (0,h_0)$ such that for any given $0<\epsilon<1$ and for all $0<h<\wt h_0,$ the estimates below hold:
	\begin{itemize}
		\item[$(a)$]  $\;\|\Pc-\Pc_h\pi_h\|_{\Lc(\Lt)}\le C h^{2(1-\epsilon)} ,$  
		{ \n (b) } $\n \left\vert  J(z^\sharp,u^\sharp) - J_h(z_h^\sharp,u^\sharp_h)\right\vert\le C h^{2(1-\epsilon)}, $
		\item[$(c)$] $
		\|\Bc^*\Pc - \Bc_h^*\Pc_h\pi_h\|_{\Lc(\Lt)}   \le C h^{2(1-\epsilon)} $ and { $(d)$} $\|\Bc^*\Pc-\Bc_h^*\Pc_h\|_{\Lc(V_h,\Lt)} \le  Ch^{2(1-\epsilon)}.$
	\end{itemize}
	Here, the constant $C>0$ is independent of $h$ but depends on $\epsilon.$
\end{Theorem}
\begin{Theorem}[error estimates for stabilized solutions and stabilizing control]\label{thVB:main-conv-new}
	Let $\gamma,$ and $(z^\sharp,u^\sharp),$ for any $z_0\in \Lt,$ be as obtained in Theorem \ref{thVB-mainstab}. Let $h_0,$ $\omega_P,$ and $(z_h^\sharp, u_h^\sharp)$, for $z_{0_h}=\pi_h z_0$ be as obtained in Theorem \ref{thVB:dro}. For any $\wt\gamma$ satisfying $\wt\gamma <\min\{\gamma,\omega_P\},$  there exists $\wt h_0 \in (0,h_0)$ such that for any $0<\epsilon<1$ and for all $0<h<\wt h_0$, the following estimates hold: 
	\begin{itemize}
		\item[$(a)$] $\n\|z^\sharp(t) - z^\sharp_h(t)\|\le C h^{2(1-\epsilon)} \frac{e^{-\wt\gamma t}}{t} \|z_0\|  $ for all $t>0,$   
		\item[$ (b)$] $\n\|z^\sharp(\cdot) - z^\sharp_h(\cdot)\|_{L^2(0,\infty;\Lt)}\le C h^{1-\epsilon} \|z_0\| ,$
		\item[$(c)$] $\n\|u^\sharp(t) - u^\sharp_h(t)\|\le C h^{2(1-\epsilon)}  \frac{e^{-\wt\gamma t}}{t} \|z_0\| $ for all $t>0,$ and 
		\item[{ $(d)$}] $\n\|u^\sharp(\cdot) - u^\sharp_h(\cdot)\|_{L^2(0,\infty;\Lt)}\le C h^{1-\epsilon} \|z_0\|.$ 
	\end{itemize}
	Here, the constant $C>0$ is independent of $h$ but depends on $\epsilon,\gamma$ and $\omega_P.$
\end{Theorem}

%
%
%
%

\subsection{Implementation} \label{subVB-NI-Lin}

For any $h>0,$ let $n_h$ be the dimension of $V_h$ which is equal to the number of interior nodes for the triangulation $\mathcal{T}_h$ due to homogeneous Dirichlet boundary condition.  The semi-discrete formulation corresponding to \eqref{eqVB-approxShiftSys} seeks $\wt z_h(\cdot) \in V_h$ such that
{\small
	\begin{equation}\label{eqVB-semidisNI}
		\begin{aligned}
			& \langle \widetilde{z_h}'(t), \phi_h\rangle  = -\eta \langle\nabla \widetilde{z_h}(t),\nabla\phi_h\rangle -\langle y_s \textbf{v}\cdot \nabla \wt z_h, \phi_h\rangle -\langle \textbf{v}\cdot\nabla y_s \wt z_h, \phi_h\rangle  +(\omega-\nu_0)\langle \widetilde{z_h}(t),\phi_h\rangle+\langle \widetilde{u}_h, \phi_h\rangle, \\
			& \langle \wt z_h(0), \phi_h\rangle = \langle  z_0, \phi_h\rangle,
		\end{aligned}
	\end{equation}
}for all  $\phi_h\in V_h.$
Let  $\displaystyle \widetilde{z_h}(t):=\sum_{i=1}^{n_h}z_i(t)\phi_h^i$, where $\{\phi_h^i\}_{i=1}^{n_h}$ is
the canonical basis functions of $V_h.$ A substitution of this to \eqref{eqVB-semidisNI} leads to the system 
\begin{equation} \label{eqnVB:discrete}
	{\Mrm}_h {\Zrm}_h'(t)=({\Arm}_h+\omega {\Mrm}_h){\Zrm}_h(t)+\Brm_h\mathfrak{u}_h \text{ for all } t>0, \quad {\Mrm}_h{\Zrm}_h(0)=(\langle  z_0, \phi_h\rangle),
\end{equation}
where ${\Arm}_h=-\eta{\Krm}_h - {\Arm}^1_h - {\Arm}^2_h-\nu_0{\Mrm}_h$ with  ${\Krm}_h=  (\langle \nabla \phi_h^i,\nabla\phi_h^j\rangle)_{1\le i,j\le n_h},$  ${\Arm}^1_h=(\langle y_s \textbf{v}\cdot \nabla \phi_h^i, \phi_h^j\rangle)_{1\le i,j\le n_h},$  $ {\Arm}^2_h=(\langle \textbf{v}\cdot \nabla y_s  \phi_h^i, \phi_h^j\rangle)_{1\le i,j\le n_h} ,$  ${\Mrm}_{h}=(\langle \phi_h^i,\phi_h^j\rangle)_{1\le i,j\le n_h},$ ${\Zrm}_h:=(z_1,...,z_{n_h})^T$ and $\mathfrak{u}_h=(u_1,...,u_{n_h})^T\in\mathbb{R}^{n_h\times 1}$ is the control. Since, the Gram matrix ${\Mrm}_{h}$ is invertible, Picard's existence theorem imply that for each $h>0,$ \eqref{eqnVB:discrete} has a unique global solution.

\medskip
\noindent To implement an example, we follow the same procedure, and use a time solver as in \cite[Sections 8.1 and 8.2]{WKRPCE}. In fact, we obtain the computed stabilizing control $\mathfrak{u}_h(t)=-{\Mrm}_{h}^{-1} \Brm_h^T {\Mrm}_{h}^{-1} \Prm_h {\Zrm}_h(t),$ where $\Prm_h$ is the unique solution of

\begin{equation} \label{eq:ARE-mat}
	\Prm_h ({\Mrm}_{h}^{-1}  ({\Arm}_h+\omega {\Mrm}_h) ) + ({\Arm}_h+\omega {\Mrm}_h)^T  {\Mrm}_{h}^{-1}\Prm_h -  \Prm_h {\Mrm}_{h}^{-1} \Brm_h {\Mrm}_{h}^{-1} \Brm_h^T {\Mrm}_{h}^{-1} \Prm_h + {\Mrm}_{h}= 0, \quad \Prm_h^T=\Prm_h\ge 0,
\end{equation}
(see (8.5) in \cite{WKRPCE}) and ${\Zrm}_h(t)$ is solution of the closed loop system 
\begin{align*}
	{\Mrm}_h {\Zrm}_h'(t)= \left(  ({\Arm}_h+\omega {\Mrm}_h) - \Brm_h  \mathbb{S}_h \Prm_h \right) {\Zrm}_h(t)\text{ for all } t>0, \quad {\Mrm}_h{\Zrm}_h(0)=(\langle  z_0, \phi_h\rangle).
\end{align*}
Here,
\begin{align} \label{eqn-mathbb-Sh}
	{\Mrm}_{h} \mathbb{S}_h =  \mathbb{B}_h^T, \text{ and } \mathbb{B}_h \text{ solves } {\Mrm}_{h} \mathbb{B}_h = \Brm_h.
\end{align}

\noindent Let $\wt z_{h_i}$ and $\wt z_{h_{i+1}}$ be the computed solutions at $i-$th and $(i+1)-$th time levels, and let $\wt u_{h_i}$ and $\wt u_{h_{i+1}}$ be the computed stabilizing controls at $i-$th and $(i+1)-$th time levels, respectively, for $i=1,2,\ldots .$ The computed errors in different norms are 
\begin{align*}
	\text{err}_{L^2}(\wt z_{h_i})=\| \wt z_{h_{i+1}} - \wt z_{h_i}\|, \quad \text{err}_{H^1}(\wt z_{h_i})=\| \wt z_{h_{i+1}} - \wt z_{h_i}\|_{H^1(\Omega)} , \text{ and } \text{err}_{L^2}(\wt u_{h_i})=\| \wt u_{h_{i+1}} - \wt u_{h_i}\|.  
\end{align*}
The computational order of convergence $\alpha_{i+1}$ at $i-$th time level is computed using 
\begin{align*}
	\alpha_{i+1} \approx \log(e_{i+1}/e_i)/\log(h_{i+1}/h_i) \text{ for } i=1,2,3,\ldots, 
\end{align*}
where $e_i$ and $h_i$ are error and the discretization parameter at the $i-$th level, respectively.

\medskip
\noindent \textbf{Example 1.} Choose $\Omega=\mathcal{O}=(0,1)\times (0,1),$ $\eta=1,\, \nu_0=0,\,\omega=24,\, \textbf{v}=(1,1),\, y_s(x_1,x_2)=x_1x_2(1-x_1) (1-x_2)$ in \eqref{eqVB-semidisNI}. Also, we choose $y_0(x_1,x_2)=\sin(\pi x_1)\sin(\pi x_2)$ and hence the initial condition is $ z_0(x_1,x_2)= y_0(x_1,x_2)-y_s(x_1,x_2) = \sin(\pi x_1)\sin(\pi x_2) - x_1x_2(1-x_1)(1-x_2) .$

\noindent Initially, we calculate the solution of the system \eqref{eqnVB:discrete} with $\mathfrak{u}_h=0$, and present the corresponding data in Figure \ref{figVB:Ex-1 sol WC}. Figure \ref{figVB:Ex-1 sol WC}(a) reveals the instability of the solution without control by showing a continuous growth in energy over time $t$. Figure \ref{figVB:Ex-1 sol WC}(b) represents the evolution of the solution on log scale. To further analyze the results, we represent the computed errors on a log-log scale against $h$ in Figure \ref{figVB:Ex-1 sol WC}(c).

\begin{figure}[ht!]
	\includegraphics[width=.33\textwidth]{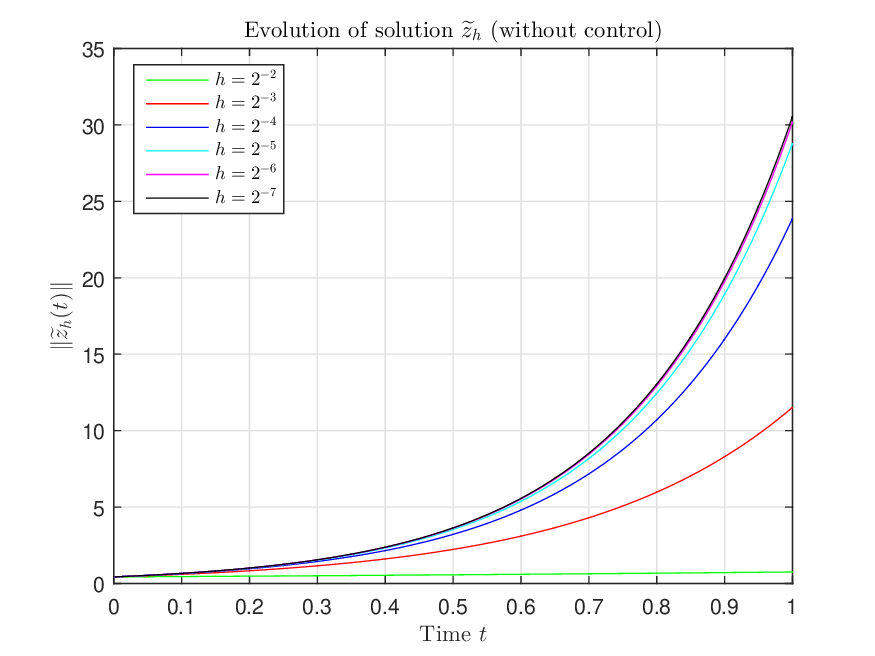}\hfill
	\includegraphics[width=.33\textwidth]{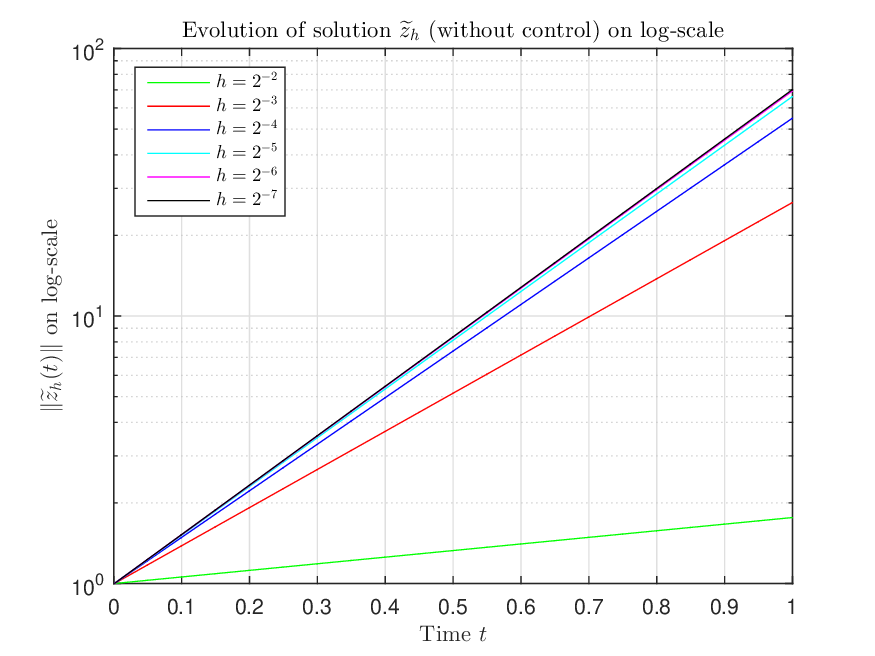}\hfill
	\includegraphics[width=.33\textwidth]{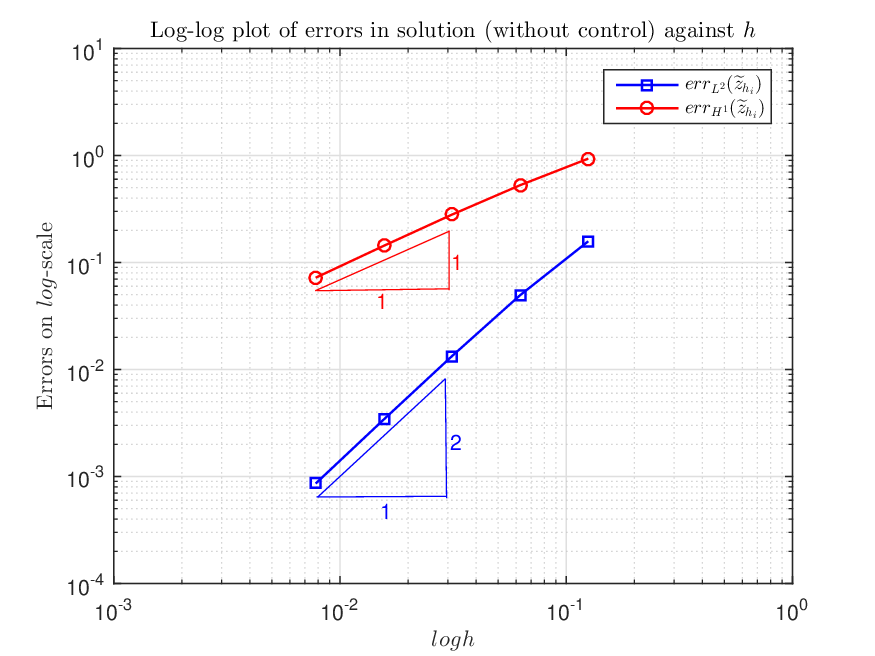}
	\caption{(Example 1.) (a) Evolution of computed solution $y_h$ in $L^2$-norm with time $t,$ (b) on $\log$-scale, and (c) $\log-\log$ plot of errors with discretization parameter $h$.} \label{figVB:Ex-1 sol WC}   
\end{figure}

\medskip
\noindent We compute the feedback control to obtain a stabilized solution and the corresponding data are presented in Figure \ref{figVB:Ex-1 sol Stab}. The evolution of the stabilized solution can be observed in Figure \ref{figVB:Ex-1 sol Stab}(a) as the energy decreases with time. Evolution of stabilized solution on log scale, stabilizing control and errors on log-log scale are plotted in Figure \ref{figVB:Ex-1 sol Stab}(b)-(d). Additionally, the errors and orders of convergence for both the stabilized solution and stabilizing control are summarized in Table \ref{tabVB-Ex1-EO_stab}.

\begin{figure}[ht!]
	\includegraphics[width=.24\textwidth]{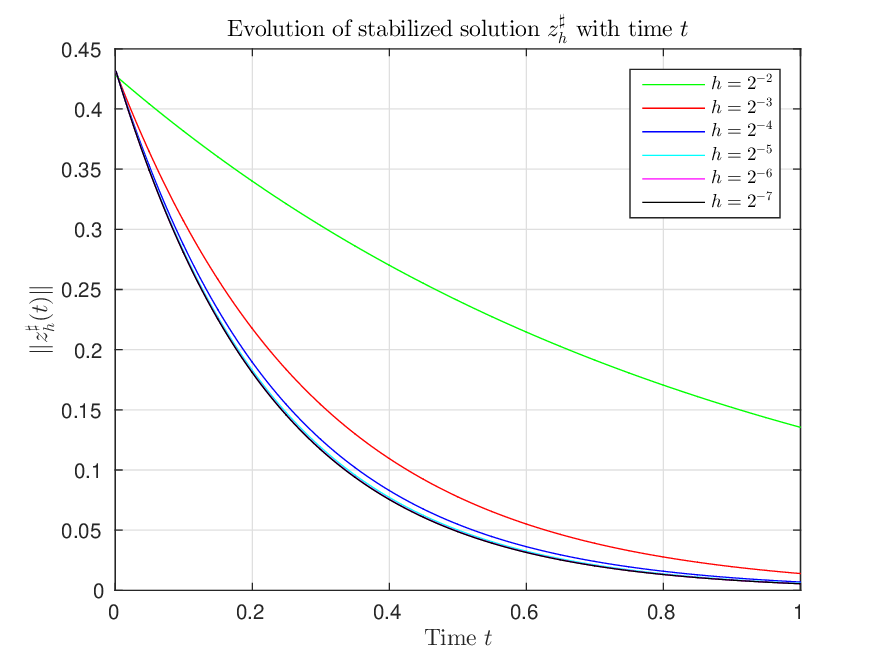}\hfill
	\includegraphics[width=.24\textwidth]{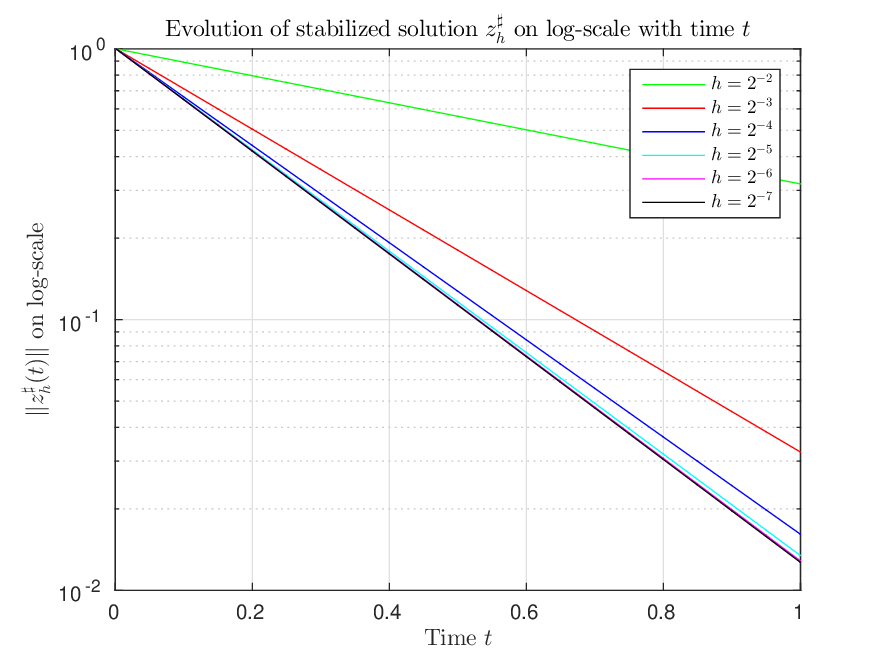}\hfill
	\includegraphics[width=.24\textwidth]{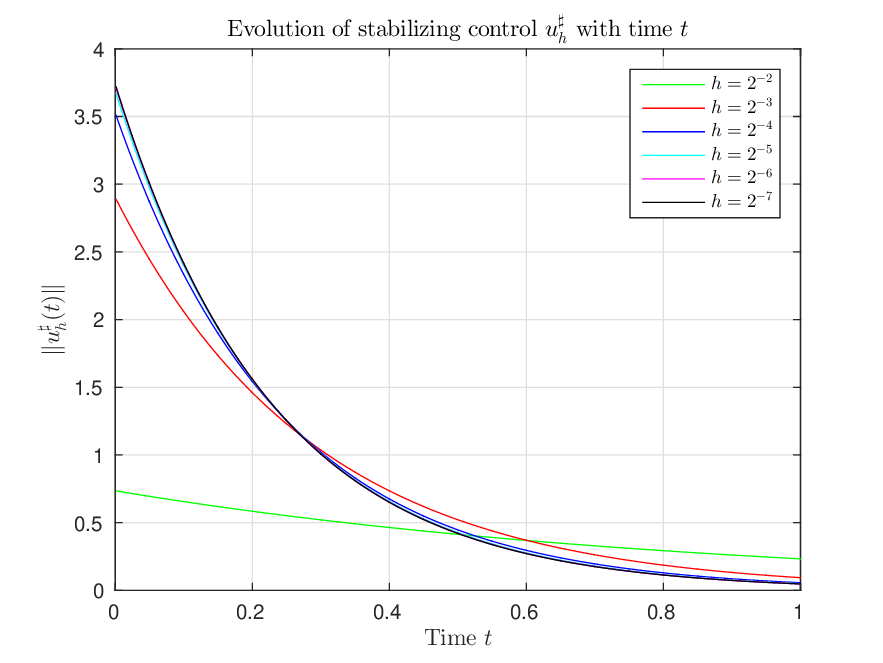}
	\hfill
	\includegraphics[width=.24\textwidth]{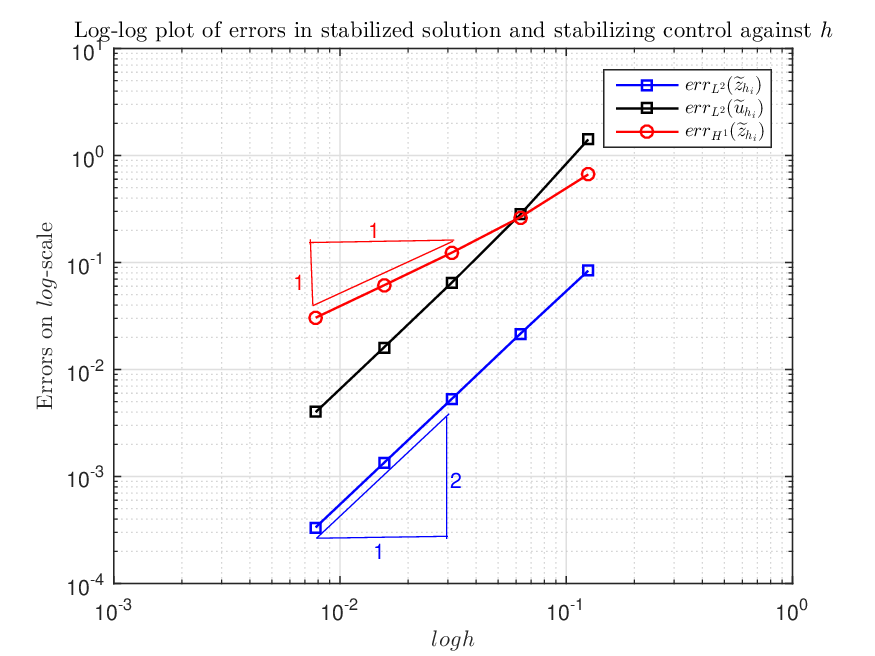}
	\caption{(Example 1.) (a) Evolution of \textbf{stabilized solution}, (b) on $\log$-scale, (c) evolution of stabilizing control and (d) error plot in $\log-\log$ scale. } \label{figVB:Ex-1 sol Stab}      
\end{figure}

\begin{table}[ht!]
	\centering
	\begin{tabular}{|c||c|c||c|c||c|c||}
		\hline
		$h$ & $\text{err}_{L^2}(\wt z_h)$ & Order & $\text{err}_{H^1}(\wt z_h)$ & Order & $\text{err}_{L^2}(\wt u_h)$ & Order\\
		\hline \hline
		$1/2^2$ & 8.368279e-02  & ---      & 6.648702e-01  & --- & 1.408594 & --- \\
		\hline
		$1/2^3$ & 2.145333e-02 &  1.963728    & 2.646716e-01  & 1.328869 & 2.798470e-01 &  2.331545 \\
		\hline
		$1/2^4$ & 5.328601e-03 &   2.009373   & 1.237391e-01  & 1.096901 & 6.506024e-02 &   2.104790 \\
		\hline
		$1/2^5$ & 1.337335e-03 &   1.994395  & 6.113389e-02 & 1.017257 & 1.603811e-02 &   2.020271    \\
		\hline
		$1/2^6$ & 3.351877e-04 &   1.996320  & 3.051795e-02 & 1.002313 & 4.000597e-03 &   2.003217  \\
		\hline
		
	\end{tabular}
	\caption{(Example 1.) Errors and convergence orders of \textbf{stabilized} solution and \textbf{stabilizing} control evaluated at $T=0.1$.} \label{tabVB-Ex1-EO_stab}
\end{table}

\section{Non-linear system}\label{secVB-NLS}

\noindent In this section, we study the stabilization of the non-linear system \eqref{eqVB-mainburger} around steady state $y_s$ with decay $-\omega<0,$ for any given $\omega>0$ by the control obtained in Theorem \ref{thVB-mainstab}. Here we show that \eqref{eqVB-linarnd-ys} is locally stabilizable with exponential decay $-\omega<0,$ for any $\omega>0.$ We establish a stability result for the discrete non-linear closed-loop system (see \eqref{eqVB-disClNlOpForm}). We implement an example to validate the theoretical results of this section, and compute errors and convergence orders.

\medskip
\noindent In Theorem \ref{thVB:dro}, we saw that the discrete linear system is uniformly exponentially stabilizable by feedback control obtained by solving a discrete Riccati equation. Also, the operator $\Ac_{\omega_h, \Pc_h}:=\Ac_{\omega_h}-\Bc_h\Bc_h^*\Pc_h$ generates uniformly analytic semigroup for all $0<h<h_0.$ Thus, from Theorem \ref{thVB:dro} and \cite[Proposition 2.9, Chapter I, Part II]{BDDM}, we have
\begin{align*}
	\Re(\lambda) \le -\omega_P \text{ for all } \lambda \in 	\sigma(\Ac_{\omega_h}-\Bc_h\Bc_h^*\Pc_h) \text{ for all } 0<h<h_0.
\end{align*}

\noindent For the stabilization of non-linear discrete system (see Theorem \ref{thVB-DisMainStabNL}), we assume the following:\\
\textbf{Assumption.}
\begin{itemize}
	\item[$(\mathcal{A}_2)$] $\quad \Re\left(\langle (\Ac_{\omega_h} -\Bc_h \Bc_h^* \Pc_h) \phi_h, \phi_h \rangle\right) \le -\omega_P \|\phi_h\|^2$ for all $\phi_h \in V_h.$
\end{itemize}

\noindent To establish error estimate for the non-linear stabilized system (see Theorem \ref{thVB-mainErrNL}), we assume the following:\\
 \textbf{Assumption.}
\begin{align} \label{eqVB-assumpt-omegaP-alpha}
	(\mathcal{A}_3) \qquad 	\Re\left(\left\langle (\Ac_{\omega_h}-\Bc_h\Bc_h^*\Pc_h )\phi_h, \phi_h \right\rangle\right)  \le -\omega_P \|\phi_h\|^2 -\alpha \|\nabla \phi_h \|^2 \text{ for all }\phi_h \in V_h
\end{align}
holds for some $\alpha>0,$ where $\omega_P$ is as in Theorem \ref{thVB:dro}. Note that Assumption $(\mathcal{A}_3)$ implies $(\mathcal{A}_3).$
\subsection{Stabilizability of non-linear system} \label{secVB-stabNL}
We start with a regularity result. 
Consider the closed loop system
\begin{equation} \label{eqVB-clslp-opninhm}
\begin{aligned}
\wt z'(t)=(\Acw-\Bc\Bc^* \Pc)\wt z(t) + g(\cdot,t) \text{ for all }\,\, t>0, \quad \wt z(0)=z_0, 
\end{aligned}
\end{equation}
where $g\in L^2(0,\infty; \Lt)$ is given. Since, $\{e^{t(\Acw-\Bc\Bc^*\Pc)}\}_{t\ge 0}$ is exponentially stable analytic semigroup on $\Lt$ (see Theorem \ref{thVB-mainstab}), we have the following result.
\begin{Lemma} \label{lemVB:nl-1}
For any $g\in L^2(0,\infty; \Lt)$ and $z_0\in \Lt$, \eqref{eqVB-clslp-opninhm} admits a unique solution $\wt z \in C([0,\infty);\Lt)\cap L^2(0,\infty;\Lt)$ satisfying
\begin{align*}
\|\wt z\|_{L^\infty(0,\infty;\Lt)} + \| \wt z\|_{L^2(0,\infty; \Lt)} \le C \left( \|z_0\| + \|g\|_{ L^2(0,\infty; \Lt)} \right),
\end{align*}
for some constant $C>0.$
\end{Lemma} 

\begin{proof}
Since $\Acw-\Bc\Bc^* \Pc$ is exponentially stable in $\Lt,$ there exists $C>0,$ $\gamma>0$ such that $\|e^{t (\Acw-\Bc\Bc^* \Pc)}z_0\|\le C e^{-\gamma t}\|z_0\|.$ As $g\in L^2(0,\infty; \Lt),$ \cite[Proposition 3.1, Chapter 1, Part II]{BDDM} yields a unique solution $\wt z \in C([0,\infty); \Lt)$ of \eqref{eqVB-clslp-opninhm} satisfying 
\begin{align*}
\wt z(t)= e^{t(\Acw-\Bc\Bc^* \Pc)} z_0+\int_0^t e^{(t-s) (\Acw-\Bc\Bc^* \Pc)} g(\cdot, s) ds \text{ for all } t>0.
\end{align*}
Now, $\|\wt z(t)\| \le C e^{-\gamma t}\|z_0\| + \int_0^t e^{-\gamma (t-s)}\|g(s)\| ds.$ Utilize Young's inequality \eqref{eqPR-YoungIneq} to conclude the lemma.
\end{proof}

\noindent Now, we state the following regularity result for the non-homogeneous system
\begin{equation} \label{eqVB-clslpnonhm}
\begin{aligned}
& \wt z_t(\cdot, t)-\eta \Delta \wt z(\cdot,t) +y_s\textbf{v}\cdot\nabla \wt z(\cdot,t)+\textbf{v}\cdot\nabla y_s \wt z(\cdot,t)\\
& \hspace{5cm}+(\nu_0-\omega )\wt z(\cdot, t)=-\Bc \Bc^*\Pc \wt z(\cdot,t)+g(\cdot,t) \text{ in } Q,\\
& \wt z(x,t)=0 \text{ on }\Sigma, \quad \wt z(x,0)=z_0(x) \text{ in }\Omega.
\end{aligned}
\end{equation}

\begin{Proposition}\label{ppsVB-nonhm-2}
Let $\omega>0$ be any given real number and $g\in L^2(0,\infty;\Lt).$ Then for any $z_0\in \Hio,$ the closed loop system \eqref{eqVB-clslpnonhm} has a unique solution $\wt z \in L^2(0,\infty; H^2(\Omega)\cap \Hio) \cap H^1(0,\infty;\Lt) \cap C_b([0,\infty);\Hio)$ satisfying
\begin{align*}
\|\wt z\|_{L^2(0,\infty; H^2(\Omega))}+ \|\wt z\|_{H^1(0,\infty;\Lt)}+\| \wt z\|_{L^\infty (0,\infty; \Hio) } \le C_1 \left( \|z_0\|_{\Hio} +\|g\|_{L^2(0,\infty;\Lt)}\right),
\end{align*}
for some positive constant $C_1.$
\end{Proposition}
\noindent The proof of the proposition follows from \cite[Propositions 3.3 and 3.13, Chapter 1, Part II]{BDDM}.

\medskip
\noindent Next, we study the stability of the following non-linear equation:
\begin{equation} \label{eqVB-closedloopshifted}
	\begin{aligned}
	& \wt z_t(\cdot, t)-\eta \Delta \wt z(\cdot,t) +y_s\textbf{v}\cdot\nabla \wt z(\cdot,t)+\textbf{v}\cdot\nabla y_s \wt z(\cdot,t)+(\nu_0-\omega )\wt z(\cdot, t) \\
	& \hspace{5.5cm}+e^{-\omega t}\wt z(\cdot,t)\textbf{v}\cdot \nabla \wt z(\cdot,t)=-\Bc \Bc^*\Pc \wt z(\cdot,t) \text{ in } Q,\\
	& \wt z(x,t)=0 \text{ on }\Sigma, \quad \wt z(x,0)=z_0(x) \text{ in }\Omega,
	\end{aligned}
\end{equation}
where $\Pc\in \Lc(\Lt)$ is solution of \eqref{eqVB-Riccati}. We prove the stability of the above equation using the Banach fixed point theorem. Define
\begin{equation}\label{eqVB:space1}
	\begin{array}{l}
		D=\Big\lbrace \widetilde{z}\in L^2(0,\infty; (H^1_0(\Omega)\cap H^2(\Omega))\cap C([0,\infty); H^1_0(\Omega))\cap H^1(0,\infty;L^2(\Omega)), \text{ equipped} \\[2.mm]
		\text{with the norm } 
		\|\widetilde{z}\|^2_D= \|\widetilde{z}\|^2_{L^2(0,\infty; H^2(\Omega))}+\|\widetilde{z}\|^2_{L^\infty(0,\infty; H^1_0(\Omega))}
		+\|\widetilde{z}\|^2_{H^1(0,\infty;L^2(\Omega))} \Big \rbrace.
	\end{array}
\end{equation}
and for any $\rho>0$, set
\begin{equation}\label{eqVB:space2}
	\begin{array}{l}
		D_\rho=\{\widetilde{z}\in D\mid \|\widetilde{z}\|_D \le \rho\}.
	\end{array}
\end{equation}

\noindent To proceed further, we need following Lemma and the proof can be established mimicking the proof of \cite[Lemma 7.1]{WKR}. 
\begin{Lemma} \label{lemVB:liptzest}
	Let $\omega>0$ be any given real number and let 
	\begin{equation} \label{eqVB:expNonhomTerm-g}
		g(\psi)(x,t)=-e^{-\omega t} \psi(x,t) \textbf{v}\cdot \nabla \psi(x,t) \text{ for all }(x,t)\in Q.
	\end{equation}
Then for any $\psi^1, \psi^2\in D,$ we have
\begin{equation}
	\begin{aligned}
	&	\|g(\psi^1)\|_{L^2(0,\infty;\Lt)} \le C_2 \|\psi^1\|_D^2, \\
	&  \| g(\psi^1) - g(\psi^2)\|_{L^2(0,\infty;\Lt)} \le C_2 (\|\psi^1\|_D + \|\psi^2\|_D)\|\psi^1 - \psi^2\|_D,
	\end{aligned}
\end{equation}
for some positive $C_2.$
\end{Lemma}
\noindent The constant $C_2$ in the above lemma depends on  $\textbf{v},\Omega$ and embedding constant $s_0$ for $H^1(\Omega)\hookrightarrow L^4(\Omega)$ as stated in Lemma \ref{lemPR:SobEmb}. 

\begin{Proposition} \label{ppsVB-stabNL-H10}
	For any given real number $\omega>0$, there exist $\rho_0>0$ and $M>0$ such that for all $\rho\in (0,\rho_0]$ and for all $z_0\in \Hio$ with $\|z_0\|_{\Hio}\le M \rho,$ the closed loop system \eqref{eqVB-closedloopshifted} admits a unique solution $\wt z\in D_\rho$ satisfying 
	\begin{equation} \label{eqVB-LinfH10est}
		\|\wt z(\cdot,t)\|_{\Hio} \le M_1 \|z_0\|_{\Hio} \text{ for all }t>0,
	\end{equation}
for some positive constant $M_1$ independent of $z_0$ and $t$. 
\end{Proposition}

\begin{proof}
	We prove the proposition by using Banach fixed point theorem. 
	
	\noindent  For a given $\psi \in D_\rho$ and $z_0\in \Hio,$ let $\wt z^\psi$ satisfies 
\begin{equation} \label{eqVB-closedloop with non-hom gen}
	\begin{aligned}
		& \wt z^\psi_t(\cdot, t)-\eta \Delta \wt z^\psi(\cdot,t) +y_s\textbf{v}\cdot\nabla \wt z^\psi(\cdot,t)+\textbf{v}\cdot\nabla y_s \wt z^\psi(\cdot,t)+(\nu_0-\omega )\wt z^\psi(\cdot, t) \\
		& \hspace{7.5cm}=-\Bc \Bc^*\Pc \wt z^\psi(\cdot,t)+g(\psi)(\cdot,t) \text{ in } Q,\\
		& \wt z^\psi(x,t)=0 \text{ on }\Sigma, \quad \wt z^\psi(x,0)=z_0(x) \text{ in }\Omega,
	\end{aligned}
\end{equation}
where $g$ is as defined in \eqref{eqVB:expNonhomTerm-g}. Define the mapping 
\begin{align*}
	S : D_\rho \rightarrow D_\rho
\end{align*}
by $S(\psi)=\wt z^\psi,$ where $\wt z^\psi$ satisfies \eqref{eqVB-closedloop with non-hom gen} for a given $\psi \in D_\rho.$ Aim is to show that $S$ is self-map and contraction.

\noindent Under some smallness assumption on initial data $z_0$, we want to show that $\wt z^\psi \in D_\rho.$ Due to $\psi\in D_\rho$ and Lemma  \ref{lemVB:liptzest}, $g(\psi)\in L^2(0,\infty;\Lt).$ Hence Proposition \ref{ppsVB-nonhm-2} and Lemma  \ref{lemVB:liptzest} lead to $\wt z^\psi\in D$ with 
\begin{equation} \label{eqVB-nonlinExisEst}
\begin{aligned}
\|\wt z^\psi\|_{L^2(0,\infty; H^2(\Omega))}+ \|\wt z^\psi\|_{H^1(0,\infty;\Lt)}+\| \wt z^\psi\|_{L^\infty (0,\infty; \Hio) } & \le C_1 \left( \|z_0\|_{\Hio} +\|g\|_{L^2(0,\infty;\Lt)}\right),\\
& \le  C _1\|z_0\|_{\Hio}+ C_1C_2 \|\psi\|_D^2.
\end{aligned}
\end{equation}
Now, choosing 
\begin{equation} \label{eqVB: choice rho I}
\|z_0\|_{\Hio} \le \frac{\rho}{3C_1} \text{ and } \rho \le \frac{1}{3C_1C_2},
\end{equation}
we have 
\begin{align*}
	\|\wt z^\psi\|_D \le \frac{\rho}{3}+C_1C_2 \rho^2 \le \frac{\rho}{3}+\frac{\rho}{3}\le \rho
\end{align*}
and therefore $\wt z^\psi \in D_\rho.$ Thus $S$ is a self map for all $0<\rho \le \frac{1}{3C_1C_2}.$

\noindent Now, to prove $S$ is contraction, let $\psi^1, \psi^2\in D_\rho$ with $\wt z^{\psi^1}$ and $\wt z^{\psi^2}$ as the corresponding solution of \eqref{eqVB-closedloop with non-hom gen}. Set $\mathfrak{Z}:=\wt z^{\psi^1}-\wt z^{\psi^2},$ then $\mathfrak{Z}$ satisfies 
\begin{align*}
	& \mathfrak{Z}_t(\cdot, t)-\eta \Delta \mathfrak{Z}(\cdot,t) +y_s\textbf{v}\cdot\nabla \mathfrak{Z}(\cdot,t)+\textbf{v}\cdot\nabla y_s \mathfrak{Z}(\cdot,t)+(\nu_0-\omega )\mathfrak{Z}(\cdot, t) \\
& \hspace{6.5cm}=-\Bc \Bc^*\Pc \mathfrak{Z}(\cdot,t)+g(\psi^1)(\cdot,t) - g(\psi^2)(\cdot,t) \text{ in } Q,\\
& \mathfrak{Z}(x,t)=0 \text{ on }\Sigma, \quad \mathfrak{Z}(x,0)=0 \text{ in }\Omega.
\end{align*}
Proposition \ref{ppsVB-nonhm-2} followed by Lemma \ref{lemVB:liptzest} implies 
\begin{align*}
	\|\mathfrak{Z}\|_{D} \le C_1 \| g(\psi^1)-g(\psi^2)\|_{L^2(0,\infty;\Lt)} & \le C_1C_2 (\|\psi^1\|_D+\|\psi^2\|_D)\|\psi^1-\psi^2\|_D \\
	& \le 2C_1C_2 \rho \|\psi^1-\psi^2\|_D.
\end{align*}
Now, for the choice of $\rho$ as in \eqref{eqVB: choice rho I} imply 
\begin{align*}
	\|\mathfrak{Z}\|_D=\|\wt z^{\psi^1} - \wt z^{\psi^2}\|_D \le \frac{2}{3}\|\psi^1-\psi^2\|_D.
\end{align*}
Now, set $M=\frac{1}{3C_1}$ and $\rho_0=\frac{1}{3C_1C_2},$ where $C_1$ is as in Proposition \ref{ppsVB-nonhm-2} and $C_2$ as in Lemma \ref{lemVB:liptzest}. Then for all $0<\rho \le \rho_0$ and $z_0\in \Hio$ with $\|z_0\|\le M\rho,$ the map $S:D_\rho \rightarrow D_\rho$ is self and contraction map. And, hence using a Banach fixed point theorem $S$ has a fixed point $\wt z\in D_\rho$ and thus \eqref{eqVB-closedloopshifted} admits a unique solution $\wt z \in D_\rho.$\\
Now, choose $\psi=\wt z$ in \eqref{eqVB-nonlinExisEst} and use the fact that $\wt z\in D_\rho,$ to obtain
\begin{align*}
	\frac{2}{3} \|\wt z (\cdot, t)\|_{\Hio} \le C_1 \|z_0\|_{\Hio} \text{ for all } t>0,
\end{align*}
and hence \eqref{eqVB-LinfH10est} holds.
\end{proof}

\medskip
\noindent  Setting $z(x,t)=e^{-\omega t}\wt z(x,t)$ for any $\omega>0,$ we conclude the following main result on stabilizability of the non-linear system:
\begin{equation} \label{eqVB-NLClsdLoop}
	\begin{aligned}
		& z_t-\eta \Delta z+z\textbf{v}\cdot \nabla z+y_s \textbf{v} \cdot \nabla z+\textbf{v}\cdot \nabla y_s z +\nu_0 z = -\chi_{\mathcal{O}}\Bc^*\Pc z \text{ in } Q, \\
		& z=0 \text{ on }\Sigma, \quad z(x,0)=z_0(x) \text{ in }\Omega.
	\end{aligned}
\end{equation}

\begin{Theorem} \label{thVBMain-NonLin}
	Let $\omega>0$ be any given real number. Then there exist positive constants $\rho_0,M$ depending on $\omega,\eta,\Omega$ such that for all $0<\rho\le \rho_0$ and for all $z_0\in \Hio$ with $\|z_0\|_{\Hio}\le M\rho,$ the non-linear closed loop system \eqref{eqVB-NLClsdLoop} admits a unique solution $z\in L^2(0,\infty;\Hio\cap H^2(\Omega))\cap H^1(0,\infty;\Lt) \cap C_b([0,\infty);\Hio)$ satisfying
	\begin{align*}
		\|e^{\omega(\cdot)}z\|^2_{L^2(0,\infty;H^2(\Omega))} + \|e^{\omega(\cdot)}z\|^2_{L^\infty(0,\infty;H^1_0(\Omega))} + \|e^{\omega(\cdot)}z\|^2_{H^1(0,\infty;L^2(\Omega))} \le \rho^2.
	\end{align*}
	Moreover, $\|z(\cdot,t)\|_{\Hio} \le M_1 e^{-\omega t}\|z_0\|_{\Hio}$ for all $t>0,$ and for some $M_1>0$ independent of $z_0$ and time $t.$
\end{Theorem}

\noindent Next, we establish an result for the stabilization of the non-linear discrete system. In particular, we show the stability of the semi-discrete closed loop system 
\begin{equation} \label{eqVB-disClNlOpForm}
	\begin{aligned}
		\wt z_h'(t) = (\Ac_{\omega_h}-\Bc_h\Bc_h^*\Pc_h) \wt z_h(t) -e^{-\omega t}z_h \textbf{v}\cdot \nabla z_h\text{ for all }t>0, \quad \wt z_h(0)=\pi_h z_0,
	\end{aligned}
\end{equation}
where $\Pc_h$ is solution of \eqref{eqVB-d-ARE}.

\begin{Theorem}[uniform stability for non-linear dicrete closed loop system] \label{thVB-DisMainStabNL}
	Let  $(\mathcal{A}_2)$ hold and $\omega>0$ be any given real number. Let $\omega_P>0$ and $h_0$ be as in Theorem \ref{thVB:dro}, and $z_0\in \Lt.$ Then the solution $\wt z_h$ of \eqref{eqVB-disClNlOpForm} satisfies
	\begin{align*}
		\|\wt z_h(t)\| \le C e^{-\omega_P t}\|z_0\| \text{ for all } t\ge 0, \text{ for all } 0<h<h_0,
	\end{align*}
	and for some $C>0$ independent of $h$
\end{Theorem} 

\begin{proof}
The weak formulation of \eqref{eqVB-disClNlOpForm} seeks $\wt z_h \in V_h$ such that 
{\small
\begin{equation} \label{eqVB-err cls dwk}
	\begin{aligned}
		& \langle \wt z_{h_t},\phi_h\rangle +\eta \langle\nabla \wt z_h ,\nabla\phi_h\rangle + e^{-\omega t} \langle \wt z_h\textbf{v}\cdot \nabla \wt z_h , \phi_h\rangle + \langle y_s \textbf{v} \cdot \nabla \wt z_h+\textbf{v}\cdot \nabla y_s \wt z_h, \phi_h\rangle +(\nu_0-\omega)\langle \wt z_h,\phi_h\rangle \\ 
		& \quad \qquad\qquad = \langle -\Bc_h\Bc_h^*\Pc_h\wt z_h, \phi_h\rangle, \\
		& \langle \wt z_h(0),\phi_h\rangle = \langle \pi_h z_0,\phi_h\rangle \text{ for all }\phi_h\in V_h.
	\end{aligned}
\end{equation}
}The definition of $\Ac_h$ in \eqref{eqVB-DisBil} and $\Ac_{\omega_h}$ in \eqref{eqdefVB-Awh} lead to 
\begin{equation*}
	\begin{aligned}
		\langle \wt z_{h_t},\phi_h\rangle - \langle \Ac_{\omega_h} \wt z_h ,\phi_h \rangle + e^{-\omega t} \langle \wt z_h\textbf{v}\cdot \nabla \wt z_h , \phi_h\rangle = \langle -\Bc_h\Bc_h^*\Pc_h\wt z_h, \phi_h\rangle, \text{ for all } \phi_h \in V_h,
	\end{aligned}
\end{equation*}
that is, 
\begin{equation} \label{eqVB-unistabNLOP}
	\begin{aligned}
		\langle \wt z_{h_t},\phi_h\rangle + e^{-\omega t} \langle \wt z_h\textbf{v}\cdot \nabla \wt z_h , \phi_h\rangle = \langle (\Ac_{\omega_h} -\Bc_h\Bc_h^*\Pc_h)\wt z_h, \phi_h\rangle, \text{ for all } \phi_h \in V_h.
	\end{aligned}
\end{equation}
 Utilizing these and $(\mathcal{A}_2)$ after substituting $\phi_h=\wt z_h$ in \eqref{eqVB-unistabNLOP}, we have
\begin{align*}
	\frac{1}{2}\frac{d}{dt}\|\wt z_h(t)\|^2 \le -\omega_P \|\wt z_h(t)\|^2 \text{ for all }t >0 \text{ and for all }0<h<h_0.
\end{align*}
Now, integrating the above inequality with the fact that $\|\wt z_h(0)\|=\|\pi_h z_0\|\le \|z_0\|,$ we obtain
\begin{align*}
	\|\wt z_h(t)\| \le  e^{-\omega_P t}\|z_0\| \text{ for all }t>0\text{ and for all }0<h<h_0.
\end{align*}
The proof is complete.
\end{proof}

\subsection{Error estimates} \label{secVB-ErrEst-NL}
In this subsection, we derive an error estimate for the stabilized non-linear system.

\subsubsection{Auxiliary results}
Here, we establish certain auxiliary results necessary for obtaining the error estimates in the subsequent subsection. Let $\Pc \in \Lc(\Lt)$ denote the solution of \eqref{eqVB-Riccati}, as specified in Theorem \ref{thVB-mainstab}. Consider the following non-homogeneous equation 
\begin{equation} \label{eqVB-nonhomCl}
	\begin{aligned}
		& \wt z'(t) =(\Ac_{\omega} - \Bc \Bc^*\Pc)\wt z(t) + g(t) \text{ for all }t>0, \\\
		& \wt z(0)=z_0,
	\end{aligned}
\end{equation} 
where $g(\cdot)$ is given in some suitable space.

\begin{Lemma} \label{lemVB-ImpRegStab-1}
	Let $z_0 \in D(\Ac)$ and $g \in L^2(0,\infty; D(\Ac)).$ Then \eqref{eqVB-nonhomCl} has a unique solution $\wt z \in H^1(0,\infty; \Lt)\cap C([0,\infty); D(\Ac)).$ 
\end{Lemma}
\begin{proof}
	The proof follows from \cite[Proposition 3.3, Chapter 1, part II]{BDDM} using with the fact that $\Ac_{\omega}-\Bc\Bc^*\Pc$ generates an analytic semigroup $\{e^{t(\Ac_{\omega}-\Bc\Bc^*\Pc)}\}_{t\ge 0}$ on $\Lt$ of negative type.  
\end{proof}

\begin{Lemma} \label{lemVB-ImpRegStab-2}
	Let $(\Ac, D(\Ac))$ be as defined in \eqref{eqdefVB-A}. Let $g \in L^2(0,\infty; D(\Ac))\cap H^1(0,\infty; \Lt)$ and $z_0 \in H^3(\Omega)\cap \Hio.$  Also, assume the compatibility condition $g(0) + \Ac z_0 \in \Hio.$ Then the solution $\wt z $ of \eqref{eqVB-nonhomCl} belongs to $L^2(0,\infty; H^4(\Omega)\cap \Hio) \cap C([0,\infty);H^2(\Omega)\cap \Hio) \cap H^1(0,\infty; H^2(\Omega)\cap \Hio)$ satisfying
	\begin{align*}
		\|\wt z\|_{L^2(0,\infty; H^4(\Omega))} +\|\wt z\|_{L^\infty(0,\infty; H^2(\Omega))} +\|\wt z\|_{H^1(0,\infty; H^2(\Omega))} \le M_1 \big( \|z_0\|_{H^3(\Omega)} & +\|g\|_{L^2(0,\infty; H^2(\Omega))} \\
		& +\|g'\|_{L^2(0,\infty; L^2(\Omega))} \big),
	\end{align*}
	for some $M_1=M_1(\Omega)>0.$
\end{Lemma}
\begin{proof}
	We re-write \eqref{eqVB-nonhomCl} as
	\begin{align*}
		\wt z'(t) - \Ac \wt z=\omega \wt z(t) - \Bc\Bc^*\Pc \wt z(t) + g(t)=:F(t) \text{ for all } t>0 , \quad \wt z(0)=z_0,
	\end{align*}
	where $F(t)=\omega \wt z(t) - \Bc\Bc^*\Pc \wt z(t) + g(t).$ From Lemma \ref{lemVB-ImpRegStab-1}, note that $\wt z \in H^1(0,\infty; \Lt)\cap C([0,\infty); D(\Ac)).$ Also, we have $\Pc \wt z\in D(\Ac^*)=D(\Ac) = H^2(\Omega)\cap \Hio.$ Therefore, we have $F\in L^2(0,\infty; D(\Ac))$ and $F'(t)=\omega \wt z'(t)-\Bc \Bc^*\Pc \wt z'(t)+g'(t) \in L^2(0,\infty;\Lt).$ Now, utilizing \cite[Theorems 5 and 6, Chapter 7]{Eva} with $m=1,k=1$ and the fact that $(\Ac, D(\Ac))$ generates an analytic semigroup of negative type, we conclude the proof.
\end{proof}

\noindent To proceed further, we first define 
{\small
	\begin{align*}
		\Dc=\left\lbrace \wt z\, |\, \wt z\in L^2(0,\infty; H^4(\Omega)\cap \Hio)  \cap L^\infty(0,\infty; H^2(\Omega)\cap \Hio) \cap H^1(0,\infty;  H^2(\Omega)\cap \Hio)\right\rbrace,
	\end{align*}
}equipped with the norm $\|\wt z\|^2_D:= \|\wt z\|_{L^2(0,\infty; H^4(\Omega))}^2 +\|\wt z\|_{L^\infty(0,\infty; H^2(\Omega))}^2 +\|\wt z\|^2_{ H^1(0,\infty;  H^2(\Omega))}.$ Also, for any $\rho>0,$ define
\begin{equation*}
	\begin{aligned}
		\Dc_\rho:=\{\wt z \in \Dc \, |\, \|\wt z\|_{\Dc} \le \rho\}.
	\end{aligned}
\end{equation*}

\begin{Lemma} \label{lemVB-ImpRegStab}
	For any $\psi \in L^2(0,\infty; H^4(\Omega)\cap \Hio) \cap L^\infty(0,\infty; H^2(\Omega)\cap \Hio) \cap H^1(0,\infty;H^2(\Omega)$ $\cap\Hio)$ and for any $\omega>0,$  let
	\begin{align*}
		g(\psi)= e^{-\omega t}\psi \textbf{v}\cdot \nabla \psi. 
		\textbf{}	\end{align*}
	Then for any $\psi^1, \psi^2 \in \Dc,$ \\
	$(a)$ $g(\psi^1) \in L^2(0,\infty; D(\Ac)) \cap H^1(0,\infty; \Lt)$ and
	$$\|g(\psi^1)\|_{L^2(0,\infty; H^2(\Omega))} +\left\|\frac{\partial }{\partial t}g(\psi^1)\right\|_{L^2(0,\infty; \Lt)} \le M_2 \|\psi\|_{\Dc}^2,$$ \\
	$(b)$ $\left\| g(\psi^1) -g(\psi^2)\right\|_{L^2(0,\infty;H^2(\Omega))} + \left\| \frac{\partial }{\partial t} (g(\psi^1) - g(\psi^2))\right\|_{L^2(0,\infty;\Lt)} \le M_2 \left( \|\psi^1\|_\Dc +\|\psi^2\|_\Dc \right) \|\psi^1 - \psi^2\|_\Dc,$ for some $M_2>0.$
\end{Lemma}
\begin{proof}
	From Lemma \ref{lemVB:liptzest}, $g(\psi^1)\in L^2(0,\infty;\Lt).$ 
	Next, we show $\Delta g(\psi^1)\in L^2(0,\infty; \Lt).$ Note that 
	\begin{align*}
		\Delta g(\psi^1)  = \nabla \cdot \nabla \left( e^{-\omega t} \psi^1 \textbf{v}\cdot \nabla \psi^1 \right)  = e^{-\omega t} \left( \Delta \psi^1 \textbf{v}\cdot \nabla \psi^1 +\psi^1 \Delta (\textbf{v}\cdot \nabla \psi^1) +2 \nabla \psi^1 \cdot \nabla (\textbf{v}\cdot \nabla \psi^1) \right).
	\end{align*}
	Therefore, a use of Lemmas \ref{lemPR:SobEmb} and \ref{lemVB:AgmonIE} lead to
	\begin{align*}
		\| \Delta g(\psi^1) \|_{L^2(0,\infty; \Lt)}^2 & \le \int_0^\infty \|\Delta \psi^1 \textbf{v}\cdot \nabla \psi^1\|^2 \, dt + \int_0^\infty \| \psi^1 \Delta (\textbf{v}\cdot \nabla \psi^1)\|^2 \, dt \\
		& \qquad + 2 \int_0^\infty \| \nabla \psi^1 \cdot \nabla (\textbf{v}\cdot \nabla \psi^1) \|^2 \, dt \\
		& \le s_0^4 |\textbf{v}|^2 \int_0^\infty \|\psi^1\|_{H^3(\Omega)}^2 \|\psi^1\|_{H^2(\Omega)}^2 \, dt + |\textbf{v}|^2 \int_0^\infty \|\psi^1\|_{L^\infty(\Omega)}^2 \|  \psi^1 \|_{H^3(\Omega)}^2 \, dt \\
		& \qquad + 2 s_0^4 |\textbf{v}|^2 \int_0^\infty \|\psi^1\|_{H^3(\Omega)}^2 \|\psi^1\|_{H^2(\Omega)}^2 \, dt \\
		& \le C \|\psi^1\|_{L^\infty(0,\infty; H^2(\Omega))}^2 \|\psi^1\|^2_{L^2(0,\infty; H^3(\Omega))},
	\end{align*}
	for some $C=C(C_a,s_0, \textbf{v})>0.$ Furthermore, using a Young's inequality, we can write 
	\begin{align} \label{eqVB-ImpRegNonTermEst-2}
		\| \Delta g(\psi^1) \|_{L^2(0,\infty; \Lt)} \le C \left( \|\psi^1\|_{L^\infty(0,\infty; H^2(\Omega))}^2 +  \|\psi^1\|^2_{L^2(0,\infty; H^3(\Omega))}\right).
	\end{align}
	Note that $g(\psi^1)\in H^2(\Omega)\cap \Hio,$ therefore, $\|g(\psi^1)\|_{H^2(\Omega)} \cong \|\Delta g(\psi^1)\|$ (see \cite[problem 2.17]{Kes}). Therefore, $g(\psi^1) \in L^2(0,\infty;D(\Ac)).$
	
	\noindent Note that $\frac{\partial}{\partial t} g(\psi^1) = e^{-\omega t} \left(-\omega  \psi^1 \textbf{v}\cdot \nabla \psi^1 + \psi^1_t \textbf{v}\cdot \nabla \psi^1 + \psi^1 \textbf{v}\cdot \nabla \psi^1_t\right) \in L^2(0,\infty;\Lt).$  Therefore,
	\begin{align*}
		\left\|\frac{\partial}{\partial t} g(\psi^1) \right\|^2_{L^2(0,\infty;\Lt)} & \le  \omega^2\int_0^\infty \|\psi^1 \textbf{v}\cdot \nabla \psi^1\|^2 \, dt  \int_0^\infty \|\psi^1_t \textbf{v}\cdot \nabla \psi^1 \|^2\,dt + \int_0^\infty \|\psi^1 \textbf{v}\cdot \nabla \psi^1_t\|^2 \, dt \\
		& \le |\textbf{v}|^2 \omega^2 \int_0^\infty \|\psi^1\|^2_{L^\infty(\Omega)} \|\psi^1\|_{H^1(\Omega)}^2 dt +  |\textbf{v}|^2\int_0^\infty \|\psi^1_t\|_{L^4(\Omega)}^2 \|\nabla \psi^1\|_{L^4(\Omega)}^2 \, dt \\
		& \qquad  + |\textbf{v}|^2\int_0^\infty \|\psi^1\|_{L^\infty(\Omega)}^2 \|\nabla \psi^1_t\|^2 \, dt.
	\end{align*}
	Now, use Sobolev embedding and Agmon's inequality (in Lemma \ref{lemVB:AgmonIE}) to obtain
	\begin{align*}
		\left\|\frac{\partial}{\partial t} g(\psi^1) \right\|^2_{L^2(0,\infty;\Lt)} &    \le |\textbf{v}|^2 \omega^2 C_a^2 \int_0^\infty \|\psi^1\|^3_{H^1(\Omega)} \|\psi^1\|_{H^2(\Omega)} dt+ |\textbf{v}|s_0^4\int_0^\infty \| \psi^1_t\|_{H^1(\Omega)}^2 \| \psi^1\|_{H^2(\Omega)}^2 \, dt \\
		& \qquad \quad  + |\textbf{v}|C_a^2\int_0^\infty \|\psi^1\|_{H^1(\Omega)} \|\psi^1\|_{H^2(\Omega)}\| \psi^1_t\|_{H^1(\Omega)}^2 \, dt \\
		& \le C \|\psi^1\|^2_{L^\infty(0,\infty; H^2(\Omega))} \|\psi^1\|^2_{H^1(0,\infty;\Hio)},
	\end{align*}
	for some $C=C(\textbf{v}, s_0, \omega, C_a)>0.$ Therefore, using a Young's inequality, we obtain
	\begin{align} \label{eqVB-ImpRegNonTermEst-3}
		\left\|\frac{\partial}{\partial t} g(\psi^1) \right\|_{L^2(0,\infty;\Lt)} \le C \left( \|\psi^1\|^2_{L^\infty(0,\infty; H^2(\Omega))} + \|\psi^1\|^2_{H^1(0,\infty;\Hio)} \right)
	\end{align}
	Using \eqref{eqVB-ImpRegNonTermEst-2} and \eqref{eqVB-ImpRegNonTermEst-3}, we conclude (a). 
	
	\medskip
	\noindent (b) Let $\psi^1, \psi^2 \in D.$ Observe that
	{\small
		\begin{align*}
			\Delta \left( g(\psi^1) - g(\psi^2) \right) & = e^{-\omega t} \Big( \Delta (\psi^1 -\psi^2) \textbf{v}\cdot \nabla \psi^1 +\Delta \psi^2 \textbf{v}\cdot \nabla (\psi^1 -\psi^2) + 2 \nabla (\psi^1 -\psi^2) \cdot \nabla \left(\textbf{v}\cdot \nabla \psi^1 \right) \\
			& \quad   + 2 \nabla \psi^2 \cdot \nabla  \left( \textbf{v}\cdot \nabla (\psi^1 - \psi^2) \right) + (\psi^1 -\psi^2) \Delta (\textbf{v}\cdot \nabla \psi^1) +\psi^2 \Delta \left( \textbf{v}\cdot \nabla(\psi^1-\psi^2) \right)\Big)
		\end{align*}
	}and 
	{\small
		\begin{align*}
			\frac{\partial}{\partial t} \left( g(\psi^1) -  g(\psi^2) \right) &  = e^{-\omega t} \Big( -\omega \left( (\psi^1-\psi^2)\textbf{v}\cdot \nabla \psi^1 + \psi^2 \textbf{v}\cdot \nabla(\psi^1-\psi^2)\right)  + (\psi^1-\psi^2)_t \textbf{v}\cdot \nabla \psi^1  \\
			& \qquad + \psi^2_t \textbf{v}\cdot \nabla(\psi^1-\psi^2) + (\psi^1-\psi^2)\textbf{v}\cdot \nabla \psi^1_t + \psi^2 \textbf{v}\cdot \nabla (\psi^1-\psi^2)_t \Big).
		\end{align*}
	}Therefore, proceeding as in (a), we obtain (b).
\end{proof}

\noindent Now, consider the closed loop system 
\begin{align} \label{eqVB-closedloopImprRegStab}
	\wt z'(t) =(\Ac_{\omega}-\Bc\Bc^*\Pc)\wt z(t) - e^{-\omega t} \wt z \textbf{v}\cdot \nabla \wt z \, \text{ for all } t>0, \quad \wt z(0) =z_0.
\end{align}

\begin{Theorem} \label{thVB-existImReg}
	There exist $\rho_0,M>0$ such that for all $0<\rho\le \rho_0$ and for all $z_0 \in H^3(\Omega)\cap \Hio$ with $\Ac z_0 \in \Hio$ and $\|z_0\|_{H^3(\Omega)} \le M \rho,$ the closed loop system \eqref{eqVB-closedloopImprRegStab} admits a unique solution $\wt z\in \Dc_\rho.$ 
\end{Theorem}

\begin{proof}
	We proof the theorem by using Banach fixed point theorem. For a given $\psi \in \Dc_\rho$ and $z_0\in H^3(\Omega)\cap \Hio,$ let $\wt z^\psi$ satisfy 
	\begin{align*}
		{\wt z^\psi}{'}(t)=(\Ac_{\omega}-\Bc\Bc^*\Pc)\wt z^\psi(t) + g(\psi) \text{ for all }t>0, \quad \wt z^\psi (0)=z_0.
	\end{align*}
	We prove that $\wt z^\psi \in \Dc_\rho.$ From Lemmas \ref{lemVB-ImpRegStab-2} and \ref{lemVB-ImpRegStab}(a), we have 
	\begin{align*}
		\|\wt z\|_{\Dc} &=	\|\wt z^\psi\|_{L^2(0,\infty;H^4(\Omega))}+	\|\wt z^\psi\|_{L^\infty(0,\infty;H^2(\Omega))} + \|\wt z^\psi\|_{H^1(0,\infty; H^2(\Omega))} \\
		& \le M_1 \left( \|g(\psi)\|_{L^2(0,\infty; H^2(\Omega))} + \left\|\frac{\partial g}{\partial t}(\psi)\right\|_{L^2(0,\infty; \Lt)} +\|z_0\|_{H^3(\Omega)}\right)\\
		& \le M_1M_2\|\psi\|_\Dc^2 + M_1 \|z_0\|_{H^3(\Omega)}.
	\end{align*}
	Choosing $\|z_0\|_{H^3(\Omega)} \le \frac{\rho}{3M_1},$ $\rho \le \frac{1}{3M_1M_2},$ we have $\|\wt z^\psi\|_\Dc \le \rho$ and hence $\wt z\in \Dc_\rho.$ The rest of the proof follows as in the proof of Proposition \ref{ppsVB-stabNL-H10} using Lemmas \ref{lemVB-ImpRegStab-2} and \ref{lemVB-ImpRegStab}. 
\end{proof}

\noindent Recall the closed loop system \eqref{eqVB-closedloopshifted}:
\begin{equation} \label{eqVB-CLSforErr}
	\begin{aligned}
		& \wt z_t(\cdot, t)-\eta \Delta \wt z(\cdot,t) +y_s\textbf{v}\cdot\nabla \wt z(\cdot,t)+\textbf{v}\cdot\nabla y_s \wt z(\cdot,t)+(\nu_0-\omega )\wt z(\cdot, t) \\
		& \hspace{5.5cm}+e^{-\omega t}\wt z(\cdot,t)\textbf{v}\cdot \nabla \wt z(\cdot,t)=-\Bc \Bc^*\Pc \wt z(\cdot,t) \text{ in } Q,\\
		& \wt z(x,t)=0 \text{ on }\Sigma, \quad \wt z(x,0)=z_0(x) \text{ in }\Omega,
	\end{aligned}
\end{equation}

\noindent Theorem \ref{thVB-existImReg} leads to the following result on stability of \eqref{eqVB-CLSforErr}.
\begin{Theorem} \label{thVBMain-ImpRegStab}
	Let $\omega>0$ be any given real number. Then there exist positive constants $\rho_0,M$ depending on $\omega,\eta,\Omega$ such that for all $z_0 \in H^3(\Omega)\cap \Hio$ with $\Ac z_0 \in \Hio$ and $\|z_0\|_{H^3(\Omega)} \le M \rho,$ the non-linear closed loop system \eqref{eqVB-CLSforErr} admits a unique solution $\wt z\in L^2(0,\infty;\Hio\cap H^4(\Omega)) \cap L^\infty(0,\infty;\Hio\cap H^2(\Omega))\cap H^1(0,\infty;H^2(\Omega)\cap \Hio)$ satisfying
	\begin{align*}
		\|\wt z\|^2_{L^2(0,\infty;H^4(\Omega))} + \|\wt z\|^2_{L^\infty(0,\infty;H^2(\Omega))} + \|\wt z\|^2_{H^1(0,\infty;H^2(\omega))} \le \rho^2.
	\end{align*}
\end{Theorem}

\subsubsection{Main result on error estimates}
Now, our focus is on establishing an error estimate for the stabilized solution and stabilizing control, that is, to estimate $\wt z - \wt z_h,$ where $\wt z$ and $\wt z_h$ are the solutions of the closed loop systems \eqref{eqVB-CLSforErr} and \eqref{eqVB-disClNlOpForm}, respectively. Here, we use their weak formulation. The weak formulation of the closed loop system \eqref{eqVB-CLSforErr} is
{\small
\begin{equation} \label{eqVB-err cls wk}
	\begin{aligned}
		&  \langle \wt z_t,\phi\rangle +\eta \langle\nabla \wt z ,\nabla\phi\rangle + e^{-\omega t} \langle \wt z\textbf{v}\cdot \nabla \wt z , \phi\rangle + \langle y_s \textbf{v} \cdot \nabla \wt z+\textbf{v}\cdot \nabla y_s \wt z, \phi\rangle + (\nu_0-\omega)\langle \wt z,\phi\rangle =\langle -\Bc\Bc^*\Pc\wt z, \phi\rangle, \\
		& \langle \wt z(0),\phi\rangle = \langle z_0,\phi\rangle \text{ for all } \phi\in \Hio.
	\end{aligned}
\end{equation}}
and recall \eqref{eqVB-err cls dwk} as the weak formulation of the closed loop system \eqref{eqVB-disClNlOpForm}.


\noindent Define a projection 
$\wt \pi_h: \Hio \rightarrow V_h$ defined by 
\begin{align} \label{eqVB-proj 2 elpt}
	\eta\langle \nabla (\wt z - \wt\pi_h \wt z), \nabla \phi_h \rangle +\nu_0 \langle (\wt z-\wt \pi_h \wt z) , \phi_h\rangle = 0 \text{ for all } \wt z\in \Hio,\phi_h \in V_h.
\end{align}
In the lemma below, we state some properties satisfied by the projection $\wt \pi_h.$
\begin{Lemma}{\n\cite{Kundu20,Thomee}} \label{lemVB-ErrInEPr}
	Let $\wt \pi_h$ be as defined in \eqref{eqVB-proj 2 elpt}. Then the following estimates hold:
	\begin{itemize}
		\item[$(a)$] $\|\wt z -\wt \pi_h \wt z\|_{H^j(\Omega)} \le C h^{m-j}\|\wt z\|_{H^m(\Omega)}$ for all $j=0,1,$ and $m=1,2,$ for all $\wt z\in H^m(\Omega),$
		\item[$(b)$] $\|\wt z_t - \wt \pi_h \wt z_t\| \le Ch^2 \|\wt z_t\|_{H^2(\Omega)},$
		\item[$(c)$] $\|\wt \pi_h \wt z\|_{L^\infty(\Omega)} \le C \|\wt z\|_{H^2(\Omega)}$ and
		\item[$(d)$] $\|\nabla(\wt \pi_h \wt z)\|_{L^4(\Omega)} \le C \|\wt z\|_{H^2(\Omega)}.$
	\end{itemize}
\end{Lemma}

\noindent Subtract \eqref{eqVB-err cls dwk} from \eqref{eqVB-err cls wk} to obtain
\begin{align*}
	& \langle \wt z_t -\wt z_{h_t},\phi_h\rangle + \eta \langle \nabla \wt z - \nabla \wt z_h ,\nabla\phi_h\rangle + e^{-\omega t} \langle (\wt z - \wt z_h)\textbf{v}\cdot \nabla \wt z + \wt z_h\textbf{v}\cdot \nabla (\wt z-\wt z_h) , \phi_h\rangle   \\ 
	& \, +(\nu_0-\omega)\langle \wt z- \wt z_h,\phi_h\rangle + \langle y_s \textbf{v} \cdot \nabla (\wt z -\wt z_h)+\textbf{v}\cdot \nabla y_s (\wt z -\wt z_h),\phi_h\rangle    + \langle \Bc\Bc^* \Pc \wt z - \Bc_h\Bc_h^*\Pc_h\wt z_h, \phi_h\rangle=0. 
\end{align*} 

\noindent Introduce the variables $\Xi=\wt z -\wt\pi_h\wt z$ and $\zeta=\wt z_h -\wt \pi_h \wt z.$ Note that $\wt z -\wt z_h=\Xi -\zeta.$
Now, adding and subtracting of $\wt \pi_h \wt z$ in each term of above equation, we have
{\small
	\begin{align*}
		& \langle \Xi_t - \zeta_t,\phi_h\rangle + \eta \langle \nabla \wt z - \nabla (\pi_h \wt z) ,\nabla\phi_h\rangle - \eta \langle \nabla \wt z_h - \nabla (\pi_h \wt z) ,\nabla\phi_h\rangle \\
		& \quad + e^{-\omega t} \langle (\Xi -\zeta) \textbf{v}\cdot \nabla \wt z +\wt z_h\textbf{v}\cdot \nabla(\Xi-\zeta) , \phi_h\rangle  +\nu_0\langle \wt z- \wt\pi_h\wt z,\phi_h\rangle - \nu_0\langle \wt z_h- \wt\pi_h\wt z,\phi_h\rangle \\ 
		& \quad -\omega \langle \Xi -\zeta, \phi_h\rangle + \langle y_s \textbf{v} \cdot \nabla (\Xi -\zeta)+\textbf{v}\cdot \nabla y_s (\Xi - \zeta),\phi_h\rangle  + \langle \Bc\Bc^*\Pc \wt z - \Bc_h\Bc_h^*\Pc_h\wt z_h, \phi_h\rangle=0, 
	\end{align*} 
}for all $\phi_h \in V_h.$ Now, utilizing \eqref{eqVB-proj 2 elpt}, we have 
{\small
	\begin{equation} \label{eqVB-ErrZetaPhiEqn}
		\begin{aligned}
			\langle \zeta_t , \phi_h\rangle &= -\eta \langle \nabla \zeta , \nabla \phi_h\rangle + (\omega -\nu_0 ) \langle \zeta, \phi_h\rangle   + \langle \Xi_t, \phi_h\rangle + e^{-\omega t} \langle (\Xi -\zeta) \textbf{v}\cdot \nabla \wt z ,\phi\rangle   \\
			& \quad + e^{-\omega t}\langle \wt z_h\textbf{v}\cdot \nabla(\Xi-\zeta), \phi_h\rangle -\omega \langle \Xi , \phi_h\rangle + \langle y_s \textbf{v} \cdot \nabla (\Xi -\zeta)+\textbf{v}\cdot \nabla y_s (\Xi - \zeta),\phi_h\rangle  \\
			& \quad + \langle (\Bc\Bc^*\Pc  - \Bc_h\Bc_h^*\Pc_h \pi_h ) \wt z, \phi_h\rangle  + \langle  \Bc_h\Bc_h^*\Pc_h (\pi_h \wt z - \wt z_h) , \phi_h\rangle,
		\end{aligned}
	\end{equation}
}for all $\phi \in V_h.$
\begin{Theorem} \label{thVB-zeta est}
	Let $(\mathcal{A}_3)$ holds and $M,\rho_0$ be as in Theorem \ref{thVBMain-ImpRegStab}. Let $\rho \in (0,\rho_0]$ be any number and let $z_0\in H^3(\Omega)\cap \Hio$ such that $\Ac z_0 \in \Hio$ and $\|z_0\|_{H^3(\Omega)}\le M\rho.$  Let $ \wt z_h$ be solution of \eqref{eqVB-err cls dwk} and $\wt \pi_h$ be as defined in \eqref{eqVB-proj 2 elpt}. Then $\zeta= \wt z_h -\wt \pi_h \wt z$ satisfies the following estimate:
	\begin{equation*}
		\begin{aligned}
			\|\zeta(t)\|^2 + \alpha \int_0^t \|\nabla \zeta(s)\|^2 ds \le C h^{4(1-\epsilon) }e^{-\omega_P t},
		\end{aligned}
	\end{equation*}
	for some positive constant $C=C(\rho, y_s, \textbf{v},\eta, C_a, s_0, \|z_0\|_{H^3(\Omega)})$ and for any given $0<\epsilon<1.$
\end{Theorem}

\begin{proof}
	Choose $\phi_h=\zeta$ in \eqref{eqVB-ErrZetaPhiEqn} and use the definition of $\Ac_h$ in \eqref{eqVB-DisBil} and $\Ac_{\omega_h}$ in \eqref{eqdefVB-Awh} to obtain 
	\begin{equation*} 
		\begin{aligned}
			\frac{1}{2}  \frac{d}{dt}\|\zeta\|^2 &   - \langle\Xi_t, \zeta\rangle -  e^{-\omega t} \langle (\Xi -\zeta) \textbf{v}\cdot \nabla \wt z , \zeta\rangle  - e^{-\omega t} \langle \wt z_h\textbf{v}\cdot \nabla(\Xi-\zeta) , \zeta\rangle  + \omega \langle \Xi, \zeta\rangle \\
			& \quad  - \langle y_s \textbf{v} \cdot \nabla \Xi    +\textbf{v}\cdot \nabla y_s \Xi ,\zeta\rangle  + \langle (\Bc_h\Bc_h^*\Pc_h\pi_h - \Bc\Bc^*\Pc)\wt z, \zeta\rangle \\
			&  = -\eta \|\nabla \zeta\|^2 - \langle y_s \textbf{v} \cdot \nabla \zeta   +\textbf{v}\cdot \nabla y_s \zeta ,\zeta\rangle  + (\omega-\nu_0)\|\zeta\|^2- \langle \Bc_h\Bc_h^*\Pc_h \zeta, \zeta \rangle \\
			&  = \langle (\Ac_{\omega_h} -\Bc_h \Bc_h^* \Pc_h)\zeta, \zeta \rangle .
		\end{aligned}
	\end{equation*}
	Now, utilizing \eqref{eqVB-assumpt-omegaP-alpha}, we obtain 
	{\small
		\begin{equation} \label{eqVB-errest-allTerm1}
			\begin{aligned}
				\frac{1}{2}  \frac{d}{dt}\|\zeta\|^2  +\alpha \|\nabla \zeta\|^2 +\omega_P \|\zeta\|^2 & \le  \left\vert \langle\Xi_t, \zeta\rangle \right\vert + \left\vert e^{-\omega t} \langle (\Xi -\zeta) \textbf{v}\cdot \nabla \wt z , \zeta\rangle \right\vert +\left\vert  e^{-\omega t} \langle \wt z_h\textbf{v}\cdot \nabla(\Xi-\zeta) , \zeta\rangle \right\vert \\
				& \quad    +\left\vert  \omega \langle \Xi, \zeta\rangle \right\vert + \vert \langle y_s \textbf{v} \cdot \nabla \Xi   +\textbf{v}\cdot \nabla y_s \Xi ,\zeta\rangle \vert  \\ 
				& \quad +\left\vert \langle (\Bc_h\Bc_h^*\Pc_h\pi_h - \Bc\Bc^*\Pc)\wt z, \zeta\rangle \right\vert,\\
				& =:  \sum_{i=1}^6 T_i.
			\end{aligned}
		\end{equation}
	}\textbf{Estimate of $T_1$.} Use Lemma \ref{lemVB-ErrInEPr}(b) and Young's inequality to obtain
	\begin{align*}
		T_1=\left\vert \langle \Xi_t, \zeta\rangle\right\vert\le C_1(\omega_P)\|\Xi_t\|^2 +\frac{\omega_P}{6}\|\zeta\|^2 \le C_1(\omega_P) h^4 \|\wt z_t\|_{H^2(\Omega)}^2 + \frac{\omega_P}{6} \|\zeta\|^2.
	\end{align*}
	
	\medskip
	\noindent\textbf{Estimate of $T_2.$} Note that $e^{-\omega t} \langle (\Xi -\zeta) \textbf{v}\cdot \nabla \wt z , \zeta\rangle= e^{-\omega t} \langle \Xi  \textbf{v}\cdot \nabla \wt z , \zeta\rangle - e^{-\omega t} \langle \zeta  \textbf{v}\cdot \nabla \wt z , \zeta\rangle$.
	Use \eqref{lemPR:SobEmb} and Young's inequality to obtain
	\begin{equation*}
		\begin{aligned}
			T_2 & \le |\textbf{v}| \|\Xi\| \|\nabla  \wt z\|_{L^4(\Omega)} \|\zeta\|_{L^4(\Omega)} + |\textbf{v}| \|\zeta\| \|\nabla \wt z\|_{L^4(\Omega)} \|\zeta\|_{L^4(\Omega)}\\
			& \le |\textbf{v}|s_0^2 \|\Xi\| \|\wt z\|_{H^2(\Omega)} \|\nabla \zeta\| + |\textbf{v}| s_0^2\|\zeta\| \|\wt z\|_{H^2(\Omega)}\|\nabla \zeta\| \\
			& \le \frac{\alpha}{8} \|\nabla \zeta\|^2 + C_2(\textbf{v}, s_0, \alpha) \|\Xi\|^2 \|\wt z\|_{H^2(\Omega)}^2 +C_3(\textbf{v}, \alpha, s_0)\|\zeta\|^2\|\wt z\|_{H^2(\Omega)}^2.
		\end{aligned}
	\end{equation*}

	\medskip
	\noindent\textbf{Estimate of $T_3.$}
	First note that 
	\begin{equation*}
		\begin{aligned}
			\langle \wt z_h\textbf{v}\cdot \nabla(\Xi-\zeta) , \zeta\rangle & = \langle \wt z_h \textbf{v}\cdot \nabla \Xi , \zeta \rangle - \langle (\wt z_h -\wt \pi_h \wt z)\textbf{v}\cdot \nabla \zeta , \zeta \rangle - \langle \wt \pi_h \wt z\textbf{v}\cdot \nabla \zeta , \zeta \rangle \\
			& = \langle \wt z_h \textbf{v}\cdot \nabla \Xi , \zeta \rangle  - \langle \wt \pi_h \wt z\textbf{v}\cdot \nabla \zeta , \zeta \rangle,
		\end{aligned}
	\end{equation*}
	as $\langle (\wt z_h -\wt \pi_h \wt z)\textbf{v}\cdot \nabla \zeta , \zeta \rangle=\langle \zeta\textbf{v}\cdot \nabla \zeta , \zeta \rangle=0.$ Now, utilizing Lemma \ref{lemVB-ErrInEPr}(c) and Young's inequality, we obtain
	\begin{equation*}
		\begin{aligned}
			\left\vert\langle \wt \pi_h \wt z\textbf{v}\cdot \nabla \zeta , \zeta \rangle \right\vert\le |\textbf{v}|\| \|\wt \pi_h \wt z\|_{L^\infty(\Omega)} \|\nabla \zeta\| \|\zeta\| \le \frac{\alpha}{8}\|\nabla \zeta\|^2 + C_4(\alpha, \textbf{v}) \|\wt z\|_{H^2(\Omega)}^2\|\zeta\|^2.
		\end{aligned}
	\end{equation*}
	Observe that $\langle \wt z_h \textbf{v}\cdot \nabla \Xi , \zeta \rangle= \langle \zeta \textbf{v}\cdot \nabla \Xi, \zeta\rangle - \langle \wt \pi_h \wt z \textbf{v} \cdot \nabla \zeta , \Xi\rangle - \langle \Xi \textbf{v}\cdot \nabla(\wt \pi_h \wt z), \zeta \rangle$ and therefore, \eqref{lemPR:SobEmb}, Lemma \ref{lemVB-ErrInEPr}(a)-(d),  and Young's inequality lead to
	\begin{equation*}
		\begin{aligned}
			\left\vert \langle \wt z_h \textbf{v}\cdot \nabla \Xi , \zeta \rangle\right\vert & \le |\textbf{v}| \|\zeta\| \|\nabla \Xi\|_{L^4(\Omega)} \|\zeta\|_{L^4(\Omega)} + |\textbf{v}|\|\wt\pi_h \wt z\|_{L^\infty(\Omega)} \|\nabla \zeta\| \|\Xi\| \\
			& \qquad + |\textbf{v}|\|\Xi\|\|\nabla(\wt \pi_h\wt z)\|_{L^4(\Omega)}\|\zeta\|_{L^4(\Omega)}, \\
			& \le \frac{\alpha}{8} \|\nabla \zeta\|^2 +C_5(s_0, \alpha, \textbf{v}) \|\wt z\|_{H^2(\Omega)}^2\|\zeta\|^2 + C_6(\textbf{v},\alpha)\|\Xi\|^2 \|\wt z\|_{H^2(\Omega)}^2 \\
		\end{aligned}
	\end{equation*}


	\medskip
	\noindent\textbf{Estimate of $T_4.$} Young's inequality and Lemma \ref{lemVB-ErrInEPr}(a) yield
	\begin{align*}
		T_4 =\left\vert  \omega \langle \Xi, \zeta\rangle \right\vert\le C_7(\omega,\omega_P) \|\Xi\|^2 + \frac{\omega_P}{6}\|\zeta\|^2 \le C_7(\omega,\omega_P) h^4\|\wt z\|_{H^2(\Omega)}^2 + \frac{\omega_P}{6}\|\zeta\|^2 
	\end{align*}

	\medskip
	\noindent\textbf{Estimate of $T_5.$} Observe that $
	\big\langle y_s \textbf{v} \cdot \nabla \Xi  +\textbf{v}\cdot \nabla y_s \Xi ,\zeta\big\rangle  = \big\langle  \textbf{v} \cdot \nabla (y_s\Xi ) , \zeta\big\rangle  = - \big\langle y_s \Xi, \textbf{v}\cdot \nabla\zeta\big\rangle.$
	Therefore, Young's inequality and Lemma \ref{lemVB-ErrInEPr}(a) imply
	\begin{equation*}
		\begin{aligned}
			T_5 =\left\vert  - \big\langle y_s \Xi, \textbf{v}\cdot \nabla\zeta\big\rangle\right\vert  \le |\textbf{v}| \|y_s\|_{L^\infty(\Omega)} \|\Xi\|\|\nabla\zeta\| & \le \frac{\alpha}{8} \|\nabla \zeta\|^2  +  \frac{2}{\alpha}C_a |\textbf{v}|^2\|y_s\|_{H^2(\Omega)}^2\|\Xi\|^2\\
			&    \le \frac{\alpha}{8} \|\nabla \zeta\|^2  +  C_8(\alpha, C_a ,\textbf{v}, y_s) h^4 \|\wt z\|_{H^2(\Omega)}^2  .
		\end{aligned}
	\end{equation*}

	
	\medskip
	\noindent\textbf{Estimate of $T_6$.} For any $0<\epsilon<1,$ utilize Theorem \ref{thVB:main-conv-P} and the fact that $\Bc=\Bc_h$ on $V_h$ to obtain
	\begin{equation*}
		\begin{aligned}
			T_6 = \left\vert \langle (\Bc_h\Bc_h^*\Pc_h\pi_h - \Bc\Bc^*\Pc)\wt z, \zeta\rangle \right\vert & \le \left\vert \langle \Bc(\Bc_h^*\Pc_h\pi_h - \Bc^*\Pc)\wt z, \zeta\rangle \right\vert + \left\vert \langle (\Bc_h-\Bc)\Bc_h^*\Pc_h\pi_h \wt z, \zeta\rangle \right\vert \\
			& \le  \|(\Bc_h^*\Pc_h\pi_h -\Bc^*\Pc)\wt z\| \|\zeta\| \\
			& \le C h^{2(1-\epsilon)} \|\wt z\| \|\zeta\| \le C_{9}(\omega_P)h^{4(1-\epsilon)} \|\wt z\|^2 + \frac{\omega_P}{6}\|\zeta\|^2.
		\end{aligned}
	\end{equation*}
	
	\noindent Using the estimates of $T_i,$ $1\le i \le 6,$ in \eqref{eqVB-errest-allTerm1}, we obtain
	\begin{equation*}
		\begin{aligned}
			\frac{1}{2} \frac{d}{dt}\|\zeta\|^2 +\frac{\alpha}{2} \|\nabla \zeta\|^2 + \frac{\omega_P}{2}\|\zeta\|^2 & \le C_1 h^4\|\wt z_t\|_{H^2(\Omega)}^2  +  (C_3+C_4+C_5)\|\wt z\|_{H^2(\Omega)}^2 \|\zeta\|^2\\\
			& \quad  + (C_2+C_6) \|\Xi\|^2 \|\wt z\|^2_{H^2(\Omega)} +(C_7+C_8)h^{4} \|\wt z\|_{H^2(\Omega)}^2 \\
			& \quad  + C_9h^{4(1-\epsilon)} \|\wt z\|^2.
		\end{aligned}
	\end{equation*}
	Integrate the above inequality from $0$ to $t$ to obtain 
	\begin{equation*}
		\begin{aligned}
			\|\zeta(t)\|^2 + \alpha \int_0^t \|\nabla \zeta(s)\|^2 ds & \le \|\zeta(0)\|^2 + h^{4(1-\epsilon)}  \int_0^t \Big( C_1 \|z_t(s)\|_{H^2(\Omega)}^2  \\ 
			& \quad \quad +\wt C_3 \|z(s)\|_{H^2(\Omega)}^4   +\wt C_4 \|z(s)\|_{H^2(\Omega)}^2  + C_9 \|\wt z\|^2  \Big)ds \\
			& \quad \quad   + \int_0^t (\wt C_2 \|\wt z(s)\|_{H^2(\Omega)}^2 -\omega_P) \|\zeta(s)\|^2ds,
		\end{aligned}
	\end{equation*}
	where $\wt C_2=2(C_3+C_4+C_5), \wt C_3=2(C_2+C_6)$ and $\wt C_4=2(C_7+C_8).$	Note that $\|\zeta(0)\| = \|\wt z_h(0) -\wt \pi_h  z_0\| \le \|z_0- \pi_h z_0\| +\|z_0 -\wt \pi_h z_0\| \le Ch^2 \|z_0\|_{H^2(\Omega)}. $ Now, apply Gronwall's inequality (see \cite{Canon99}) along with Theorem \ref{thVBMain-ImpRegStab} to obtain 
	\begin{equation}
		\begin{aligned}
			\|\zeta(t)\|^2 +  \alpha\int_0^t  \|\nabla \zeta(s)\|^2 ds & \le  h^{4(1-\epsilon)} \left( \|z_0\|_{H^2(\Omega)}^2 + C \right) \times e^{ \int_0^t (\wt C_2 \|\wt z(s)\|_{H^2(\Omega)}^2 -\omega_P ) ds} \\
			& \le C h^{4(1-\epsilon)} e^{-\omega_P t}.
		\end{aligned}
	\end{equation}
	for some $C=C(\rho, y_s, \textbf{v},\eta, C_a, s_0, \|z_0\|_{H^3(\Omega)})>0.$
	The proof is complete.
\end{proof}

\noindent Now, we conclude this chapter by stating and proving the main result on the error estimate for the stabilized non-linear system.
\begin{Theorem}[error estimate for non-linear stabilized system] \label{thVB-mainErrNL}
	Let $M,\rho_0$ be as in Theorem \ref{thVBMain-ImpRegStab} and $(\mathcal{A}_3)$ holds. Let $\rho \in (0,\rho_0]$ be any number and let $z_0\in H^3(\Omega)\cap \Hio$ such that $\Ac z_0 \in \Hio$ and $\|z_0\|_{H^3(\Omega)}\le M\rho.$ 
	Let $\wt u(t)=-\Bc^*\Pc \wt z(t) $ and $\wt u_h(t)=-\Bc_h^*\Pc_h \wt z_h(t), $ where $\wt z$ and $\wt z_h$ are solutions of  \eqref{eqVB-CLSforErr} and \eqref{eqVB-disClNlOpForm}, respectively. Then for any $0<\epsilon<1$ and for some $C=C(\rho, y_s, \textbf{v},\eta, C_a, s_0, \|z_0\|_{H^3(\Omega)})>0,$ we have
	\begin{itemize}
		\item[$(a)$] $\|\wt z(t) -\wt z_h(t)\| \le  C h^{2(1-\epsilon)} $ for all $t>0,$ and
		\item[$(b)$] $\|\wt u(t) -\wt u_h(t)\| \le C  h^{2(1-\epsilon)} $ for all $t>0.$
	\end{itemize}	
\end{Theorem}

\begin{proof}
	Note that $\wt z -\wt z_h= \Xi-\zeta.$ From Lemma \ref{lemVB-ErrInEPr}(a) and Theorem \ref{thVBMain-ImpRegStab}, we have 
	\begin{align*}
		\|\Xi\| \le Ch^2\|\wt z\|_{H^2(\Omega)} \le C\rho h^2. 
	\end{align*}
	Using this and Theorem \ref{thVB-zeta est}, we have 
	\begin{align*}
		\|\wt z(t) -\wt z_h(t)\| \le \|\Xi\|+\|\zeta\| \le Ch^{2(1-\epsilon)}  \text{ for all } t>0,
	\end{align*}
	for some $C=C(\rho, y_s, \textbf{v},\eta, C_a, s_0, \|z_0\|_{H^2(\Omega)})>0.$
	Now, utilizing (a) and Theorem \ref{thVB:main-conv-P}(a)-(c), we obtain the estimate in (b) 
\end{proof}

\subsection{Implementation}\label{secVB-impl-NL}
In this subsection, an implementation procedure for solving \eqref{eqVB-linarnd-ys} is discussed and it is also shown that the system is stabilizable by control in feedback form where the feedback operator is as obtained in the case of linearized case. The weak formulation of \eqref{eqVB-linarnd-ys} seeks $\wt z\in \Hio$ such that
\begin{equation}  \label{eqVB-weakformNL-NI}
	\begin{aligned}
		&  \langle \wt z_t,\phi\rangle +\eta \langle\nabla \wt z ,\nabla\phi\rangle + e^{-\omega t} \langle \wt z\textbf{v}\cdot \nabla \wt z , \phi\rangle + \langle y_s \textbf{v} \cdot \nabla \wt z+\textbf{v}\cdot \nabla y_s \wt z, \phi\rangle + (\nu_0-\omega)\langle \wt z,\phi\rangle =\langle \wt u, \phi\rangle,\\
		& \langle \wt z(0),\phi\rangle = \langle z_0,\phi\rangle,
	\end{aligned}
\end{equation}
for all $\phi\in \Hio.$ 
\noindent  Let $\mathcal{T}_h$ be the triangulation of $\Omega$ and $V_h$ be the approximated subspace of $\Lt$ as mentioned in Section \ref{subVB-NI-Lin}. Let $\{\phi_h^i\}_{i=1}^{n_h}$ be a nodal basis function for $V_h$ with dim$(V_h)=n_h.$
The semi-discrete formulation corresponding to \eqref{eqVB-weakformNL-NI} seeks $\wt z_h\in V_h$ such that
\begin{equation} \label{eqVB-disweakformNL}
	\begin{aligned}
		& \langle \wt z_{h_t},\phi_h\rangle +\eta \langle\nabla \wt z_h ,\nabla\phi_h\rangle + e^{-\omega t} \langle \wt z_h\textbf{v}\cdot \nabla \wt z_h , \phi_h\rangle + \langle y_s \textbf{v} \cdot \nabla \wt z_h+\textbf{v}\cdot \nabla y_s \wt z_h, \phi_h\rangle  \\ & \quad \qquad\qquad +(\nu_0-\omega)\langle \wt z_h,\phi_h\rangle = \langle \wt u_h, \phi_h\rangle, \\
		& \langle \wt z_h(0),\phi_h\rangle = \langle \pi_h z_0,\phi_h\rangle,
	\end{aligned}
\end{equation}
for all $\phi_h\in V_h.$
Let $\displaystyle \wt z_h=\sum_{i=1}^{n_h} z_i\phi_h^i$ and $\displaystyle \wt u_h=\sum_{i=1}^{n_h} u_i\phi_h^i.$ Then, for all $1\le k\le n_h,$ we have
{\small
\begin{equation} \label{eqVB-matform-NL-I}
	\begin{aligned}
		& \sum_{i=1}^{n_h} \frac{dz_i}{d t}\left\langle \phi_h^i , \phi_h^k\right\rangle + \eta \sum_{i=1}^{n_h} z_i\left\langle\nabla \phi_h^i ,\nabla \phi_h^k\right\rangle +e^{-\omega t} \sum_{i=1}^{n_h}z_i\left\langle \phi_h^i \sum_{j=1}^{n_h} z_j \textbf{v}\cdot \nabla \phi_h^j, \phi_h^k\right\rangle  \\
		& \quad + \sum_{i=1}^{n_h}z_i\left\langle y_s \textbf{v}\cdot  \nabla\phi_h^i, \phi_h^k\right\rangle +\sum_{i=1}^{n_h}z_i\left\langle \textbf{v}\cdot\nabla y_s \phi_h^i,\phi_h^k\right\rangle +(\nu_0-\omega)\sum_{i=1}^{n_h}z_i\left\langle \phi_h^i, \phi_h^k\right\rangle=\sum_{i=1}^{n_h} u_i\langle\phi_h^i, \phi_h^k\rangle.
	\end{aligned}
\end{equation}
}For ${\Krm}_h=  (\langle \nabla \phi_h^i,\nabla\phi_h^k\rangle)_{1\le i,k\le n_h},$ ${\Mrm}_{h}=(\langle \phi_h^i,\phi_h^k\rangle)_{1\le i,k\le n_h},$ ${\Arm}^1_h=(\langle y_s \textbf{v}\cdot \nabla \phi_h^i, \phi_h^k\rangle)_{1\le i,k\le n_h},$  $ {\Arm}^2_h=(\langle \textbf{v}\cdot \nabla y_s  \phi_h^i, \phi_h^k\rangle)_{1\le i,k\le n_h} ,$  $\Nrm_h(\Zrm_h)=\left\langle \phi_h^i \sum_{j=1}^{n_h} z_j \textbf{v}\cdot \nabla \phi_h^j, \phi_h^k\right\rangle_{1\le i,k\le n_h},$ and $\Zrm_h=(z_1,\ldots, z_{n_h})^T,$ \eqref{eqVB-matform-NL-I} can be written in the following matrix formulation
\begin{equation} \label{eqnVB:matform-NL-II}
	{\Mrm}_h {\Zrm}_h'(t)=({\Arm}_h+\omega {\Mrm}_h){\Zrm}_h(t)+ e^{-\omega t}\Nrm_h(\Zrm_h)\Zrm_h(t)+\Brm_h\mathfrak{u}_h(t),
\end{equation}
where ${\Arm}_h=-\eta{\Krm}_h -\nu_0\Mrm_{h} - {\Arm}^1_h - {\Arm}^2_h$ and $\mathfrak{u}_h=(u_1,...,u_{n_h})^T\in\mathbb{R}^{n_h\times 1}$ is control we seek. Our aim is to stabilize \eqref{eqnVB:matform-NL-II} by the same feedback control obtained in the linearized case, that is, $\mathfrak{u}_h(t)=-\mathbb{S}_h\Prm_h\Zrm_h(t),$ where $\mathbb{S}_h$ is as in \eqref{eqn-mathbb-Sh} and $\Prm_h$ solves the algebraic Riccati equation \eqref{eq:ARE-mat}. Now, we discuss the implementation procedures for \eqref{eqnVB:matform-NL-II} with $\mathfrak{u}_h(t)=-\mathbb{S}_h\Prm_h\Zrm_h(t),$ that is, 
\begin{equation} \label{eqnVB:matform-NL-III}
	{\Mrm}_h {\Zrm}_h'(t)=({\Arm}_h+\omega {\Mrm}_h){\Zrm}_h(t)+ e^{-\omega t} \Nrm_h(\Zrm_h)\Zrm_h(t)-\Brm_h \mathbb{S}_h \Prm_h \Zrm_h(t).
\end{equation}

\noindent \textbf{Time solver.} A time discretization using a
backward Euler method leads to a system
\begin{align*}
	& {\Mrm}_h \frac{{\Zrm}_h^1 - {\Zrm}_h^{0}}{\Delta t}=({\Arm}_h+\omega {\Mrm}_h){\Zrm}_h^1+e^{-\omega t}\Nrm_h(\Zrm_h^1)\Zrm_h^1-\Brm_h\mathbb{S}_h \Prm_h \Zrm_h^1,\\
	& {\Zrm}_h^0=\Zrm_0,
\end{align*}
that is, 
\begin{equation} \label{eqVB-BEul-NL-2BDF}
	\begin{aligned}
		&\left( \Mrm_{h} - \Delta t \Arm_h -\Delta t \omega \Mrm_{h} +\Delta t \Brm_h\mathbb{S}_h\Prm_h - \Delta t e^{-\omega t}\Nrm_h(\Zrm_h^1) \right)\Zrm_h^1 - \Mrm_{h} \Zrm_h^{0}=0,\\
		& {\Zrm}_h^0=\Zrm_0. \\
	\end{aligned}
\end{equation}
Using Newton's method, for known ${\Zrm}_h^0=\Zrm_0,$ we solve the above equation for $\Zrm_h^1$ as below. Denote \eqref{eqVB-BEul-NL-2BDF} as 
\begin{equation} \label{eqVB-FZ_n-NL}
	\mathcal{F}(\Zrm_h^1)=0.
\end{equation}
Note that \eqref{eqVB-BEul-NL-2BDF} and \eqref{eqVB-FZ_n-NL} are non-linear equations in matrix form. Now, for a known $\Zrm_h^0,$ we want to find $\Zrm_h^1$ that satisfies \eqref{eqVB-BEul-NL-2BDF} by using Newtons' method \cite[Section 1.8 of Chapter 1]{ChenBook}. Newton's method for solving \eqref{eqVB-FZ_n-NL} can be defined as following:
\begin{align*}
	& \text{Set }{\Zrm}_h^{1,0}={\Zrm}_h^0;\\
	& \text{iterate } {\Zrm}_h^{1,p}={\Zrm}_h^{1,p-1}+d^{1,p}, \, p=1,2,\ldots,
\end{align*}
where $d^{1,p}$ solves 

\begin{equation}
	G(\Zrm_h^{1,p-1}) d^{1,p}=-\mathcal{F}(\Zrm_h^{1,p-1}) \text{ for all } p=1,2,\ldots,
\end{equation}
where $G(\Zrm_h^{1,p-1})$ is the Jacobian matrix, that is, $G(\Zrm_h^{1,p-1})=\left(\frac{\partial \mathcal{F}_k}{\partial z_l^{1,p-1}}\right)_{1\le k, l\le n_h}$ with $\mathcal{F}_k$ as the $k$-th component of $\mathcal{F}.$ \\

\noindent \textbf{Computation of Jacobian Matrix $G$.} Recall, 
\begin{align*}
	\mathcal{F}(\Zrm_h^{1})= \left( \Mrm_{h} - \Delta t \Arm_h -\Delta t \omega \Mrm_{h} +\Delta t \Brm_h\mathbb{S}_h\Prm_h - \Delta t e^{-\omega t} \Nrm_h(\Zrm_h^1) \right)\Zrm_h^1 - \Mrm_{h} \Zrm_h^{0}.
\end{align*}
To compute the $G,$ Jacobian of $\mathcal{F},$ first note that only non-linear term in $\mathcal{F}$ is $$\displaystyle \Nrm_h(\Zrm_h^{1,p-1})\Zrm_h^{1,p-1}= \left( \sum_{i=1}^{n_h} z_i^{1,p-1}\left\langle \phi_h^i \sum_{j=1}^{n_h} z_j^{1,p-1}  \textbf{v}\cdot \nabla \phi_h^j, \phi_h^k \right\rangle \right)_{1\le k\le n_h}=: \left(N_k\right)_{1\le k \le n_h}.$$ Now, 
\begin{equation} \label{eqVB-Jacob-Nk'}
	\begin{aligned}
		\frac{\partial N_k}{\partial z_l^{1,p-1}} & = \frac{\partial}{\partial z_l^{1,p-1}} \left( \sum_{i=1}^{n_h}\sum_{j=1}^{n_h} z_i^{1,p-1}  z_j^{1,p-1}\left\langle \phi_h^i  \textbf{v}\cdot \nabla \phi_h^j, \phi_h^k \right\rangle \right) \\
		& = \sum_{j=1}^{n_h} z_{j}^{1,p-1} \left\langle \phi_h^l\textbf{v}\cdot \nabla \phi_h^j,\phi_h^k\right\rangle + \sum_{i=1}^{n_h} z_{i}^{1,p-1} \left\langle \phi_h^i\textbf{v}\cdot \nabla \phi_h^l,\phi_h^k\right\rangle.
	\end{aligned}
\end{equation}
Therefore, 
\begin{align*}
	G(\Zrm_h^{1,p-1}) = \Mrm_h -\Delta t \Arm_h-\Delta t\omega \Mrm_h+\Delta t \Brm_h\mathbb{S}_h\Prm_h - \Delta t e^{-\omega t} \left( \frac{\partial N_k}{z_l^{1,p-1}} \right)_{1\le k,l\le n_h},
\end{align*}
where $\left( \frac{\partial N_k}{z_l^{1,p-1}} \right)_{1\le k,l\le n_h}$ is as in \eqref{eqVB-Jacob-Nk'}.

\noindent Now, starting from the second time step, apply the backward difference formula 2 (BDF2, \cite{MatESAIM}) to obtain
\begin{align*}
	& {\Mrm}_h \frac{1.5{\Zrm}_h^{n+1} - 2{\Zrm}_h^{n}+0.5{\Zrm}_h^{n-1}}{\Delta t}=({\Arm}_h+\omega {\Mrm}_h){\Zrm}_h^{n+1}+e^{-\omega t}\Nrm_h(\Zrm_h^{n+1})\Zrm_h^{n+1}-\Brm_h \mathbb{S}_h \Prm_h \Zrm_h^{n+1},
\end{align*}
for all $n=1,2,\ldots.$ Re-write the last two equations to obtain 
\begin{equation} \label{eqVB-BEul-NL-2BDF2}
	\begin{aligned}
		&\left( 1.5\Mrm_{h} - \Delta t \Arm_h -\Delta t \omega \Mrm_{h} +\Delta t \Brm_h \mathbb{S}_h \Prm_h - \Delta t e^{-\omega t}\Nrm_h(\Zrm_h^{n+1}) \right)\Zrm_h^{n+1} \\
		& \qquad\qquad \qquad\qquad \qquad\qquad- 2\Mrm_{h} \Zrm_h^{n}+0.5\Mrm_{h} \Zrm_h^{n-1}=0 \text{ for all }n=1,2,\ldots,\\
	\end{aligned}
\end{equation}
Now, denoting the above equation as $\mathcal{F}(\Zrm_h^{n+1})=0$ for all $n=1,2,\ldots,$ proceed as in above to solve for $\Zrm_h^{2}$ where $\Zrm_h^{0}$ and $\Zrm_h^{0}$ are known.  

\medskip
\noindent \textbf{Numerical results.}
Here, we implement two numerical examples: first one is an example having exact solution and second one is for stabilization. The data for the next two examples are tabulated in Table \ref{tab:data-NL}.

\begin{table}[ht!]
	\centering
	\begin{tabular}{|c||c|c|c|c|c|c|c|}
		\hline
		Data & $\omega$ &   $\eta$ & $\nu_0$ &$\textbf{v}$ & $y_0(x_1,x_2)$ & $y_s(x_1,x_2,t)$  \\ 
		\hline
		\hline
		& & & & & & \\
		{\bf Example 2} & $0$ & $5$ & 0 & $(1,1)$ & $\sin(\pi x_1)\sin(\pi x_2)$ & $x_1x_2(1-x_1)(1-x_2)$\\
		&  &  &  &  & $+x_1x_2(1-x_1)(1-x_2)$ & \\
		& & & & & &\\
		\hline
		& & & & & &\\
		{\bf Example 3} & 25 & $1$ & 0& (1,1) & $x_1x_2(1-x_1) (1-x_2)$ & $\sin(\pi x_1)\sin(\pi x_2)$  \\
		& & & & & &\\
		\hline
	\end{tabular}
	\caption{Data for Examples 2 \& 3.}\label{tab:data-NL}
\end{table}

\subsubsection{Example 2} In this example, we add a force function $g(x_1,x_2,t)$ on the right hand side and take $u=0,$ so that solution of the corresponding system is known. Consider the force function as 
\begin{align*}
	g(x_1,x_2,t)=& (1+2\pi^2\eta)e^t\sin(\pi x_1)\sin(\pi x_2) + (x_1x_2(1 - x_1)  (1 - x_2) + e^t\sin(\pi x_1)\sin(\pi x_2))\\
	& \quad \times e^t\pi (\cos(\pi x_1)\sin(\pi x_2)+\sin(\pi x_1)\cos(\pi x_2))+	\big(x_2(1-2x_1)(1-x_2) \\
	& \quad +x_1(1-x_1)(1-2x_2)\big) e^t\sin(\pi x_1)\sin(\pi x_2)
\end{align*}
and then the exact solution of \eqref{eqVB-weakformNL-NI} is  $z(x_1,x_2,t)=e^t \sin(\pi x_1)\sin(\pi x_2)$ with $z_0(x_1,x_2)=\sin(\pi x_1)\sin(\pi x_2).$  Errors of the computed solution are plotted on log-log scale against $h$ in Figure \ref{figVB-NL-Ex3-loglog sol} while errors and orders of convergence are tabulated in Table \ref{tab:Ex3_NL_exact_OC}.

%
%

\noindent \begin{minipage}{0.35\textwidth}
	\vspace{1cm}
	\includegraphics[width=\textwidth]{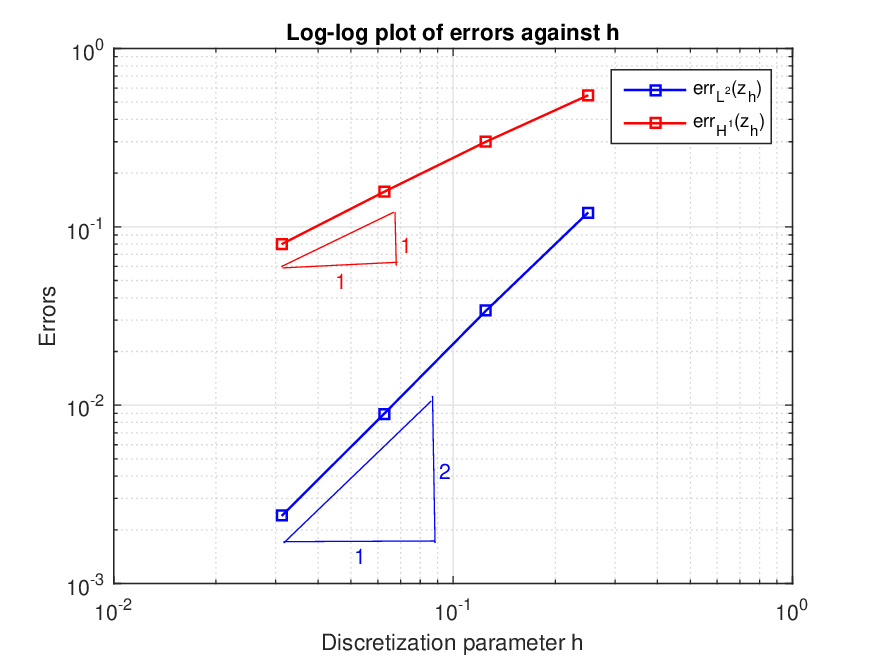}
	\captionof{figure}{ (Example 2.) Log-log plots of errors against discretization parameter $h$.}
	\label{figVB-NL-Ex3-loglog sol} 
\end{minipage}
\begin{minipage}{0.65\textwidth}
	\begin{tabular}{|c||c|c||c|c||}
		\hline
		$h$ & $err_{L^2}(\wt z_h)$ &   Order & $err_{H^1}(\wt z_h)$ & Order  \\ 
		\hline
		\hline
		$\frac{1}{2^2}$ & 1.203788e-01 & & 5.469386e-01 &  \\
		\hline
		$\frac{1}{2^3}$ & 3.386487e-02 & 1.829720 & 3.001354e-01 &  0.865765 \\
		\hline
		$\frac{1}{2^4}$ & 8.919755e-03 & 1.924713 & 1.569304e-01 & 0.935488 \\
		\hline
		$\frac{1}{2^5}$ & 2.394282e-03 & 1.897410 & 7.999054e-02 & 0.972223 \\
		\hline
	\end{tabular} \captionof{table}{(Example 2.) Errors and orders of convergence for the computed solution.}\label{tab:Ex3_NL_exact_OC}
\end{minipage}


\begin{figure}[ht!]
	\includegraphics[width=.45\textwidth]{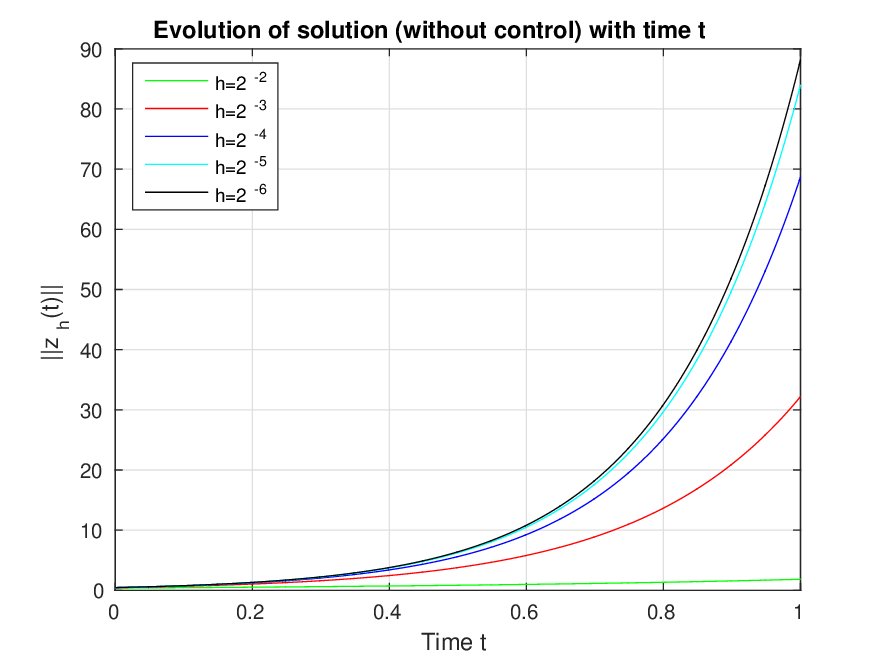}\hfill
	\includegraphics[width=.45\textwidth]{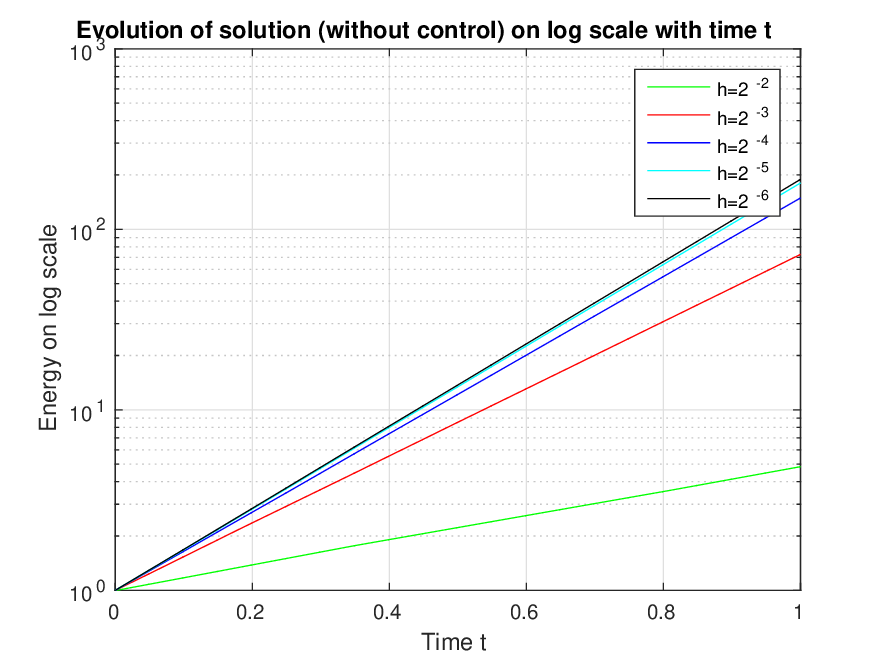}
	\caption{(Example 3.) (a) Evolution of solution without control, (b) on log scale. }     \label{figNL-ex4-WC-NL}  
\end{figure}

\noindent \begin{minipage}{0.35\textwidth}
	\vspace{1cm}
	\includegraphics[width=\textwidth]{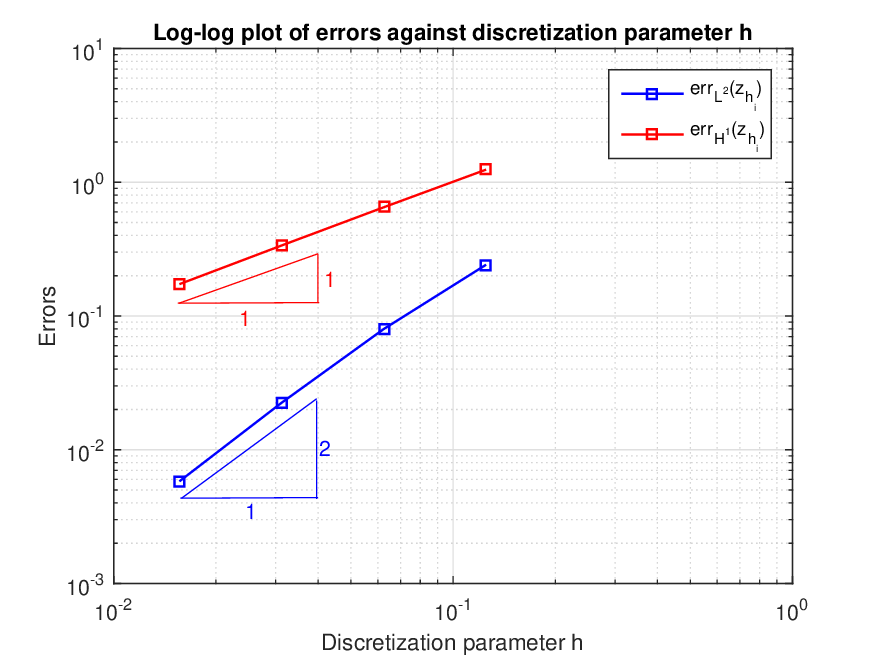}
	\captionof{figure}{ (Example 3.) Log-log plots of errors against discretization parameter $h$.}
	\label{figNL-ex4-WC-NL--1} 
\end{minipage}
\begin{minipage}{0.65\textwidth}
	{\small
	\begin{tabular}{|c||c|c||c|c||}
		\hline
		$h$ & $err_{L^2}(z_h)$ &   Order & $err_{H^1}(z_h)$ & Order  \\ 
		\hline
		\hline
		
		$\frac{1}{2^3}$ & 2.416313e-01 & --- & 1.242321 &  --- \\
		\hline
		$\frac{1}{2^4}$ & 8.000403e-02 & 1.594663 & 6.504246e-01 & 0.933584 \\
		\hline
		$\frac{1}{2^5}$ & 2.248219e-02 & 1.831290 & 3.368542e-01 & 0.949257 \\
		\hline
		$\frac{1}{2^5}$ & 5.814795e-03 & 1.950982 & 1.731318e-01 & 0.960253 \\
		\hline
	\end{tabular} \captionof{table}{(Example 3.) Errors and orders of convergence for the computed solution.}\label{tab:Ex4_NL_WC} }
\end{minipage}

\subsubsection{Example 3}
With the data as given in Table \ref{tab:data-NL}, we first compute the solution without control and observe that the energy of the solution increases as time increases. This behavior is plotted in   Figure \ref{figNL-ex4-WC-NL}(a) - (b). The computed errors are plotted on log-log scale against $h$ in Figure \ref{figNL-ex4-WC-NL--1} and tabulated in Table \ref{tab:Ex4_NL_WC} that shows a quadratic and linear rates of convergence in $L^2$ and $H^1$ norms respectively.

%

\noindent Next, we computed the feedback control and the corresponding solution is stabilized. The evolution of the stabilized solution with time $t$ is plotted in Figure \ref{figVB-ex3-StabSol-NL}. Figure \ref{figVB-ex3-StabCtrl-NL} (a) - (b) represents the evolution of stabilizing control with time $t.$ The errors of the solution and control are plotted in log-log scale in Figure \ref{figVB-ex3-StabCtrl-NL}(c) and the orders of convergence are tabulated in Table \ref{tab:Ex4_NL_Stab}.

\begin{figure}[ht!]
	\includegraphics[width=.45\textwidth]{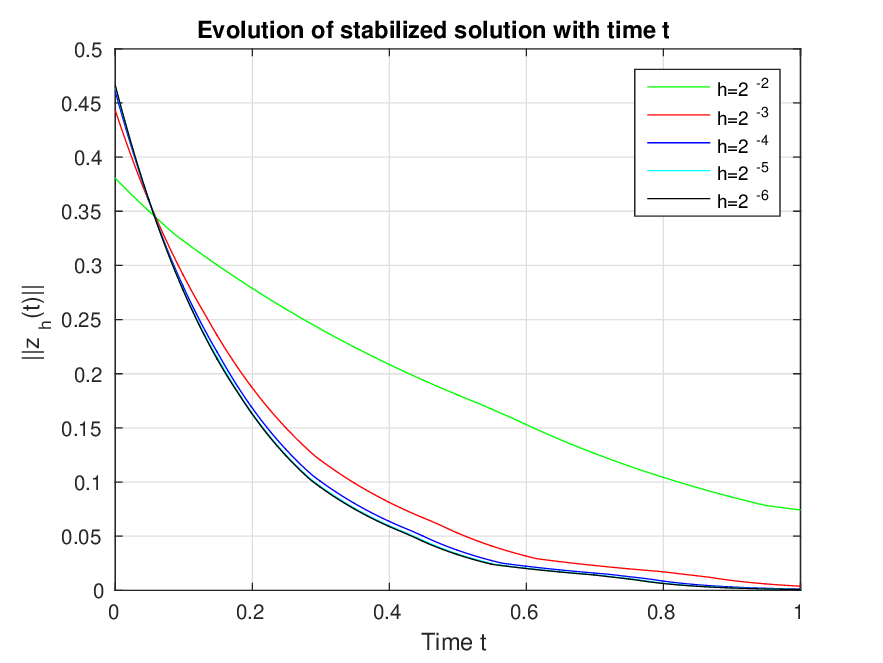}\hfill
	\includegraphics[width=.45\textwidth]{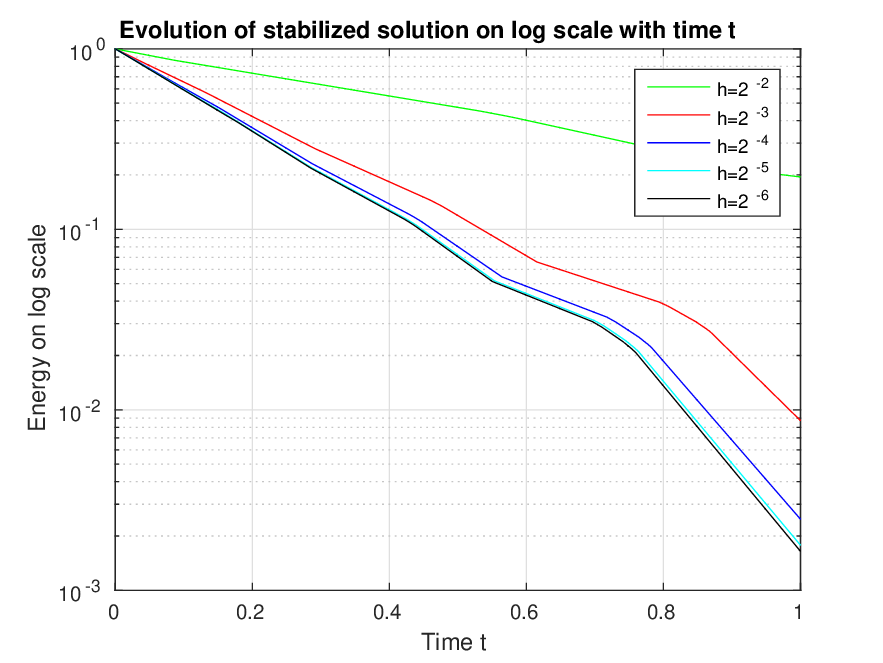}
	\caption{(Example 3.) (a) Evolution of \textbf{stabilized solution}, (b) on log scale. }  \label{figVB-ex3-StabSol-NL}    
\end{figure}

\begin{figure}[ht!]
	\includegraphics[width=.32\textwidth]{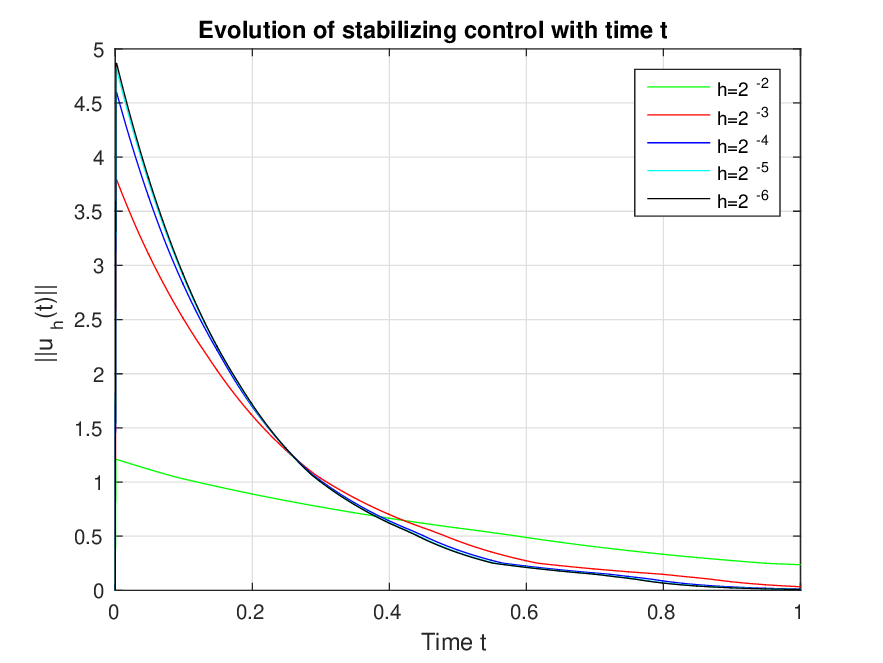}\hfill
	\includegraphics[width=.32\textwidth]{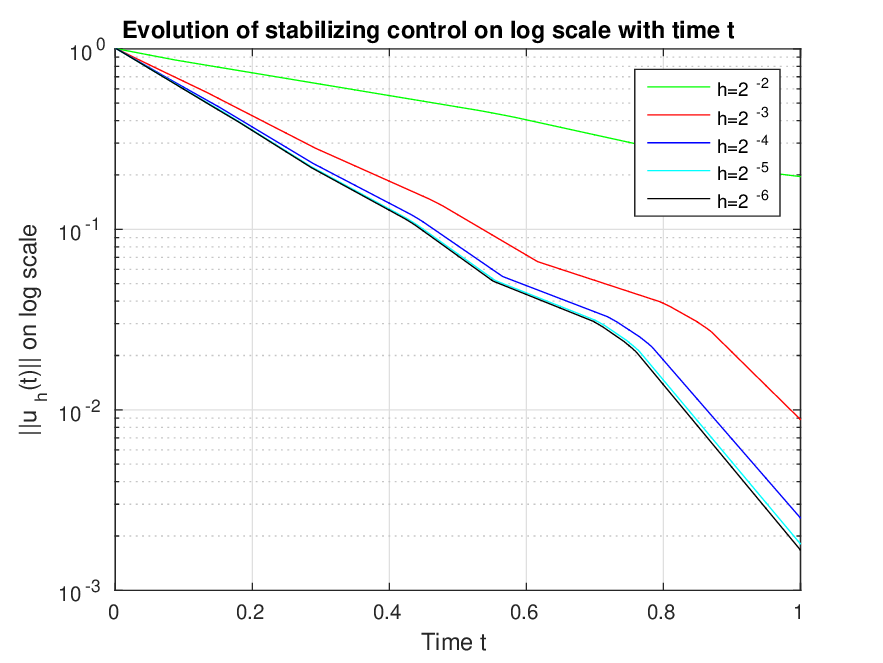}\hfill
	\includegraphics[width=.32\textwidth]{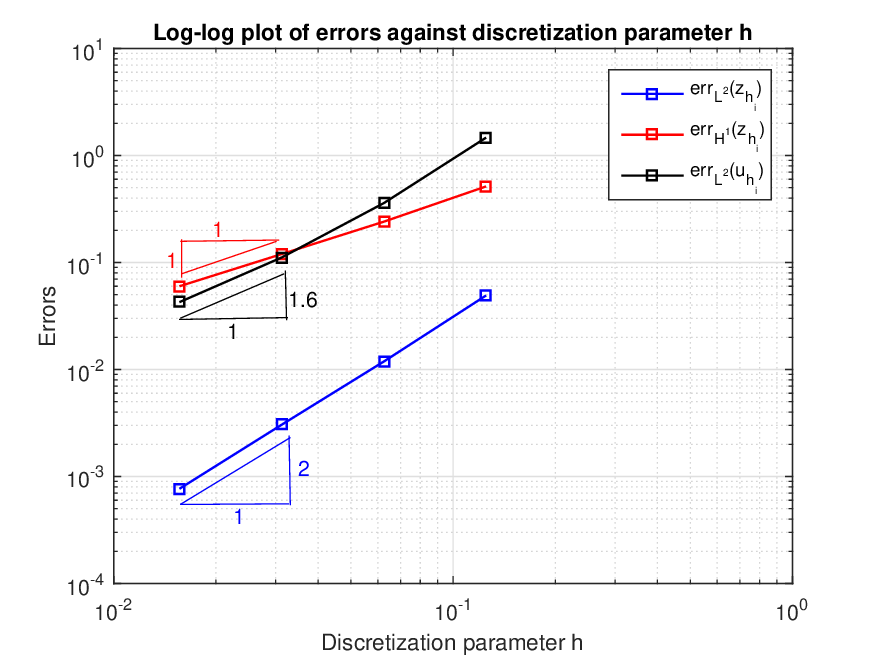}
	\caption{(Example 3.) (a) Evolution of \textbf{stabilizing control}, (b) on log scale and (c) error plot in log-log scale. }  \label{figVB-ex3-StabCtrl-NL}      
\end{figure}

\begin{table}[ht!]
	\centering
	\begin{tabular}{|c||c|c||c|c||c|c||}
		\hline
		$h$ & $err_{L^2}(\wt z_h)$ &   Order & $err_{H^1}(\wt z_h)$ & Order  & $err_{L^2}(\wt u_h)$ & Order \\ 
		\hline
		\hline
		$\frac{1}{2^3}$ & 4.914705e-02 & --- & 5.137541e-01 &  --- & 1.467911  & --- \\
		\hline
		$\frac{1}{2^4}$ & 1.187577e-02 & 2.049082 & 2.408034e-01 & 1.093222 & 3.591835e-01 & 2.030972\\
		\hline
		$\frac{1}{2^5}$ & 3.0482403e-03 & 1.961973 & 1.198641e-01 & 1.006455 & 1.114329e-01 & 1.688544\\
		\hline
		$\frac{1}{2^6}$ & 7.687377e-04 & 1.987413 & 6.002921e-02 & 0.997664 & 4.264087e-02 & 1.385867\\
		\hline
	\end{tabular}
	\caption{Errors and orders of convergence of \textbf{stabilized} solution and \textbf{stabilizing} control.}\label{tab:Ex4_NL_Stab}
\end{table}

%

\appendix

\section{}
\subsection{Well-posedness of steady state equation}
In this subsection, we study the well-posedness of \eqref{eqVB-steadystate}, that is, 
\begin{equation} \label{eqVB-steadystateAgain}
	\begin{aligned}
		-\eta \Delta y_s(x)+y_s(x) \textbf{v}.\nabla y_s(x) +\nu_0 y_s(x) & = f_s \text{ in }\Omega, 
		& y_s =0 \text{ on }\partial\Omega.
	\end{aligned}
\end{equation}
 In fact, we show that  for any given $f_s\in L^2(\Omega),$ there exists a unique solution  $y_s$ of \eqref{eqVB-steadystateAgain} in $H^2(\Omega)\cap \Hio.$ 
\begin{Proposition}
Let for any $\rho>0,$ $D_\rho=\{ y\in H^2(\Omega)\cap \Hio\,|\, \|y\|_{H^2(\Omega)}\le \rho \}$ and $\nu_0\ge 0.$ There exists $\rho_0>0,$ such that for any $0<\rho\le \rho_0$ and given $f\in \Lt$ with $\|f\|\le \frac{\rho}{2},$ there exists a unique solution $y_s \in D_\rho$ of \eqref{eqVB-steadystateAgain}.
\end{Proposition}
\begin{proof}
We prove the proposition using Banach fixed point theorem in the following steps. 

\noindent \textbf{Step 1.} Consider the following equation
\begin{align}\label{eq-linlaplc}
-\eta \Delta y_s + \nu_0 y_s=f \text{ in }\Omega, \quad  y_s=0 \text{ on }\partial\Omega.
\end{align}
The weak formulation corresponding to \eqref{eq-linlaplc} seeks $y_s\in \Hio$ such that
\begin{align*}
\eta \langle \nabla y_s, \nabla \phi\rangle + \nu_0\langle y_s, \phi\rangle=\langle f,\phi\rangle \quad \text{ for all }\phi\in \Hio.
\end{align*}
For $\nu_0\ge 0,$ using Poincar\'{e} inequality, it can be observed that $\eta\|\nabla \phi\|^2+\nu_0\|\phi\|^2 \ge \eta \|\nabla\phi\|^2.$ Thus, an application of Lax-Milgram theorem \cite[Theorem 3.1.4]{Kes} leads to the existence of a unique solution $y_s\in \Hio$ and a elliptic regularity \cite[Theorem 4, Chapter 5]{Eva} implies $y_s\in H^2(\Omega)$ with
\begin{align*}
\|y_s\|_{H^2(\Omega)} \le C \|f\|,
\end{align*}
for some $C>0.$ 

\medskip
\noindent \textbf{Step 2.} Now, for a given $\psi\in H^2(\Omega),$ consider
\begin{align} \label{eq-linlaplNonH}
-\eta \Delta y_s + \nu_0 y_s=f +g(\psi)\text{ in }\Omega, \quad  y_s=0 \text{ on }\partial\Omega,
\end{align}
where $g(\psi)=-\psi\textbf{v}\cdot \nabla \psi.$ Our aim is to show that \eqref{eq-linlaplNonH} has unique solution $y_s\in H^2(\Omega)\cap \Hio$ by showing $g(\psi)\in \Lt$ and using Step 1. Note that, $g(\psi)=\displaystyle \sum_{i=1}^d \psi v_i\frac{\partial \psi}{\partial x_i},$ where $\textbf{v}=(v_1,v_2,\ldots, v_d)^T\in \mathbb{R}^d.$ A H\"{o}lders inequality leads to 
{\small
\begin{align*}
\left\| \psi v_i\frac{\partial \psi}{\partial x_i}\right\|^2 = \int_\Omega \left \vert \psi v_i\frac{\partial \psi}{\partial x_i} \right\vert^2 dx \le \left( \int_\Omega |\psi|^4dx \right)^{1/2} \left( \int_\Omega \left \vert v_i\frac{\partial \psi}{\partial x_i} \right\vert^4dx\right)^{1/2} = \|\psi\|_{L^4(\Omega)}^2 \left\|v_i\frac{\partial\psi}{\partial x_i}\right\|_{L^4(\Omega)}^2.
\end{align*}
}for all $1\le i \le d.$ The Sobolev embedding \eqref{lemPR:SobEmb} for $\Omega\subset \mathbb{R}^d,\, d\in \{1,2,3\}$  implies
\begin{align} \label{eqboundfor nh term}
\|g(\psi)\| \le \|\psi\|_{L^4(\Omega)} |\textbf{v}| \|\nabla \psi\|_{L^4(\Omega)} \le s_0^2 |\textbf{v}| \|\psi\|_{H^1(\Omega)} \|\psi\|_{H^2(\Omega)}.
\end{align}
Thus, for any given $\psi \in H^2(\Omega)\cap\Hio,$ $g(\psi)\in \Lt,$ and hence using Step 1, there exists a unique solution $y_s\in H^2(\Omega)\cap\Hio$ satisfying 
\begin{align} \label{eqVB-regestStnSol}
	\|y_s\|_{H^2(\Omega)} \le C \|f\|+s_0^2|\textbf{v}|\|\psi\|_{H^1(\Omega)}\|\psi\|_{H^2(\Omega)},
\end{align}
for some $C>0$.

\noindent \textbf{Step 3.} For any $\rho>0,$ define $D_\rho=\{ y\in H^2(\Omega)\cap \Hio\, |\, \|y\|_{H^2(\Omega)}\le \rho\}.$ In this step, we find a $\rho_0>0$ such that for all $0<\rho\le \rho_0,$ the map $S:D_\rho \longrightarrow D_\rho$ defined by $S(\psi)=y_s^\psi,$ where $y_s^\psi$ is solution of \eqref{eq-linlaplNonH}, is well defined and contraction. For $f\in \Lt$ with $\|f\|_{\Lt}\le \frac{\rho}{2C}$ and $\rho \le \frac{1}{2|\textbf{v}|s_0^2},$ \eqref{eqVB-regestStnSol} implies 
$$\|S(\psi)\|_{H^2(\Omega)} =\|y_s^\psi\|_{H^2(\Omega)} \le C\|f\|_{\Lt}+s_0^2|\textbf{v}|\|\psi\|_{H^1(\Omega)}\|\psi\|_{H^2(\Omega)} \le \rho.$$
Therefore, $S$ is a self map.

\noindent Now, to show contraction, let $\psi^1,\, \psi^2\in D_\rho$ be given and let $y_s^{\psi^1}$ and $y_s^{\psi^2}$ be the corresponding solutions of \eqref{eq-linlaplNonH}. Then, $y_s^{\psi^1}-y_s^{\psi^2}$ satisfies
\begin{align*}
-\eta \Delta (y_s^{\psi^1}-y_s^{\psi^2})+\nu_0(y_s^{\psi^1}-y_s^{\psi^2})=g(\psi^1)-g(\psi^2) \text{ in }\Omega, \, y_s^{\psi^1}-y_s^{\psi^2}=0\text{ on }\partial\Omega.
\end{align*}
From Step 1, we have 
\begin{align*}
\|S(\psi^1)-S(\psi^2)\|_{H^2(\Omega)}=\|y_s^{\psi^1}-y_s^{\psi^2}\|_{H^2(\Omega)} \le C \|g(\psi^1)-g(\psi^2)\|_{\Lt}.
\end{align*}
Note that 
\begin{align*}
	g(\psi^1)-g(\psi^2)=\psi^1\textbf{v}\cdot \nabla \psi^1-\psi^2\textbf{v}\cdot \nabla \psi^2=\displaystyle \sum_{i=1}^d \left( \psi^1 v_i\left( \frac{\partial \psi^1}{\partial x_i}-\frac{\partial \psi^2}{\partial x_i}\right)  + (\psi^1-\psi^2)v_i \frac{\partial \psi^2}{\partial x_i} \right).
\end{align*}
We estimate $g(\psi^1)-g(\psi^2)$ as estimated $g(\psi)$ in \eqref{eqboundfor nh term} and obtain 
\begin{align*}
\| g(\psi^1)-g(\psi^2)\|_{\Lt} & \le s_0^2 |\textbf{v}| \left( \|\psi^1\|_{H^1(\Omega)} \|\psi^1-\psi^2\|_{H^2(\Omega)} + \|\psi^1-\psi^2\|_{H^1(\Omega)} \|\psi^2\|_{H^2(\Omega)} \right)\\
& \le 2 s_0^2 |\textbf{v}|\rho \|\psi^1-\psi^2\|_{H^2(\Omega)}.
\end{align*}
Therefore, choosing $\rho_0=\frac{1}{4s_0^2 |\textbf{v}|},$ we obtain for all $0<\rho\le \rho_0,$
\begin{align*}
\|S(\psi^1)-S(\psi^2)\|_{H^2(\Omega)} \le C \| g(\psi^1)-g(\psi^2)\|_{\Lt}  \le \frac{1}{2} \|\psi^1-\psi^2\|_{H^2(\Omega)}.
\end{align*}
Hence, by using Banach fixed point theorem, we conclude the proposition.
\end{proof}

\subsection{Proof of Theorem \ref{thVB-anasgp}} \label{prof of thVB-anasgp}
	We first determine $\theta_0\in (\frac{\pi}{2},\pi)$ such that (a) holds. Set $\theta_0=\pi-\tan^{-1}\left(\frac{\alpha_1}{\eta/2}\right),$ where $\alpha_1>0$ is as in \eqref{eqVB-BddBil}. Note that $\theta_0\in (\frac{\pi}{2},\pi).$ 
	
	\noindent (a) In the next three steps, we show that $\Sigma^c(-\nuh;\theta_0)\subset \rho(\Ac)$ and the resolvent estimate holds.
	
	\noindent \textbf{Step 1.} ($\mu\in \Sigma^c(-\nuh;\theta_0)$ with $\Re(\mu) \ge -\nuh$). Let $\mu\in \Cb$ be arbitrary with $\Re(\mu)\ge -\nuh.$ First, we show that for any given $g\in \Lt,$ there exists a unique $z\in D(\Ac)$ such that $(\mu I-\Ac)z=g.$  That is, we want to solve
	\begin{align}\label{eqVB-weakform}
		a(z,\phi)+\mu \langle z, \phi\rangle=\langle g,\phi\rangle \text{ for all }\phi\in \Hio,
	\end{align}
	for $z,$ where $a(\cdot,\cdot)$ is as defined in \eqref{eqVB-Sesq}. As $\Re(\mu)\ge -\nuh,$ \eqref{eqVB-coerc} implies $\Re\left(a(\phi,\phi)+\mu\langle \phi,\phi\rangle\right)$ $\ge \Re\left(a(\phi,\phi)\right) -\nuh \|\phi\|^2 \ge \frac{\eta}{2}\|\nabla \phi\|^2.$ Therefore, by Lax-Milgram Theorem (\cite[15, Theorem 1, Section 1, Chapter VII]{RDJLL}) for any given $g\in \Lt,$ there exists a unique solution $z\in \Hio$ of \eqref{eqVB-weakform}.
	Note that, $(\mu I-\Ac)z=g$ implies
	\begin{align} \label{eqVB-resoleq}
		-\eta \Delta z+y_s\textbf{v}\cdot\nabla z+\textbf{v}\cdot\nabla y_s z+(\mu+\nu_0)z=g \text{ in }\Omega,\, z=0 \text{ on }\partial\Omega.
	\end{align}
	Using the regularity results for elliptic equation \cite{Kes}, we have $z\in D(\Ac).$ Substitute $\phi=z$ in \eqref{eqVB-weakform}, use \eqref{eqVB-coerc} and a Cauchy-Schwarz inequality to obtain
	\begin{align*}
		\frac{\eta}{2} \|\nabla z\|^2 \le \Re\left( a(z,z)+\mu \langle z,z\rangle \right)  \le  \|g\|\,\|z\|.
	\end{align*} 
	This along with the Poincar\'{e} inequality leads to
	\begin{align}
		\| z\| \le C\|g\|,
	\end{align}
	for some $C=C(\eta,C_p)>0.$ 
	
	\medskip 
	\noindent \textbf{Step 2. Resolvent estimate for $\Re(\mu)\ge -\nuh,\, \mu\neq -\nuh$.} Let $\mu=-\nuh+\rho e^{i\theta},$ $\rho>0$ and $-\frac{\pi}{2}\le \theta\le \frac{\pi}{2}.$ Substitute $\phi=e^{i\frac{\theta}{2}}z$ in \eqref{eqVB-weakform} to obtain  
	\begin{align}\label{eqVB-weakTheta2}
		a(z,e^{i\frac{\theta}{2}}z)+(-\nuh+\rho e^{i\frac{\theta}{2}})\langle z, e^{i\frac{\theta}{2}}z\rangle =\langle g , e^{i\frac{\theta}{2}}z\rangle.
	\end{align}
	From the definition of $a(\cdot,\cdot)$ in \eqref{eqVB-Sesq} and from \eqref{eqVB-weakTheta2}, we obtain
	{\small
		\begin{align*}
			\cos(\theta/2)\eta\|\nabla z\|^2+(\nu_0-\nuh+\rho)\cos(\theta/2)\|z\|^2 &  \le \left\vert\Re\left(a(z,e^{i\frac{\theta}{2}}z)+(-\nuh+\rho e^{i\theta})\langle z, e^{i\frac{\theta}{2}}z\rangle\right)\right\vert \\
			& \qquad +\left\vert \langle y_s\textbf{v}\cdot\nabla z, e^{i\frac{\theta}{2}}z\rangle +\langle\textbf{v}\cdot\nabla y_s z, e^{i\frac{\theta}{2}}z\rangle\right\vert \\
			& \le \|g\|\,\|z\| + \left\vert \langle y_s\textbf{v}\cdot\nabla z, e^{i\frac{\theta}{2}}z\rangle +\langle\textbf{v}\cdot\nabla y_s z, e^{i\frac{\theta}{2}}z\rangle\right\vert.
		\end{align*}
	}This inequality and a similar estimate as used to obtain \eqref{eqVB-est term ys} lead to 
	{\small
		\begin{align*}
			&	\cos(\theta/2)\left(\eta\|\nabla z\|^2+(\nu_0-\nuh+\rho)\|z\|^2 \right) &   \le \|g\|\,\|z\| + \frac{\eta}{\sqrt{2}} \|\nabla z\|^2 + \frac{|\textbf{v}|^2}{\sqrt{2}\eta} (C_2^2+s_0^4) \|y_s\|_{H^2(\Omega)}^2 \| z\|^2. 
		\end{align*}
	}In view of  \eqref{eqdefVB-nu0}, and the fact that $\cos(\theta/2)\ge \cos(\pi/4)\ge \frac{1}{\sqrt{2}}$ (as $-\frac{\pi}{2}\le \theta\le \frac{\pi}{2}$), we have 
	\begin{align*}
		\rho \cos(\theta/2)\|z\|^2 \le C \|g\|\,\|z\|.
	\end{align*}
	Noting that $\rho=|\mu+\nuh|$ and $\cos(\theta/2)\ge \cos(\pi/4) \ge \frac{1}{\sqrt{2}},$ we have 
	\begin{align}
		\|R(\mu,\Ac)\|_{\Lc(\Lt)} \le \frac{C}{|\mu+\nuh|} \text{ for all } \mu (\neq -\nuh) \text{ with } \Re(\mu) \ge -\nuh,
	\end{align}
	for some positive constant $C$ independent of $\mu.$

	\noindent \textbf{Step 3. Case of any $\mu\in \Sigma^c(-\nuh;\theta_0)$ with $\Re(\mu)<-\nuh$.} The proof is analogous to Step 3 of \cite[proof of Theorem 3.4(a) in Section A.1]{WKRPCE}.

	\medskip 
	\noindent (b) - (c) The proofs of (b) and (c) follow utilizing (a) and proof of \cite[Theorem 3.4]{WKRPCE}. \qed

\section*{Declaration}
\noindent \textbf{Conflict of interest.} This work was done when I was a Ph.D. student at the Department of Mathematics, IIT Bombay, and during that time, I was supported by an institute TA fellowship. 

\medskip

\noindent \textbf{Acknowledgment.} I am thankful to Prof. Neela Nataraj, Prof. Debanjana Mitra, and Prof. Mythily Ramaswamy for their valuable insights, feedback, and encouragement during various stages of this research. Their dedications and expertise played a pivotal role in the successful completion of this study.

\bibliographystyle{amsplain}
\bibliography{References}

\end{document}